\documentclass[a4paper,fleqn]{cas-sc}
\usepackage[numbers]{natbib}
\usepackage{amsmath,amssymb,mathtools,comment}
\usepackage{graphicx,booktabs,overpic}
\usepackage{mathrsfs}
\usepackage{enumitem}

\newtheorem{theorem}{Theorem}[section]
\newtheorem{lemma}[theorem]{Lemma}

\newtheorem{proposition}[theorem]{Proposition}
\newtheorem{remark}[theorem]{Remark}
\newenvironment{proof}{\par\noindent{\it Proof.}\ }{\hfill$\square$\par\medskip}
\numberwithin{equation}{section}

\definecolor{mygray}{rgb}{0.7,0.7,0.7}

\newcommand{\bu}{\boldsymbol{u}}

\newcommand{\bn}{\boldsymbol{n}}
\newcommand{\br}{\boldsymbol{r}}

\newcommand{\bx}{\boldsymbol{x}}
\newcommand{\cero}{\boldsymbol{0}}

\newcommand{\bxi}{\boldsymbol{\xi}}
\newcommand{\bzeta}{\boldsymbol{\zeta}}
\newcommand{\btau}{\boldsymbol{\tau}}

\newcommand{\vdiv}{\mathop{\mathrm{div}}\nolimits}

\newcommand{\bH}{\mathbf{H}}

\newcommand{\bL}{\mathbf{L}}

\newcommand{\bM}{\mathbf{M}}
\newcommand{\bV}{\mathbf{V}}
\newcommand{\bW}{\mathbf{W}}

\newcommand{\rL}{\mathrm{L}}
\newcommand{\rQ}{\mathrm{Q}}

\newcommand{\rH}{\mathrm{H}}

\newcommand{\rW}{\mathrm{W}}

\newcommand{\be}{\boldsymbol{e}}

\newcommand\cA{\mathcal{A}}

\newcommand\cE{\mathcal{E}}

\newcommand\cQ{\mathcal{Q}_h}

\newcommand\cT{\mathcal{T}}

\usepackage{mathrsfs}
\newcommand\rqu{\mathsf{q}}
\newcommand{\mP}{\mathcal{P}}
\newcommand{\bB}{\mathbf{B}}

\newcommand{\bbR}{\mathbb{R}}
\newcommand{\ds}{\displaystyle}
\def\qan{{\quad\hbox{and}\quad}}

\newcommand{\rP}{\mathrm{P}}

\def\bRT{\mathbf{RT}}

\allowdisplaybreaks

\begin{document}
\let\WriteBookmarks\relax
\def\floatpagepagefraction{1}
\def\textpagefraction{.001}
\shorttitle{Perturbed saddle-point problems with non-regular loads}
\shortauthors{A.F. AlSohaim et~al.}

\title [mode = title]{Perturbed saddle-point problems in $\mathrm{L}^p$ with non-regular loads}

\author[1,2]{Abeer F. AlSohaim}
\ead{abeer.alsohaim@monash.edu}

\affiliation[1]{organization={School of Mathematics, Monash University},
                addressline={9 Rainforest Walk},
                postcode={3800},
                city={Clayton},
                state={Victoria},
                country={Australia}}

\affiliation[2]{organization={Department of Mathematics and Statistics, College of Science, Imam Mohammad Ibn Saud Islamic University (IMSIU)},
                city={Riyadh},
                country={Saudi Arabia}}

\author[3]{Thomas F\"uhrer}
\ead{thfuhrer@uc.cl}

\affiliation[3]{organization={Facultad de Matem\'aticas, Pontificia Universidad Cat\'olica de Chile},
                city={Santiago},
                country={Chile}}

\author[1,4]{Ricardo Ruiz-Baier}[orcid=0000-0003-3144-5822]
\cormark[1]
\ead{ricardo.ruizbaier@monash.edu}

\affiliation[4]{organization={Universidad Adventista de Chile},
                addressline={Casilla 7-D},
                city={Chill\'an},
                country={Chile}}

\author[5]{Segundo Villa-Fuentes}
\ead{segundo.villafuentes@ulagos.cl}

\affiliation[5]{organization={Departamento de Ciencias Exactas, Universidad de Los Lagos},
                addressline={Casilla 933},
                city={Osorno},
                country={Chile}}

\cortext[cor1]{Corresponding author}

\begin{abstract}  
In this work, we develop the {continuous and} discrete solvability analysis for  perturbed saddle-point problems in Banach spaces with forcing terms regularised by means of a projector constructed using the adjoint of a weighted Cl\'ement quasi-interpolation. We take as driving example the linearised Poisson--Boltzmann {equation} (an advection-diffusion-reaction problem) in mixed form. We use perturbation arguments on the continuous and discrete levels and then derive a priori estimates that remain valid when the load that appears on the right-hand side of the ``second'' equation is in $\mathrm{H}^{-1}$. Further, we show a supercloseness result and {analyse convergence} of an adequate adaptation of Stenberg postprocessing for mixed advection equations with non-regular data. We provide numerical results that illustrate the convergence of the proposed scheme.  
\end{abstract}


\maketitle

%%%%%%%%%%%%%%%%%%%%%%%%%%%%%%%%%%%%%%%%%%%%%%%%%%%%%%%%%%%%%%%%%%%%%%%%%%%%%%%%%%%%%%%%%%%%%%%%%%%%%%%%%%%%%%%%%%%%%%%%%%%%

\section{Introduction}
We are interested in the mixed form of {advection-diffusion-reaction} equations with singular sources. The main motivation comes from electrochemical flows, in models describing the distribution of electric potential in the direction normal to charged surfaces (see, for instance, \cite{Chen2007}).  In this case, the continuous problem does not possess the usual regularity for elliptic equations. Moreover, while the discrete problem can be defined, the error in the usual energy norms lacks sense, and the typical error estimates need to be modified accordingly. 

Relevant references regarding finite element discretisation include the works for primal formulations in \cite{scott1973finite,casas,koppl2014optimal,d2012finite,drelichman2020weighted} (see also \cite{ern2020quasi} for HHO schemes), as well as  saddle-point problems such as Stokes, Darcy, elasticity, and coupled flow-transport treated in, e.g.,  
\cite{bertoluzza2018local,allendes2024darcy,boon2023analysis}. 
%In \cite{alsohaim25}, a new formulation for the Stokes--Poisson--Boltzmann equations is developed using a fixed-point strategy under smallness assumptions on the data, together with a finite element approximation based on generalised Taylor--Hood spaces. 
More broadly, one way of designing well-defined methods for problems with loads in $\rH^{-1}$ is to construct quasi-interpolation operators on piecewise polynomial spaces (see, e.g., 
\cite{fuhrer2024quasiinterpolatorsapplicationpostprocessingfinite,millar2022projection}). 
%{[need more details on the bibliography]}

Here, and because of the structure of the model equations, the singular forcing appears as the right-hand side of the constraint equation in the saddle-point problem. Problems of this type have been recently addressed in \cite{fuhrer2024mixed} for the case of mixed Poisson, and also in \cite{fuhrer2022minres} for first-order least-squares formulations. To handle the singular force, following \cite{fuhrer2024mixed}  we replace it by a regularised version using a linear and bounded {projection} operator. 
{While~\cite{fuhrer2024mixed} considers projections onto piecewise constants only, here, we extend their construction to piecewise polynomials.}
{Such a map is constructed as the adjoint of an operator that combines a weighted Cl\'ement quasi-interpolator (which has second-order approximation properties and whose weights are piecewise constant) with element bubble functions.}  Moreover, to analyse the resulting mixed formulation in terms of the pseudo-potential flux and the potential, we rely on the theory developed in \cite{correa22}, together with the Banach--Ne\v{c}as--Babu\v{s}ka theorem (cf.\ \cite{ern04}). This framework allows us, by means of a global inf--sup condition combined with a smallness assumption on the advecting velocity field, to prove the well-posedness of the problem at the continuous level. For the finite element approximation, we  use Raviart--Thomas spaces and, exploiting their inf--sup stability, we can appeal to the discrete perturbed saddle-point theory from   \cite{correa22}. 
%The scheme defined with the regularisation is shown optimally convergent, whereas its non-regularised counterpart is not. 
{Solutions to the regularised scheme are shown to satisfy a perturbed {quasi-optimality} result in Banach spaces where the perturbation includes the approximation error of the advection field and the best-approximation of a lower-order term.}

We also {consider accuracy enhancement by a} Stenberg-type postprocessing \cite{stenberg1991postprocessing, chen1995lp}, whose  analysis in the present context requires a reformulation of the typical Aubin--Nitsche duality arguments including additional consistency terms which we control with data smallness assumptions and properties of the proposed smoother. The postprocessing {technique} requires higher regularity to achieve optimal convergence \cite{stenberg1991postprocessing}.

To the best of the authors' knowledge, the present paper is the first work addressing the analysis of low-regularity forcing terms for advection-diffusion-reaction in the non-Hilbertian setting and in the presence of a lower-diagonal block perturbation. 

%%%%%%%%%%%%%%%%%%%%%%%%%%%%%%%%%%%%%%%

The contents of the paper have been organised as follows. Section~\ref{sec:cont} presents the strong form of the modified Poisson--Boltzmann equations including a load regularisation that turns the continuous problem depending on the mesh. We derive a weak formulation in Banach spaces, and establish the well-posedness using Banach--Ne\v{c}as--Babu\v{s}ka theorem, data smallness assumptions, and the theory of perturbed saddle-point problems. In Section~\ref{sec:fe} we construct a mixed finite element scheme using conforming Raviart--Thomas elements and show discrete inf{--}sup conditions by means of a quasi-interpolator. We prove the unique solvability of the discrete problem following a similar structure to the continuous case. Next, Section~\ref{sec:strang} briefly addresses C\'ea estimates, and we derive precise a priori error estimates with respect to the original solution of the linearised Poisson--Boltzmann equation. We also present a postprocess that allows us to achieve higher order approximation of the double layer potential (the primal unknown), and show its properties in Section \ref{sec:stenbergL2}. Section~\ref{sec:results} has a collection of numerical results that highlight the convergence properties and suitability of the proposed formulation. Finally, in Section~\ref{sec:concl} we close with a brief discussion and comment on ongoing extensions of this work. 

%%%%%%%%%%%%%%%%%%%%%%%%%%%%%%%%%%%%

\section{Continuous setting}\label{sec:cont}

Let us consider a Lipschitz and simply connected bounded domain {$\Omega\subset\mathbb{R}^d$, $d\in\{2,3\}$,} with boundary $\partial\Omega = \overline{\Gamma_D} \cup \overline{\Gamma_N}$ disjointly split into two parts where different types of boundary conditions will be considered, and denote by $\bn$ the outward unit normal vector on the boundary.
{We assume that the Dirichlet part $\Gamma_D$ has positive $(d-1)$-dimensional measure, whereas $\Gamma_N$ is allowed to be empty (the case of pure Dirichlet conditions, which is used in some of the numerical examples of Section~\ref{sec:results}). The role of this assumption is discussed in Remark~\ref{rem:GammaD} below.}

We recall the standard notation for Lebesgue spaces $\rL^t(\Omega)$, $t \in {[1,+\infty]}$,
with norm $\|\bullet\|_{0,t;\Omega}$, and for Sobolev spaces {$\rW^{s,t}(\Omega)$, $s\geq 0$, endowed with the norm $\|\bullet\|_{s,t;\Omega}$ and seminorm $|\bullet|_{s,t;\Omega}$. For $t=2$ we write $\rH^s(\Omega)=\rW^{s,2}(\Omega)$. Vector-valued spaces and their elements are denoted in bold, and we use the same notation for their norms.}  In
particular, ${\rH}^{1/2}(\partial\Omega)$ stands for the space of traces
of functions of $\rH^1(\Omega)$ and ${\rH}^{-1/2}(\partial\Omega)$
denotes its dual (similarly for subsets of the boundary). We recall from \cite[Section 2.4.2]{gatica14} the definition of the space $\rH^{1/2}_{00}(\Gamma_D) = \{\eta \in \rH^{1/2}(\Gamma_D): E_{00}(\eta) \in \rH^{1/2}(\partial\Omega)\}$ supplied with the norm  \[ \|\bullet\|_{1/2,00,\Gamma_D} := \| E_{00}(\bullet)\|_{1/2,\partial\Omega},\]
where 
$E_{00}:\rH^{1/2}(\Gamma_D)\to \rH^{1/2}(\partial\Omega)$ denotes the extension-by-zero operator 
\[E_{00}(\eta)= \begin{cases}\eta & \text{on $\Gamma_D$},\\ 0 & \text{on $\partial\Omega\setminus\Gamma_D$},\end{cases} \qquad \forall \eta \in \rH^{1/2}(\Gamma_D).\]
Furthermore, the restriction of $\psi\in \rH^{-1/2}(\partial\Omega)$ to $\Gamma_D$, defined as 
\[ \langle \psi|_{\Gamma_D},\eta \rangle_{\Gamma_D} = \langle\psi,E_{00}(\eta)\rangle_{\partial\Omega} \quad \forall \eta \in \rH^{1/2}_{00}(\Gamma_D),\]
belongs to the dual space, denoted as $[\rH^{1/2}_{00}(\Gamma_D)]'=:\rH^{-1/2}_{00}(\Gamma_D)$. 
Its norm is 
\begin{equation}\label{eq:norm-12}
\|\psi|_{\Gamma_D}\|_{-1/2,00,\Gamma_D} 
= \sup_{0\neq \eta\in \rH_{00}^{1/2}(\Gamma_D)}\frac{\langle \psi,E_{00}(\eta)\rangle_{\partial\Omega} }{\|E_{00}(\eta)\|_{1/2,\partial\Omega}}.\end{equation}
Finally, for any pair of normed spaces $(X, \|\bullet\|_{X})$ and  
$(Y, \|\bullet \|_{Y})$, we provide the product space $X \times Y$ with the  
natural norm $\| (x,y) \|_{X \times Y}:=\|x\|_{X}+\|y\|_{Y}$  
for all $(x,y)\in X \times Y$. 

For a Dirichlet datum $\psi_D \in \rH^{1/2}_{00}(\Gamma_D) $,  the motivating problem under consideration is the nonlinear Poisson--Boltzmann equation \cite{Chen2007}, 
\begin{subequations} \label{eq:PBstrong:primal}
    \begin{align}
        \kappa\psi  - \vdiv(\varepsilon\nabla \psi-\bu\psi) &= g && \text{in } \Omega, \label{eq:primal} \\
        \psi &= \psi_{D} && \text{on } \Gamma_D, \label{eq:bc1:primal} \\
        ( \varepsilon\nabla \psi-\bu\psi) \cdot \bn &= 0 && \text{on } \Gamma_N, \label{eq:bc2:primal}
    \end{align}
\end{subequations}
including a fixed arbitrary advecting velocity field  $\bu \in \bL^4(\Omega)$. This particular regularity comes from the specific functional structure of the Navier--Stokes equations in  mixed form from \cite{correa24}, aimed at coupling with the Poisson--Boltzmann equation as in \cite{alsohaim25}. 
The permittivity coefficient $\varepsilon$ is heterogeneous and bounded away from zero, {and} the  reaction function $\kappa$ depends nonlinearly on the potential and it is monotone and unbounded. {The source $g$ models fixed charges, which can be point charges (Dirac measures) or charges distributed along curves or surfaces. Since $\rH^1(\Omega)\not\hookrightarrow C(\overline\Omega)$ for $d\geq 2$, a point Dirac measure does not belong to $\rH^{-1}(\Omega)$ (see, e.g., \cite{casas,koppl2014optimal,d2012finite,drelichman2020weighted}). In contrast, by the trace theorem, a source $\langle\delta_\gamma,v\rangle:=\int_\gamma v$ concentrated on a manifold $\gamma$ of codimension one (a curve if $d=2$, a surface if $d=3$) belongs to $(\rH^1(\Omega))'$; in fact, $\delta_\gamma\in\rH^{-1/2-\epsilon}(\Omega)$ for all $\epsilon>0$ \cite{li2021regularity,cao2025regularity}. The present analysis covers the latter type of sources, but not point sources.}
However, and as a preliminary step, we consider that $\kappa\in \rL^2(\Omega)$ is heterogeneous but independent of $\psi$, and that {$g \in \widetilde\rH^{-1}(\Omega):=(\rH^1(\Omega))'$}. 

{In the work at hand we are therefore interested in a mixed formulation of the linearised Poisson--Boltzmann equation} in terms of the \emph{double layer potential} $\psi$ and the \emph{pseudo potential flux} $\bzeta$. This gives  
\begin{subequations} \label{eq:PBstrong}
    \begin{align}
        \bzeta &= \varepsilon \nabla \psi - \bu \psi && \text{in } \Omega, \label{eq:PB1} \\
        \kappa\psi  - \vdiv\bzeta &= g && \text{in } \Omega, \label{eq:PB2} \\
        \psi &= \psi_{D} && \text{on } \Gamma_D, \label{eq:bc1} \\
        \bzeta \cdot \bn &= 0 && \text{on } \Gamma_N. \label{eq:bc2}
    \end{align}
\end{subequations}

\subsection{{Analysis of the primal formulation}}
%{[maybe we should write \eqref{eq:PBstrong:primal} before writing the mixed form \eqref{eq:PBstrong}, and this subsection can be called "Analysis of the primal formulation"?]} 
%Let us consider the system~\eqref{eq:PBstrong} and define the notion of a solution to this system. 
%Equation~\eqref{eq:PB1} is understood in $\rL^2(\Omega)$, whereas the second equation~\eqref{eq:PB2} is understood in the distributional sense.
%Eliminating $\bzeta$ from~\eqref{eq:PBstrong} yields the primal system
%Furthermore,~\eqref{eq:bc2:primal} is understood as the distributional conormal derivative, see~\cite[Lemma~4.3]{McLean}.

{Let $\rH^1_D(\Omega):=\{v\in\rH^1(\Omega):\ v|_{\Gamma_D}=0\}$.} {The weak formulation of~\eqref{eq:primal}} reads as follows: Find $\psi\in {\rH^1(\Omega)}$ such that
\begin{align}\label{eq:primal:weak}
\begin{split}
  \int_\Omega (\kappa \psi \varphi + \varepsilon\nabla\psi\cdot\nabla\varphi - \bu\psi\cdot\nabla \varphi) &= {\langle g,\varphi\rangle} 
  \quad\forall \varphi \in \rH^1_D(\Omega), \quad
  \psi|_{\Gamma_D} = \psi_D.
\end{split}
\end{align}

We assume that
{$\varepsilon\in \rL^\infty(\Omega)$, $0<\underline\varepsilon\leq \varepsilon\leq \overline\varepsilon :=\|\varepsilon\|_{0,\infty;\Omega}$, $\kappa\in \rL^2(\Omega)$, $\kappa\geq 0$ (a.e. in $\Omega$) and that}
the norm of $\bu$ is sufficiently small to guarantee existence and uniqueness of solutions, i.e., 
\begin{align}\label{eq:assumption:u}
  \|\bu\|_{0,4;\Omega} < \frac{\underline\varepsilon}{C_\mathrm{Sob}{(1+c_P^2)^{1/2}}},
\end{align}
where {$c_P>0$ denotes Poincar\'e's constant and} $C_\mathrm{Sob}>0$ denotes the continuity constant of the Sobolev embedding $\rH^1_D(\Omega) \to\rL^4(\Omega)$, i.e., the smallest {constants} such that 
\begin{align*}
  {\|\varphi\|_{0,2;\Omega}\leq c_P \|\nabla \varphi\|_{0,2;\Omega},} 
  \quad\|\varphi\|_{0,4;\Omega} \leq C_\mathrm{Sob} {\|\varphi\|_{1,2;\Omega}}
\end{align*}
holds for all $\varphi\in \rH^1_D(\Omega)$.
{We stress that assumption~\eqref{eq:assumption:u} implies that there exists {$\vartheta\in(0,1)$} such that
  \begin{align*}
    \|\bu\|_{0,4;\Omega} \leq {\vartheta} \frac{\underline\varepsilon}{C_\mathrm{Sob}{(1+c_P^2)^{1/2}}}.
  \end{align*}
  This observation is used in the proof of several results below.} 
  In the following, we use notation $\rH^{-1}_D(\Omega) = (\rH^1_D(\Omega))'$.
And note that any $g\in \widetilde\rH^{-1}(\Omega)$ satisfies
\begin{align*}
  \|g\|_{\rH^{-1}_D(\Omega)} := \sup_{0\neq v\in \rH^1_D(\Omega)} \frac{\langle g,v\rangle}{\|v\|_{1,2;\Omega}}
  \leq \sup_{0\neq v\in \rH^1(\Omega)} \frac{\langle g,v\rangle}{\|v\|_{1,2;\Omega}} =: \|g\|_{\widetilde\rH^{-1}(\Omega)}
\end{align*}
with $\langle g,v\rangle$ denoting the duality pairing on $\widetilde\rH^{-1}(\Omega)\times \rH^1(\Omega)$.

{Using the Sobolev embedding we note that, for $\varphi\in \rL^{4/3}(\Omega)$, 
\begin{align}\label{eq:dualAndSobolevEstimate}
  \|\varphi\|_{\rH^{-1}_D(\Omega)} = \sup_{0\neq v\in \rH^1_D(\Omega)} \frac{\langle \varphi,v\rangle}{\|v\|_{1,2;\Omega}}
  \leq \sup_{0\neq v\in \rH^1_D(\Omega)} \frac{\|\varphi\|_{0,4/3;\Omega}\|v\|_{0,4;\Omega}}{\|v\|_{1,2;\Omega}}
  \leq C_\mathrm{Sob} \|\varphi\|_{0,4/3;\Omega}.
\end{align}}

\begin{proposition}\label{prop:primal}
  There exists a constant $C>0$ such that for any $g\in \widetilde\rH^{-1}(\Omega)$ and $\psi_D\in \rH^{1/2}_{00}(\Gamma_D)$, Problem~\eqref{eq:PBstrong:primal} admits a unique {weak} solution $\psi\in {\rH^1(\Omega)}$ {(in the sense of~\eqref{eq:primal:weak})} and
  \begin{align*}
    \|\psi\|_{1,2;\Omega} \leq C(\|\psi_D\|_{1/2,00;\Gamma_D} + \|g\|_{\rH^{-1}_D(\Omega)}).
  \end{align*}
\end{proposition}
\begin{proof}
  By~\cite[Theorem~4.10]{McLean} it remains to show that the bilinear form corresponding to the variational formulation~\eqref{eq:primal:weak} is coercive and bounded. 
  We further note that~\cite{McLean} states the bound with $\|g\|_{\widetilde\rH^{-1}(\Omega)}$ on the right-hand side. This can be improved to the present bound by following the very same arguments in the proof which are therefore omitted. 

  Let $\psi,\varphi\in \rH^1_D(\Omega)$ be given, and note that 
  \begin{align*}
    \int_\Omega |\kappa \psi \varphi + \varepsilon\nabla\psi\cdot\nabla\varphi - \bu\psi\cdot\nabla \varphi| \leq C(1+\|\bu\|_{0,4;\Omega}) \|\psi\|_{1,2;\Omega}\|\varphi\|_{1,2;\Omega},
  \end{align*}
  where $C>0$ is some generic constant. 
  Further, by our assumption on the norm of $\bu$, {there exists {$\vartheta\in (0,1)$} such that}
  \begin{align*}
    -\int_\Omega |\bu\psi\cdot\nabla \psi| \geq -\|\bu\|_{0,4;\Omega}\|\psi\|_{0,4;\Omega}\|\nabla\psi\|_{0,2;\Omega}
    \geq {-\vartheta\underline\varepsilon \|\nabla\psi\|_{0,2;\Omega}^2},
  \end{align*}
  {where we have used that $\|\psi\|_{0,4;\Omega}\leq C_\mathrm{Sob}\|\psi\|_{1,2;\Omega}\leq C_\mathrm{Sob}(1+c_P^2)^{1/2}\|\nabla\psi\|_{0,2;\Omega}$ for all $\psi\in\rH^1_D(\Omega)$.}
  Then, {using $\kappa\geq 0$} {and, once more, Poincar\'e's inequality}, we arrive at 
  \begin{align*}
    \int_\Omega (\kappa \psi^2 + \varepsilon|\nabla\psi|^2 - \bu\psi\cdot\nabla \psi) &\geq {(1-\vartheta)}\underline\varepsilon\|\nabla \psi\|_{0,2;\Omega}^2 + \int_\Omega \kappa \psi^2 \geq {\frac{(1-\vartheta)\underline\varepsilon}{1+c_P^2}\|\psi\|_{1,2;\Omega}^2}
    \qquad{\forall\psi\in\rH^1_D(\Omega)},
  \end{align*}
  thus concluding the proof.
\end{proof}

%{TF: It should be possible to extend the last result to functionals $g\in \rH_D^{-1}(\Omega) = (\rH^{1}_D(\Omega))'$, but we would need to do the proof ourselves.}
{We remark that the assumption $\kappa\geq 0$ can be replaced by assuming that $\|\kappa\|_{0,2;\Omega}$ is sufficiently small.}

\subsection{{Notation and approximation spaces for mixed formulations}}
Before we define and analyse the regularised problem we need to introduce some spaces and notation.
In view of the functional structure of the weak formulation associated with \eqref{eq:PBstrong-star} (dictated by the integrability of each term, in particular the one coming from the last term on the right-hand side of the constitutive equation \eqref{eq:PB1-star}), we consider the following spaces for the pseudo potential flux and double layer potential, respectively 
\[
\bH:= \bH_N(\vdiv_{4/3};\Omega) = \{ \bxi \in\bL^2(\Omega):\vdiv\bxi \in \rL^{4/3}(\Omega) \ \text{and}\ \bxi\cdot \bn=0 \ \text{on}\ \Gamma_N \}, \quad \rQ := \rL^4(\Omega). 
\]
{Here, $\bxi\cdot\bn=0$ on $\Gamma_N$ means that $\int_\Omega(\bxi\cdot\nabla v+v\vdiv\bxi)=0$ for all $v\in\rH^1_D(\Omega)$. These spaces are endowed with the norms
\[
\|\bxi\|_{\bH}:=\|\bxi\|_{\vdiv_{4/3};\Omega}:=\|\bxi\|_{0,2;\Omega}+\|\vdiv\bxi\|_{0,4/3;\Omega},\qquad \|\varphi\|_{\rQ}:=\|\varphi\|_{0,4;\Omega}.
\]}

Let $\cT_h$ denote a family of non-degenerate triangular / tetrahedral meshes on $\Omega$ {(assumed to be compatible with the boundary partition, that is, each boundary facet is contained either in $\overline{\Gamma_D}$ or in $\overline{\Gamma_N}$)} and denote by $\cE_h$ the set of all facets (edges in 2D) in the mesh. 
By $h_K$ we denote the diameter of the element $K$ and by $h_F$ we denote the length/area of the facet $F$. As usual, by $h$ we denote the maximum of the diameters of elements in $\cT_h$. For all meshes we assume that they are sufficiently regular (there exists a uniform positive constant $\eta_1$ such that each element $K$ is star-shaped with respect to a ball of radius greater than $\eta_1 h_K$. It is also assumed that there exists $\eta_2>0$ such that for each element and every facet $F\in \partial K$, we have that $h_F\geq \eta_2 h_K$, see, e.g., \cite{ern04}). 

By $\rP_k(K)$  we will denote the {space of scalar polynomials of total degree at most $k\geq0$} defined locally on $K\in \cT_h${, and by $\rP_k(\cT_h)$ the space of (possibly discontinuous) functions whose restriction to each $K\in\cT_h$ belongs to $\rP_k(K)$}. In addition, we denote by $\bRT_k(K):=[\rP_k(K)]^d {+} \rP_k(K)\,\bx$ the local Raviart--Thomas space and by $\bRT_k(\cT_h)$ its global counterpart. 

Spaces $\bH_h$ and $\rQ_h$ are suitable subspaces of the functional spaces $\bH$ and $\rQ$, respectively.
{For the analysis of the discretised mixed problem below}, we require kernel characterisation and discrete inf{--}sup conditions and so we can simply take, for example, Raviart--Thomas elements  
\[\bH_h = \bRT_k(\cT_h)\cap \bH = \{ \bxi_h \in \bH_N(\vdiv_{4/3};\Omega): \bxi_h|_K \in \bRT_k(K),\ \forall\, K\in \cT_h\}, \quad \rQ_h = \rP_k(\cT_h). \]
{Note that $\vdiv(\bH_h)\subset\rQ_h$ which is a property that will be used in the analysis.
In particular, we even have that $\vdiv(\bH_h)=\rQ_h$ {(since $|\Gamma_D|>0$)}.
We denote by $\mP^k_h\colon {\rL^1(\Omega)}\to \rQ_h$ the $\rL^2(\Omega)${-}orthogonal projection{, that is, for $w\in\rL^1(\Omega)$, $\mP_h^k(w)$ is the unique element of $\rP_k(\cT_h)$ such that $\int_\Omega\mP_h^k(w)\,q_h=\int_\Omega w\,q_h$ for all $q_h\in\rP_k(\cT_h)$. Note that $\mP_h^k\mP_h^0=\mP_h^0$, since $\rP_0(\cT_h)\subseteq\rP_k(\cT_h)$.}}

\subsection{The perturbed regularised primal problem}
In order to deal with the singular force $g\in \widetilde\rH^{-1}(\Omega)$ we replace it by a regularised one, $\cQ g$, where
$\cQ\colon \rH^{-1}_D(\Omega)\to \rQ_h$ denotes a bounded linear operator with the properties
\begin{equation}\label{eq:prop0-Q}
  {\cQ q_h = q_h \quad \forall q_h\in\rQ_h}, \quad \|\cQ\|_{\rH^{-1}_D(\Omega)\to \rH^{-1}_D(\Omega)} + \|\cQ\|_{\rL^2(\Omega)\to \rL^2(\Omega)} \leq C_{\cQ}<\infty,
\end{equation}
{with $C_{\cQ}$ independent of $h$, where $\rL^2(\Omega)\supset\rQ_h$ is identified with a subspace of $\rH^{-1}_D(\Omega)$ via the $\rL^2$ product, and $\cQ g$ denotes $\cQ$ applied to the restriction of $g$ to $\rH^1_D(\Omega)$. A construction of $\cQ$ for any $k\geq0$ is given in Section~\ref{sec:constructionQ}.}
{Additionally, $\bu$ is replaced} by an approximation {$\bu_h\in{\bL^4(\Omega)}$}.
This yields an auxiliary problem which in its (primal) weak form reads as follows: Find $\psi^\star\in \rH^1(\Omega)$ such that
\begin{align}\label{eq:aux:weak}
  \int_\Omega (\kappa \psi^\star \varphi + \varepsilon\nabla\psi^\star\cdot\nabla\varphi - \bu_h\psi^\star\cdot\nabla \varphi) &= \int_\Omega \cQ g\varphi 
  \quad\forall \varphi \in \rH^1_D(\Omega), \qquad
  \psi^\star|_{\Gamma_D} = \psi_D.
\end{align}
Existence and uniqueness of solutions follows as in Proposition~\ref{prop:primal} assuming that
\begin{align}\label{eq:assumption:uh}
  \|\bu_h\|_{0,4;\Omega} < \frac{\underline\varepsilon}{C_\mathrm{Sob}{(1+c_P^2)^{1/2}}}.
\end{align}
Note that if $\bu_h\to \bu$ {in the sense of ${\bL^4(\Omega)}$ when $h\to 0$}, it follows that there exists a sufficiently small $h_0$ such that~\eqref{eq:assumption:uh} is satisfied for $h\leq h_0$ provided that assumption~\eqref{eq:assumption:u} holds. 

\begin{lemma}\label{lem:quasiopt:primal}
  {Assume that~\eqref{eq:assumption:u} and~\eqref{eq:assumption:uh} hold, and} let $\psi$ and $\psi^\star$ denote the unique solutions of Problems~\eqref{eq:primal:weak} and~\eqref{eq:aux:weak}, respectively.
  Then, there exist constants $C_1>0$, $C_2>0${, independent of $h$,} such that
  \begin{align*}
    \|\psi-\psi^\star\|_{1,2;\Omega} &\leq C_1(\|g\|_{\rH^{-1}_D(\Omega)}+\|\psi_D\|_{1/2,00;\Gamma_D})\|\bu-\bu_h\|_{0,4;\Omega}
  \\
  &\qquad + C_2(\min_{\btau_h\in\bH_h}\|\bzeta-\btau_h\|_{0,2;\Omega} + {\min_{\chi_h\in\rQ_h}\|\kappa\psi-\chi_h\|_{\rH^{-1}_D(\Omega)}}),
  \end{align*}
  where $\bzeta = \varepsilon\nabla \psi - \bu \psi $, and, with $\bzeta^\star = \varepsilon\nabla\psi^\star-\bu_h\psi^\star$,
  \begin{align*}
    \|\psi-\psi^\star\|_{0,4;\Omega} + \|\bzeta-\bzeta^\star\|_{0,2;\Omega} &\leq C_1(\|g\|_{\rH^{-1}_D(\Omega)}+\|\psi_D\|_{1/2,00;\Gamma_D})\|\bu-\bu_h\|_{0,4;\Omega}
  \\
  &\qquad + C_2(\min_{\btau_h\in\bH_h}\|\bzeta-\btau_h\|_{0,2;\Omega} + {\min_{\chi_h\in\rQ_h}\|\kappa\psi-\chi_h\|_{\rH^{-1}_D(\Omega)}}).
  \end{align*}
\end{lemma}
\begin{proof}

  Here $C$ denotes a generic constant independent of $h$. As in the proof of Proposition~\ref{prop:primal}, \eqref{eq:assumption:uh} implies that the bilinear form of~\eqref{eq:aux:weak} is coercive on $\rH^1_D(\Omega)$, with a constant $c_0>0$. 

%  \smallskip\noindent\textit{Step 1 (perturbation of the velocity).} 
First we proceed to do a perturbation of the velocity. 
  Let $\widetilde\psi^\star\in \rH^1(\Omega)$ denote the unique solution to
\begin{align*}
  \int_\Omega (\kappa \widetilde\psi^\star \varphi + \varepsilon\nabla\widetilde\psi^\star\cdot\nabla\varphi - \bu_h\widetilde\psi^\star\cdot\nabla \varphi) &= \langle g,\varphi\rangle
  \quad\forall \varphi \in \rH^1_D(\Omega), \qquad
  \widetilde\psi^\star|_{\Gamma_D} = \psi_D,
\end{align*}
and set $w := \psi-\widetilde\psi^\star \in \rH^1_D(\Omega)$. Subtracting the problem above from~\eqref{eq:primal:weak} and using that $\bu\psi-\bu_h\widetilde\psi^\star = (\bu-\bu_h)\psi + \bu_h w$, we obtain 
\begin{align*}
  \int_\Omega(\kappa w\varphi + \varepsilon\nabla w\cdot\nabla \varphi - \bu_h w\cdot\nabla\varphi) = \int_\Omega(\bu-\bu_h)\psi\cdot\nabla\varphi \qquad\forall\varphi\in\rH^1_D(\Omega).
\end{align*}
Testing with $\varphi=w$ and using H\"older's inequality and the Sobolev embedding, we get
\begin{align*}
  c_0\|w\|_{1,2;\Omega}^2 \leq \|\bu-\bu_h\|_{0,4;\Omega}\|\psi\|_{0,4;\Omega}\|\nabla w\|_{0,2;\Omega} \leq C_\mathrm{Sob}\|\bu-\bu_h\|_{0,4;\Omega}\|\psi\|_{1,2;\Omega}\|w\|_{1,2;\Omega}.
\end{align*}
Hence, by Proposition~\ref{prop:primal},
\begin{align}\label{eq:aux-w}
  \|\psi-\widetilde\psi^\star\|_{1,2;\Omega} \leq C(\|\psi_D\|_{1/2,00;\Gamma_D} + \|g\|_{\rH^{-1}_D(\Omega)})\|\bu-\bu_h\|_{0,4;\Omega}.
\end{align}
Moreover, with $\widetilde{\bzeta}^\star := \varepsilon\nabla\widetilde\psi^\star-\bu_h\widetilde\psi^\star$ we have $\bzeta-\widetilde{\bzeta}^\star = \varepsilon\nabla w - (\bu-\bu_h)\psi - \bu_h w$, and therefore
\begin{align}\label{eq:aux-zeta}
  \|\bzeta-\widetilde{\bzeta}^\star\|_{0,2;\Omega} &\leq \overline\varepsilon\|\nabla w\|_{0,2;\Omega} + \|\bu-\bu_h\|_{0,4;\Omega}\|\psi\|_{0,4;\Omega} + \|\bu_h\|_{0,4;\Omega}\|w\|_{0,4;\Omega}\nonumber
  \\
  &\leq C(\|\psi_D\|_{1/2,00;\Gamma_D} + \|g\|_{\rH^{-1}_D(\Omega)})\|\bu-\bu_h\|_{0,4;\Omega}.
\end{align}

%\smallskip\noindent\textit{Step 2 (perturbation of the load).} 
Next, we perturb the load. The difference $w^\star:=\widetilde\psi^\star-\psi^\star\in\rH^1_D(\Omega)$ satisfies
\begin{align*}
  \int_\Omega(\kappa w^\star\varphi + \varepsilon\nabla w^\star\cdot\nabla \varphi - \bu_h w^\star\cdot\nabla\varphi) = \langle (1-\cQ)g,\varphi\rangle \qquad\forall\varphi\in\rH^1_D(\Omega),
\end{align*}
and thus, by coercivity, 
\begin{align}\label{eq:aux-wstar}
  \|\widetilde\psi^\star-\psi^\star\|_{1,2;\Omega}\leq c_0^{-1}\|(1-\cQ)g\|_{\rH^{-1}_D(\Omega)}.
\end{align}
For $\btau\in\bL^2(\Omega)$, let $\vdiv\btau\in\rH^{-1}_D(\Omega)$ be defined by $\langle\vdiv\btau,v\rangle:=-\int_\Omega\btau\cdot\nabla v$, so that $\|\vdiv\btau\|_{\rH^{-1}_D(\Omega)}\leq\|\btau\|_{0,2;\Omega}$ and $g=\kappa\widetilde\psi^\star-\vdiv\widetilde{\bzeta}^\star$ in $\rH^{-1}_D(\Omega)$. For $\btau_h\in\bH_h$, this functional coincides with $\vdiv\btau_h\in\rQ_h$, and thus~\eqref{eq:prop0-Q} gives $(1-\cQ)\vdiv\btau_h=0$. Since also $(1-\cQ)\chi_h=0$ for $\chi_h\in\rQ_h$, we have
\begin{align*}
  (1-\cQ)g = (1-\cQ)(\kappa\widetilde\psi^\star-\chi_h) - (1-\cQ)\vdiv(\widetilde{\bzeta}^\star-\btau_h),
\end{align*}
and the boundedness of $\cQ$ in $\rH^{-1}_D(\Omega)$ implies
\begin{align*}
  \|(1-\cQ)g\|_{\rH^{-1}_D(\Omega)} \leq (1+C_{\cQ})\bigl(\|\kappa\widetilde\psi^\star-\chi_h\|_{\rH^{-1}_D(\Omega)} + \|\widetilde{\bzeta}^\star-\btau_h\|_{0,2;\Omega}\bigr)
  \qquad\forall\btau_h\in\bH_h,\ \forall\chi_h\in\rQ_h.
\end{align*}
Moreover, by~\eqref{eq:dualAndSobolevEstimate} and the Sobolev embedding,
\begin{align*}
  \|\kappa\widetilde\psi^\star-\chi_h\|_{\rH^{-1}_D(\Omega)} &\leq \|\kappa\psi-\chi_h\|_{\rH^{-1}_D(\Omega)} + C_\mathrm{Sob}\|\kappa(\psi-\widetilde\psi^\star)\|_{0,4/3;\Omega}
  \\ &\leq \|\kappa\psi-\chi_h\|_{\rH^{-1}_D(\Omega)} + C_\mathrm{Sob}^2\|\kappa\|_{0,2;\Omega}\|\psi-\widetilde\psi^\star\|_{1,2;\Omega}, \\
  \|\widetilde{\bzeta}^\star-\btau_h\|_{0,2;\Omega} &\leq \|\bzeta-\btau_h\|_{0,2;\Omega} + \|\bzeta-\widetilde{\bzeta}^\star\|_{0,2;\Omega}.
\end{align*}
Combining the last three estimates with~\eqref{eq:aux-w}--\eqref{eq:aux-zeta}, we arrive at
\begin{align}\label{eq:aux-Qg}
  \|(1-\cQ)g\|_{\rH^{-1}_D(\Omega)} &\leq C\Bigl(\min_{\btau_h\in\bH_h}\|\bzeta-\btau_h\|_{0,2;\Omega} + \min_{\chi_h\in\rQ_h}\|\kappa\psi-\chi_h\|_{\rH^{-1}_D(\Omega)}\Bigr) \nonumber\\
  &\qquad + C(\|\psi_D\|_{1/2,00;\Gamma_D} + \|g\|_{\rH^{-1}_D(\Omega)})\|\bu-\bu_h\|_{0,4;\Omega}.
\end{align}

%\smallskip\noindent\textit{Step 3 (conclusion).} 
Finally, the first estimate of the lemma follows from $\psi-\psi^\star = (\psi-\widetilde\psi^\star)+(\widetilde\psi^\star-\psi^\star)$, the triangle inequality, \eqref{eq:aux-w}, \eqref{eq:aux-wstar} and~\eqref{eq:aux-Qg}. For the second estimate, we note that $\|\psi-\psi^\star\|_{0,4;\Omega}\leq C_\mathrm{Sob}\|\psi-\psi^\star\|_{1,2;\Omega}$ and, as in~\eqref{eq:aux-zeta}, 
\begin{align*}
  \|\bzeta-\bzeta^\star\|_{0,2;\Omega} \leq (\overline\varepsilon + C_\mathrm{Sob}\|\bu_h\|_{0,4;\Omega})\|\psi-\psi^\star\|_{1,2;\Omega} + C_\mathrm{Sob}\|\psi\|_{1,2;\Omega}\|\bu-\bu_h\|_{0,4;\Omega},
\end{align*}
so that the second estimate follows from the first one and Proposition~\ref{prop:primal}.

\end{proof}

If $\kappa\in \rL^4(\Omega)$, then $\kappa\psi \in \rL^2(\Omega)$, and a duality argument together with the approximation properties of $\mP_h^k$ shows that
\[
  \min_{\chi_h\in\rQ_h}\|\kappa\psi-\chi_h\|_{\rH^{-1}_D(\Omega)} \lesssim h \|(1-\mP_h^k)(\kappa\psi)\|_{0,2;\Omega}.
\]
If, in addition, $\kappa$ is piecewise constant, then $(1-\mP_h^k)(\kappa\psi)|_K = \kappa|_K(1-\mP_h^k)\psi|_K$ for all $K\in\cT_h$, and therefore 
$\|(1-\mP_h^k)(\kappa\psi)\|_{0,2;\Omega} \lesssim \min_{v_h\in \rQ_h} \|\psi-v_h\|_{0,4;\Omega}$.

\subsection{Mixed formulation}
Let us now rewrite a  regularised counterpart of the strong mixed form \eqref{eq:PBstrong}: 
\begin{subequations} \label{eq:PBstrong-star}
    \begin{align}
      \bzeta^\star &= \varepsilon \nabla \psi^\star - {\bu_h} \psi^\star && \text{in } \Omega, \label{eq:PB1-star} \\
        \kappa\psi^\star  - \vdiv\bzeta^\star &= \cQ g && \text{in } \Omega, \label{eq:PB2-star} \\
        \psi^\star &= \psi_{D} && \text{on } \Gamma_D, \label{eq:bc1-star} \\
        \bzeta^\star \cdot \bn &= 0 && \text{on } \Gamma_N.\label{eq:bc2-star}
    \end{align}
\end{subequations}
We remark that the flux and potential $(\bzeta^\star,\psi^\star)$ depend upon discretisation---through the dependence on the discrete map $\cQ$ and on the discrete advection velocity $\bu_h$---but they are not necessarily discrete. %{[this is no longer true, right? $\cQ$ seems  not associated with any discretisation yet \dots]}

Using appropriate integration by parts formulas in Banach spaces, we obtain a perturbed saddle-point problem: for given $\kappa\in \rL^2(\Omega)$, $\varepsilon \in \rL^\infty(\Omega)$, {$\cQ g\in \rL^{4/3}(\Omega)$}, and $\psi_D \in \rH^{1/2}_{00}(\Gamma_D)$, find $(\bzeta^\star,\psi^\star) \in \bH\times \rQ$ such that 
\begin{equation} \label{eq:weak}
\begin{aligned}
  a(\bzeta^\star,\bxi) &+ b(\bxi,\psi^\star) + d_{{\bu_h}}(\bxi,\psi^\star) &= F(\bxi) && \forall \bxi \in \bH, \\
    b(\bzeta^\star,\varphi) &- c(\psi^\star,\varphi) &= G(\varphi) && \forall \varphi \in \rQ,
\end{aligned}
\end{equation}
where (again for a fixed $\bu_h$) the bilinear forms $a:\bH\times \bH\to \bbR$, $b:\bH\times \rQ\to \bbR$, $c:\rQ\times \rQ\to \bbR$, $d_{{\bu_h}}:\bH\times \rQ\to \bbR$, and the linear functionals $F:\bH\to \bbR$ and $G:\rQ\to \bbR$ are defined as
\begin{equation}\label{eq:forms-and-functionals}
\begin{array}{cc}
\ds  a(\bzeta^\star,\bxi) := \int_\Omega \frac{1}{\varepsilon} \bzeta^\star\cdot \bxi\, ,\quad b(\bxi,\varphi) := \int_\Omega \varphi\, \vdiv \bxi\, , \quad c(\psi^\star,\varphi) := \int_\Omega \kappa\, \psi^\star\, \varphi \, ,\\[2ex]
\ds d_{{\bu_h}}(\bxi,\varphi) := \int_\Omega \frac{1}{\varepsilon} {{\bu_h}} \cdot \bxi\, \varphi \, , \quad F(\bxi) := \langle \bxi \cdot \bn, \psi_D \rangle \, , \quad G(\varphi) := -\int_\Omega \cQ g\, \varphi \, ,
\end{array}
\end{equation}
% \begin{subequations}
% \begin{align}
%     a(\bzeta,\bxi) &:= \int_\Omega \frac{1}{\varepsilon} \bzeta\cdot \bxi 
%     && \forall (\bzeta,\bxi) \in \bH \times \bH, \\
%     b(\bxi,\varphi) &:= \int_\Omega \varphi\, \vdiv \bxi 
%     && \forall (\bxi,\varphi) \in \bH \times \rQ, \label{b}\\
%     c(\psi,\varphi) &:= \int_\Omega \kappa\, \psi\, \varphi 
%     && \forall (\psi, \varphi) \in \rQ \times \rQ, \\
%     d_{\bu}(\bxi,\varphi) &:= \int_\Omega \frac{1}{\varepsilon} {\bu} \cdot \bxi\, \varphi 
%     && \forall (\bxi,\varphi) \in \bH \times \rQ, \\
%     F(\bxi) &:= \langle \bxi \cdot \bn, \psi_D \rangle 
%     && \forall \bxi \in \bH, \\
%     G(\varphi) &:= -\int_\Omega \cQ g\, \varphi 
%     && \forall \varphi \in \rQ,
% \end{align}
% \end{subequations}
 where $\langle \bullet,\bullet\rangle$ stands for the duality pairing between $\rH^{1/2}_{00}(\Gamma_D) $ and $\rH^{-1/2}_{00}(\Gamma_D)$. 

Equivalently, by defining the bilinear forms $A,A_{ {\bu_h}}:(\bH\times \rQ) \times  (\bH\times \rQ) \rightarrow \bbR$  as 
\begin{subequations}
\begin{align}\label{eq:A}
A((\bzeta^\star,\psi^\star),(\bxi,\varphi))&:=a(\bzeta^\star,\bxi)+ b(\bxi,\psi^\star)+ b(\bzeta^\star,\varphi) - c(\psi^\star,\varphi), \quad 
\end{align}
\text{and}
\begin{align}\label{eq:Au}
A_{ {\bu_h}}((\bzeta^\star,\psi^\star),(\bxi,\varphi))&:= A((\bzeta^\star,\psi^\star),(\bxi,\varphi))  + d_{{\bu_h}}(\bxi,\psi^\star), 
\end{align}
\end{subequations}
for all $(\bzeta^\star,\psi^\star),(\bxi,\varphi) \in \bH \times \rQ$, we can rewrite \eqref{eq:weak}  as 
$\text{find } (\bzeta^\star,\psi^\star) \in \bH \times \rQ \text{ such that}$
\begin{equation} \label{eq:Gweak}
  A_{ {{\bu_h}}}((\bzeta^\star,\psi^\star),(\bxi,\varphi))=F(\bxi)+G(\varphi) \qquad \forall (\bxi,\varphi) \in \bH \times \rQ.
\end{equation}
% Note that the last term is well-defined when using the map $\cQ:\rH^{-1}(\Omega) \to \mathbb{P}_k(\cT_h) \subset \rL^{4/3}(\Omega)$. We also denote by $B$ and $B^t$ the operators induced by the bilinear form $b(\bullet,\bullet)$
% \[B: \bH \to \rQ', \quad B(\bxi)(\varphi) := b(\bxi,\varphi), \quad B^t: \rQ \to \bH', \quad B^t(\varphi)(\bxi):=b(\bxi,\varphi)\]
% and define similarly the operators $A:\bH \to \bH'$ and $C:\rQ\to\rQ'$ through $a(\bullet,\bullet)$ and $c(\bullet,\bullet)$, respectively. 
% With these we regard \eqref{eq:weak} as a perturbed saddle-point problem  with operator $\begin{pmatrix} A & B^* \\ B & -C \end{pmatrix}$ in turn perturbed with $d_{\bu}(\bullet,\bullet)$ and having right-hand side $(F,G)^T$.
The well-posedness analysis follows from \cite[Theorem 3.4]{correa22}, along with the Banach--Ne\v{c}as--Babu\v{s}ka Theorem \cite[Theorem 2.6]{ern04}. In what follows, we set the necessary conditions on the bilinear forms to apply the aforementioned theorems. We start by proving the boundedness of the bilinear forms and functionals, which follows directly from the Cauchy--Schwarz and H\"older inequalities:
\begin{equation}\label{eq:bound-forms-and-functionals}
\begin{array}{cc}
\ds  \left|a(\bzeta^\star,\bxi)\right| 
    \leq \|a\| \|\bzeta^\star\|_\bH \|\bxi\|_\bH \,,\quad \left|b(\bxi,\varphi)\right|  
    \leq \|\bxi\|_{\bH} \|\varphi\|_{\rQ} \,,\quad \left|c(\psi^\star,\varphi)\right|
    \leq \|c\| \|\psi^\star\|_{\rQ} \|\varphi\|_\rQ, \\[2ex]
\left|d_{ {\bu_h}}(\bxi,\varphi)\right| 
    \leq \|d\| \|\bxi\|_{\bH} {\|\varphi\|_{\rQ}} \,,\quad \left|F(\bxi)\right|
    \leq \|F\| \|\bxi\|_\bH \,,\quad
    \left|G(\varphi)\right|
    \leq \|G\| \|\varphi\|_\rQ.
    \end{array}
\end{equation}
% \begin{lemma}\label{lem:bound}
% The bilinear forms, linear functionals, and trilinear perturbation are all bounded in their respective norms:
% \begin{subequations}
% \begin{align}
%     \left|a(\bzeta,\bxi)\right| 
%     &= \left|\int_\Omega \frac{1}{\varepsilon} \bzeta\, \bxi \right|
%     \leq \frac{1}{\underline{\varepsilon}} \|\bzeta\|_{0,\Omega} \|\bxi\|_{0,\Omega}
%     \leq \|a\| \|\bzeta\|_\bH \|\bxi\|_\bH, \\
%     \left|b(\bxi,\varphi)\right| 
%     &= \left|\int_\Omega \varphi\, \vdiv \bxi\right|
%     \leq \|\varphi\|_{0,4;\Omega} \|\vdiv\bxi\|_{0,\frac{4}{3};\Omega} 
%     \leq \|b\| \|\varphi\|_{\rQ} \|\bxi\|_{\bH}, \\
%     \left|c(\psi,\varphi)\right| 
%     &= \left|\int_\Omega \kappa(\bar\psi)\, \psi\, \varphi \right|
%     \leq \|\kappa(\bar\psi)\|_{0,\Omega} \|\psi\|_{\rQ} \|\varphi\|_\rQ
%     \leq \|c\| \|\psi\|_{\rQ} \|\varphi\|_\rQ, \\
%    \label{inq:upper-bounds-d} \left|d_{ {\bu}}(\bxi,\varphi)\right| 
%     &= \left|\int_\Omega \frac{1}{\varepsilon} \, \varphi\,  {\bu} \cdot \bxi\right|
%     \leq \|a\| \|\psi\|_{0,4;\Omega} \| {\bu}\|_{0,4;\Omega} \|\bxi\|_{0,\Omega}
%     \leq \|a\| \| {\bu}\|_{0,4;\Omega} \|\psi\|_{\rQ} \|\bxi\|_{\bH}, \\
%     \left|F(\bxi)\right| 
%     &= \left|\langle \bxi \cdot \bn, \psi_D \rangle\right|
%     \leq \|\bxi \cdot \bn\|_{-\frac12,00;\Gamma_D} \|\psi_D\|_{\frac12,00;\Gamma_D}
%     \leq \|F\| \|\bxi\|_\bH, \\
%     \left|G(\psi)\right| 
%     &= \left|\int_\Omega \cQ g\, \psi\right|
%     \leq \|G\| \|\psi\|_\rQ,
% \end{align}\end{subequations}
 Here, the bounding constants are 
 \[
\|a\| := \frac{1}{\underline{\varepsilon}}, \; 
\|c\| := \|\kappa\|_{{0,2;\Omega}}, \; 
\|d\| := \frac{1}{\underline{\varepsilon}} \|{\bu_h}\|_{0,4;\Omega}, \; 
\|F\| := C_T \|\psi_D\|_{1/2,00;\Gamma_D}, \; 
\|G\| := \|\cQ g\|_{0,4/3;\Omega},
\]
 %$\|a\|:=\dfrac{1}{\underline{\varepsilon}}$, $\|c\|:=\|\kappa\|_{0,\Omega}$, $\|d\|:= \dfrac{1}{\underline{\varepsilon}} \|\bu\|_{0,4;\Omega}$, $\|F\|:=C_T \|\psi_D\|_{1/2,00;\Gamma_D}$, and $\|G\|:=\|\cQ g\|_{0,4/3;\Omega}$, 
where $C_T$ denotes the constant from the trace inequality. 
%\end{lemma}
% \begin{proof}
% It follows straightforwardly from H\"older and Cauchy--Schwarz inequalities in combination with the definition of the $\bH$ and $\rQ$ norms. 
% \end{proof}

The following result establishes the inf{--}sup condition for the bilinear form $b(\bullet,\bullet)$.
\begin{lemma}\label{lem:inf}
There exists a positive constant $\beta$ such that 
\begin{equation}
\sup_{\cero\neq\bxi \in \bH}\dfrac{b(\bxi,\varphi)}{\|\bxi\|_\bH}\geq \beta \|\varphi\|_\rQ \qquad \forall \varphi \in \rQ.
\end{equation}
\end{lemma}
\begin{proof}
For a given $\varphi \in \rQ$, we set $\widehat{\varphi} := |\varphi|^{2} \varphi$. Then, the following holds (cf. \cite[Lemma 2.2]{gatica22}):
 \begin{equation}\label{eq:dual}
\widehat{\varphi}\in \rL^{4/3}(\Omega) \qan \int_\Omega \varphi\, \widehat\varphi = \|\varphi\|_{0,4;\Omega}\|\widehat{\varphi}\|_{0,4/3;\Omega}
 \end{equation}
Now, let $z \in \rH^1_D(\Omega)$ be the unique solution (by the Lax--Milgram lemma) of the auxiliary problem  
\begin{equation}\label{eq:aux-problem}
-\Delta z = \widehat{\varphi}\quad  \text{in } \Omega \,,\qquad z = 0 \quad  \text{on } \Gamma_D \,,\qquad \nabla z \cdot \bn = 0  \quad\text{on } \Gamma_N .  
\end{equation}
{Using H\"older's inequality, the Sobolev embedding and Poincar\'e's inequality we obtain the bound
\begin{equation}\label{eq:dep-z-phi}
  \|\nabla z\|_{0,2;\Omega} \leq C_\mathrm{Sob}(1+c_P^2)^{1/2} \|\widehat{\varphi}\|_{0,4/3;\Omega}.
\end{equation}
}
%Using Poincar\'e’s inequality (with constant $c_P$), H\"older’s inequality, and the Sobolev embedding $\|\varphi\|_{0,4;\Omega}\leq C_{\mathrm{Sob}} \|\varphi\|_{{1,2;}\Omega}$ for all $\varphi\in \rH^1(\Omega)$, we obtain the following bound  
%\begin{equation}\label{eq:dep-z-phi}
%  \|z\|_{{1;\Omega\Omega} \leq \frac{C_{\mathrm{Sob}}}{c_P^2} \|\widehat{\varphi}\|_{0,4/3;\Omega}.
%\end{equation}
Then, we define $\widehat{\bxi}:=-\nabla z \in \bL^2(\Omega)$, and observe that (see \eqref{eq:aux-problem})
$\vdiv \widehat{\bxi}=-\Delta z = \widehat{\varphi}$ in $\Omega$, with $\widehat{\bxi}\cdot \bn |_{\Gamma_N}=0$. Hence, $\widehat{\bxi} \in \bH$, and combined with \eqref{eq:dep-z-phi}, this implies
\begin{equation*}
  \|\widehat{\bxi}\|_\bH\leq \left({C_\mathrm{Sob}(1+c_P^2)^{1/2}} + 1\right)\|\widehat\varphi\|_{0,4/3;\Omega}.
\end{equation*}
% Now, let $\widehat{\bxi}:=-\nabla z \in \bL^2(\Omega)$. Therefore,
% from \eqref{eq:aux-problem} we have $\vdiv \widehat{\bxi}=-\Delta z = \widehat{\varphi}$ in $\Omega$, implying that $\vdiv \widehat{\bxi} \in \rL^{\frac{4}{3}}(\Omega)$. Also, $\widehat{\bxi}\cdot \bn |_{\Gamma_N}= - \nabla z \cdot \bn |_{\Gamma_N}= 0 $, yielding $\widehat{\bxi} \in \bH$. Hence 
%    \begin{align}
%     \|\widehat{\bxi}\|_\bH=\|\widehat{\bxi}\|_{0,\Omega}+\|\vdiv\widehat{\bxi}\|_{0,\frac{4}{3};\Omega}&=\|\nabla z\|_{0,\Omega}+\|\widehat\varphi\|_{0,\frac{4}{3};\Omega} \notag\\
%     %&\leq \|z\|_{1,\Omega}+\|\widehat\varphi\|_{0,\frac{4}{3};\Omega} \notag \\ 
%     &\leq \left(\frac{C_{\mathrm{Sob}}}{c^2_P} + 1\right)\|\widehat\varphi\|_{0,\frac{4}{3};\Omega}.\label{eq:beta}
%    \end{align} 
Finally, using the definition of $b(\bullet,\bullet)$ (cf. \eqref{eq:forms-and-functionals}), the second part of \eqref{eq:dual}, and the last estimate, we obtain 
\begin{align*}
\ds \sup_{\cero \neq\bxi \in \bH} \dfrac{b(\bxi,\varphi)}{\|\bxi\|_\bH}\geq \dfrac{b(\widehat{\bxi},\varphi)}{\|\widehat{\bxi}\|_\bH}\geq \dfrac{\|\varphi\|_\rQ \|\widehat\varphi\|_{0,4/3;\Omega}}{({C_\mathrm{Sob}(1+c_P^2)^{1/2}} + 1)\|\widehat\varphi\|_{0,4/3;\Omega}}\geq \beta \|\varphi\|_\rQ,
\end{align*}
with $\beta:=\left( {C_\mathrm{Sob}(1+c_P^2)^{1/2}} + 1\right)^{-1}$. %Therefore, the proof is complete.
\end{proof}

{Let} $\bB:\bH\rightarrow\rQ'$ and $\bB^t:\rQ\rightarrow\bH'$ be the bounded linear operators induced by $b(\bullet,\bullet)$ 
\[
\begin{aligned}
\bB(\bxi)(\varphi)&:=b(\bxi,\varphi) && \forall \bxi\in \bH,\,\, \forall\varphi \in \rQ\\
\qan \bB^t(\varphi)(\bxi)&:=b(\bxi,\varphi) && \forall\varphi \in \rQ,\,\,\forall \bxi\in \bH,
    \end{aligned}
\]
we have the following kernel characterisations
\begin{subequations}\label{eq:V_W}
\begin{align}
\bV &: = \mathrm{ker}(\bB) = \{\bxi \in \bH: b(\bxi,\varphi) = 0 \quad \forall \varphi \in \rQ\} = \{   \bxi \in \bH: \vdiv \bxi = 0\quad  \mathrm{in}\ \Omega\}, \label{eq:KerV}\\
\rW &:= \mathrm{ker}(\bB^t) = \{\varphi \in \rQ: b(\bxi,\varphi) = 0 \quad \forall \bxi \in \bH\} = \{0\}. \label{eq:KerW}
\end{align}
\end{subequations}

\begin{remark}[Role of the assumption $|\Gamma_D|>0$]\label{rem:GammaD}

The assumption $|\Gamma_D|>0$ is used for: (i) the Poincar\'e inequality on $\rH^1_D(\Omega)$, which gives coercivity in Proposition~\ref{prop:primal} (note that $\kappa$ may vanish) and quantifies~\eqref{eq:assumption:u} and~\eqref{eq:assumption:uh}; (ii) the unique solvability of~\eqref{eq:aux-problem} and the bound~\eqref{eq:dep-z-phi}, which give $\beta$; and (iii) $\ker(\bB^t)=\{0\}$ and $\vdiv(\bH_h)=\rQ_h$. If $\Gamma_D=\emptyset$, then $b(\bxi,1)=\int_\Omega\vdiv\bxi=0$ for all $\bxi\in\bH$, the constants belong to $\ker(\bB^t)$ and Lemma~\ref{lem:inf} fails. The case $\Gamma_N=\emptyset$ is covered by the analysis, with $\rH^1_D(\Omega)=\rH^1_0(\Omega)$.

\end{remark}

From the definition of $\bV$ and the boundedness of the permittivity, it readily follows that
\begin{equation}\label{eq:a-ellipticity}
\ds  a(\bxi,\bxi)  \geq \alpha_a\|\bxi\|^2_{\bH} \quad \forall \bxi \in \bV\, ,\qquad 
 c(\varphi,\varphi)  \geq 0 \quad \forall \varphi \in \rQ, 
\end{equation}
with $\alpha_a:=\dfrac{1}{{\overline{\varepsilon}}}$. {
    Moreover, both bilinear forms $a(\bullet,\bullet)$ and $c(\bullet,\bullet)$ are symmetric, and the non-negativity of $c(\bullet,\bullet)$ follows from the assumption $\kappa\geq 0$.} 

{By the boundedness} of the involved bilinear forms \eqref{eq:bound-forms-and-functionals}, the inf{--}sup condition of the bilinear form $b(\bullet,\bullet)$ in Lemma~\ref{lem:inf}, the characterisation of the kernel in \eqref{eq:V_W}, the ellipticity of $a(\bullet,\bullet)$ on the kernel of $\bB$ \eqref{eq:a-ellipticity}, and {the symmetry and positive semidefiniteness of $c(\bullet,\bullet)$}, we have that the hypotheses of \cite[Theorem 3.4]{correa22} are satisfied for  $A(\bullet,\bullet)$, implying that there exists a constant ${\alpha_{\cA}}>0$ such that
\begin{equation} \label{eq:A-inf-sup}
  \sup_{\cero\neq(\bxi,\varphi) \in \bH \times \rQ}\dfrac{{A}((\bzeta^\star,\psi^\star),(\bxi,\varphi))}{\|(\bxi,\varphi)\|_{\bH\times\rQ}} \geq \alpha_\cA   \|(\bzeta^\star,\psi^\star)\|_{\bH\times\rQ} \quad \forall (\bzeta^\star,\psi^\star)\in{\bH\times\rQ},
\end{equation}
with $\alpha_{\cA} :=\dfrac{1}{2\tilde{c}}$, where $\tilde{c}:=\max\{c_3,c_4\}$, and
\begin{equation}\label{eq:C-A}
\begin{aligned}
c_1 &:= \frac{1}{{\alpha_a}} 
+ \frac{1}{\beta} \left(1 + \frac{\|a\|}{{\alpha_a}} \right) 
+ \frac{\|c\|}{\beta^2} \left(1 + \frac{\|a\|}{{\alpha_a}} \right)^2, \\
c_2 &:= \frac{1}{\beta} \left(1 + \frac{\|a\|}{{\alpha_a}} \right) 
\left\{ 1 + \frac{\|a\|}{\beta} 
+ \frac{\|a\|^2 \|c\|}{\beta^3} \left(1 + \frac{\|a\|}{{\alpha_a}} \right) \right\}, \\
c_3 &:= \frac{1}{{\alpha_a}} 
+ \frac{1}{\beta} \left(1 + \frac{\|a\|}{{\alpha_a}} \right) 
\left\{ 1 + \left(1 + \frac{\|a\|}{\beta} \right) 
\left(2 \|c\| \max\{c_1, c_2\} \right)^{1/2} \right\}, \\
c_4 &:= \frac{1}{\beta} \left(1 + \frac{\|a\|}{{\alpha_a}} \right) 
\left(1 + \frac{\|a\|}{\beta} \right) \left\{ 1 + \left(2 \|c\| \max\{c_1, c_2\} \right)^{1/2} \right\}.
\end{aligned}
\end{equation}
{Furthermore}, using the bound of the bilinear form $d_{{\bu_h}}(\bullet,\bullet)$ (cf. \eqref{eq:bound-forms-and-functionals}), the definition of the bilinear forms $A(\bullet,\bullet)$ and $A_{\bu_h}(\bullet,\bullet)$ (cf. \eqref{eq:A} and \eqref{eq:Au}), and the global inf{--}sup condition \eqref{eq:A-inf-sup}, we can readily obtain
\begin{equation} 
\sup_{\cero\neq(\bxi,\varphi) \in \bH \times \rQ} \dfrac{ A_{{\bu_h}}((\bzeta^\star,\psi^\star),(\bxi,\varphi))}{\|(\bxi,\varphi)\|_{\bH\times\rQ}} \geq \left\{\alpha_{\cA} - \|a\|\,\| {{\bu_h}}\|_{0,4;\Omega}\right\}   \|(\bzeta^\star,\psi^\star)\|_{\bH\times\rQ} \quad \forall (\bzeta^\star,\psi^\star)\in{\bH\times\rQ}.
\end{equation}
Then, assuming that (see also Remark~\ref{rem:restrictive} below) 
\begin{equation}\label{eq:small-u}
  \| {{\bu_h}}\|_{0,4;\Omega} \leq \frac{\alpha_{\cA}}{2\|a\|},
\end{equation} 
%{(TF: please, someone check how the latter bound relates to the assumptions on the norm of $\bu_h$ in the primal system!)}
a global inf{--}sup stability is established for $A_{{\bu_h}}$: 
\begin{equation} \label{eq:Au-inf-sup}
  \sup_{\cero\neq(\bxi,\varphi) \in \bH \times \rQ}\dfrac{A_{{\bu_h}}((\bzeta^\star,\psi^\star),(\bxi,\varphi))}{\|(\bxi,\varphi)\|_{\bH\times\rQ}} \geq \frac{\alpha_{\cA}}{2}   \|(\bzeta^\star,\psi^\star)\|_{\bH\times\rQ} \quad \forall (\bzeta^\star,\psi^\star)\in{\bH\times\rQ}.
\end{equation}
Similarly, by the symmetry of {$A(\bullet,\bullet)$} and under the same assumption on {$\bu_h$} \eqref{eq:small-u}, we obtain 
\begin{equation} \label{eq:Au-inf-sup-sym}
  \sup_{\substack{0 \neq (\bzeta^\star,\psi^\star) \in \bH \times \rQ }}\dfrac{A_{{\bu_h}}((\bzeta^\star,\psi^\star),(\bxi,\varphi))}{\|(\bzeta^\star,\psi^\star)\|_{\bH\times\rQ}} \geq \frac{\alpha_\cA}{2} \|(\bxi,\varphi)\|_{\bH\times\rQ}   \quad \forall (\bxi,\varphi)\in{\bH\times\rQ}.
\end{equation}
In this way, we are now able to establish the well-posedness of \eqref{eq:weak} (equivalently \eqref{eq:Gweak}).
\begin{theorem}\label{th:solv}
  Given ${\bu_h}\in \bL^4(\Omega)$ satisfying \eqref{eq:small-u}, there exists a unique $(\bzeta^\star,\psi^\star)\in \bH \times \rQ$ solution to \eqref{eq:weak}, and there holds
    \begin{equation}\label{eq:stability}
        \|\bzeta^\star\|_\bH+\|\psi^\star\|_\rQ\leq \frac{2}{\alpha_{\cA}}\max\{1,C_T\}\left\{\|\psi_D\|_{1/2,00;\Gamma_D}+\|\cQ g\|_{0,4/3;\Omega}\right\}.
    \end{equation}
\end{theorem}
\begin{proof} 
Since \eqref{eq:Au-inf-sup} and \eqref{eq:Au-inf-sup-sym}  guarantee that $A_{{\bu_h}}(\bullet,\bullet)$ satisfies the hypotheses of the BNB Theorem \cite[Theorem 2.6]{ern04}, the proof becomes a direct application of this theorem. Hence, the estimate \eqref{eq:stability} follows from \eqref{eq:Gweak} together with the boundedness of $F(\bullet)$  and $G(\bullet)$ (cf. \eqref{eq:bound-forms-and-functionals}).
\end{proof}
\begin{remark}\label{rem:restrictive}
    {Since $\|a\|=1/\underline\varepsilon$, \eqref{eq:small-u} reads $\|\bu_h\|_{0,4;\Omega}\leq\frac{\alpha_\cA}{2}\underline\varepsilon$, while the primal problem only requires~\eqref{eq:assumption:uh}. The constant $\alpha_\cA$ also depends on $\|\kappa\|_{0,2;\Omega}$, $\overline\varepsilon$ and $\beta$ (cf.~\eqref{eq:C-A}), so~\eqref{eq:small-u} is in general more restrictive. Both conditions are sufficient but not necessary.}
\end{remark}

%%%%%%%%%%%%%%%%%%%%%%%%%%%%%%%%%%%%%%%%%%%%%%%%%%%%%%%%%%%%%%%%%%%%%%%%%%%%%%%%%%%%%%%%%%%%%
\section{Discrete problem}\label{sec:fe}
Let us denote by $\bM_h$ a finite-dimensional subspace of $\bL^4(\Omega)$ (for example, $[\rP_k(\cT_h)]^d$), and consider for this section that the given vector field $\bu_h \in \bM_h$. 

Following the usual approach, we obtain the Galerkin discretisation of \eqref{eq:PB1-star}-\eqref{eq:PB2-star} as a perturbed saddle-point problem that looks for  $(\bzeta_h,\psi_h) \in \bH_h\times \rQ_h$ such that 
\begin{equation} \label{eq:weak-h}
\begin{aligned}
    a(\bzeta_h,\bxi_h) &+ b(\bxi_h,\psi_h) + d_{\bu_h}(\bxi_h,\psi_h) &= F(\bxi_h) && \forall \bxi_h \in \bH_h, \\
    b(\bzeta_h,\varphi_h) &- c(\psi_h,\varphi_h) &= G(\varphi_h) && \forall \varphi_h \in \rQ_h,
\end{aligned}
\end{equation}
 
We remark that, for our particular choice of Raviart--Thomas elements (and similarly for other finite elements such as the Brezzi--Douglas--Marini (BDM) elements), the following property holds. Since we will later need a more explicit characterisation of the kernel of the discrete divergence operator, denoted by $\bV_h$, this property will be used
%We remark that, for our particular choice of Raviart--Thomas elements (and similarly for other finite elements such as The Brezzi--Douglas--Marini (BDM) elements), the discrete space $\bH_h$ is contained in $\bRT_k(\cT_h)$. In addition, the divergence of any function in $\mathbb{RT}_k(K)$ is in $\mathbb{P}_k(K)$, therefore, the following inclusion holds 
\begin{equation}\label{eq:kernel-assumption}
    \vdiv \bH_h \subseteq \rQ_h.
\end{equation}

Let us recall some preliminary results. First, we consider the quasi-interpolation operator defined in \cite[Theorem~2.2]{Ern18}{ (see also \cite{ern2016mollification} and, for mixed boundary conditions, \cite{licht2019smoothed}, where $\rL^p$-stable commuting projections preserving the boundary condition on $\Gamma_N$ are constructed)},
$
\Pi^{\mathsf{RT}}_{\rqu}:\bH \rightarrow \bH_h$,
which satisfies the following properties:
 \begin{itemize}
     \begin{subequations}
    \item For all $\bxi_h \in \bH_h$, we have
     \begin{align}\label{eq:quasi-invariant}
           \Pi^{\mathsf{RT}}_{\rqu}(\bxi_h)=\bxi_h.
     \end{align}
     \item For $1\leq p \leq \infty$, there exists $C_{\rqu}$, independent of $h$, such that
        \begin{align}\label{eq:quasi-bdd}
            \|\Pi^{\mathsf{RT}}_{\rqu} \bxi\|_{0,p;\Omega} \leq C_{\rqu} \| \bxi\|_{0,p;\Omega}, 
            % \label{eq:quasi-stable}
            % \text{and } \|\bxi-\Pi^{\mathsf{RT}}_{\rqu}(\bxi)\|_{0,p;\Omega} &\leq C_{\rqu} \underset{\bxi_h\in \bH_h}{\inf}\|\bxi-\bxi_h\|_{0,p;\Omega}, 
        \end{align}
      for all $\bxi \in \rL^p(\Omega)$.
        \end{subequations}
     \item The following commuting property (cf. \cite[Section 5.3]{Colmenares20}) holds 
        \begin{equation}\label{eq:commuting}
          \vdiv(\Pi^{\mathsf{RT}}_{\rqu}\bxi) = {\mP_h^k}(\vdiv\,\bxi)
         \quad \forall \bxi \in \bH,
        \end{equation}
where $\mP_h^k:\rL^1(\Omega)\rightarrow\rP_k(\cT_h)$ is the $\rL^2(\Omega)$-orthogonal projector {introduced in Section~\ref{sec:cont}}.
%the usual $\rL^2(\Omega)$-orthogonal projector.
%{TF: maybe explicitely mention that $\mP_h^k$ is well defined for $p<2$? ?}
 \end{itemize}
%{TF: Does someone has a reference for the next result? \
{The following result can be found in 
%for $p \in (1,\infty)$ it is in 
\cite[Eq. (4.4)]{gatica22}.}
{\begin{lemma}\label{lem:boundednessL2projector}
  The $\rL^2(\Omega)$-orthogonal projector onto the piecewise polynomial space $\rQ_h$ is bounded as a mapping from $\rL^p(\Omega)\to \rL^p(\Omega)$ for all $p\in[1,\infty]$.
\end{lemma}
}

In order to analyse the discrete solvability of \eqref{eq:weak-h} we use \cite[Theorem 3.5]{correa22}. We begin by considering $\bB_h : \bH_h \to \rQ_h'$ and $\bB_h^t : \rQ_h \to \bH_h'$ to be the discrete bounded linear operators induced by $b(\bullet,\bullet)$, that is
\[
\begin{aligned}
\bB_h(\bxi_h)(\varphi_h) &:= b(\bxi_h, \varphi_h) 
&& \forall\, \bxi_h \in \bH_h, \; \forall\, \varphi_h \in \rQ_h, \\
\bB_h^t(\varphi_h)(\bxi_h) &:= b(\bxi_h, \varphi_h) 
&& \forall\, \varphi_h \in \rQ_h, \; \forall\, \bxi_h \in \bH_h.
\end{aligned}
\]

Using the property \eqref{eq:kernel-assumption} and proceeding similarly as in the continuous case, we have that the following discrete kernel characterisations hold
\begin{subequations}
\begin{align}
\bV_h &:= \ker(\bB_h) 
    = \left\{ \bxi_h \in \bH_h : b(\bxi_h,\varphi_h) = 0, \quad \forall\, \varphi_h \in \rQ_h \right\} 
    \equiv \left\{ \bxi_h \in \bH_h : \vdiv \bxi_h = 0  \right\}, \label{eq:KerV-h}\\
    \rW_h &:= \ker(\bB_h^t) 
    = \left\{ \varphi_h \in \rQ_h : b(\bxi_h,\varphi_h) = 0, \quad \forall\, \bxi_h \in \bH_h \right\} 
    \equiv \{0\}, \label{eq:KerW-h}
\end{align}\end{subequations}
thus $\bV_h\subseteq\bV$, and from \eqref{eq:a-ellipticity}, we obtain
\begin{equation}\label{eq:a-ellipticity-d}
    a(\bxi_h,\bxi_h)  \geq \alpha_a\|\bxi_h\|^2_{\bH} \quad \forall \bxi_h \in \bV_h,\qquad 
        c(\varphi_h,\varphi_h)  \geq 0 \quad \forall \varphi_h \in \rQ_h,
\end{equation}
with $\alpha_a$ the same constant as in \eqref{eq:a-ellipticity}.
\begin{lemma}\label{lem:inf-h}
There exists a positive constant $\beta_d>0$, independent of h, such that 
\begin{equation}\label{eq:inf-h}
    \sup_{\cero\neq\bxi_h \in \bH_h}
    \dfrac{b(\bxi_h,\varphi_h)}{\|\bxi_h\|_\bH}
    \geq \beta_d \|\varphi_h\|_\rQ 
    \qquad \forall \varphi_h \in \rQ_h.
\end{equation}
\end{lemma}

\begin{proof} 
 We first invoke the continuous inf{--}sup condition (cf.~Lemma~\ref{lem:inf}) where for a given $\varphi_h \in \rQ_h\subseteq \rQ$, we have $\bar{\varphi}_h\in\rL^{4/3}(\Omega)$ such that there exists $\bxi\in \bH,$ with $\vdiv\bxi=\bar{\varphi}_h$ and 
 \begin{equation}\label{eq:xi-estimate}    \|\bxi\|_\bH\leq\frac1\beta\|\bar{\varphi}_h\|_{0, 4/3;\Omega}.
 \end{equation}
 Next, we take $\bar{\bxi}_h\in \bH_h$ to be the quasi-interpolation of $\bxi$ (i.e. $\Pi^{\mathsf{RT}}_{\rqu}\bxi=\bar{\bxi}_h$). Then, using \eqref{eq:quasi-bdd} and the definition of the $\bH$-norm, we obtain
 \begin{equation}\label{eq:xi_h-estimate}
     \|\bar{\bxi}_h\|_{{0,2;\Omega}}\leq C_{\rqu}\|\bxi\|_{{0,2;\Omega}}\leq C_{\rqu}\|\bxi\|_\bH.
\end{equation}
On the other hand, by applying the commuting property of $\Pi^{\mathsf{RT}}_{\rqu}$ (cf. \eqref{eq:commuting}), we have 
\begin{equation}
\vdiv\bar{\bxi}_h=\vdiv\Pi^{\mathsf{RT}}_{\rqu}\bxi=\mP_h^k(\vdiv\bxi),
\end{equation}
which together with \eqref{eq:xi_h-estimate}, \eqref{eq:xi-estimate}, the continuity of $\mP_h^k$ (with $C_{\mP}>0$, independent of $h$), the definition of the $\bH$-norm, and \eqref{eq:xi-estimate}, gives
\begin{equation}\label{eq:beta-h}
\begin{array}{ll}
  \ds \|\bar{\bxi}_h\|_\bH=\|\bar{\bxi}_h\|_{{0,2;\Omega}}+\|\vdiv\bar{\bxi}_h\|_{0,\frac43;\Omega} \leq \frac{C_{\rqu}}{\beta}\|\bar{\varphi}_h\|_{0,\frac43;\Omega}+\|\mP_h^k(\vdiv\bxi)\|_{0,\frac43;\Omega} \\
\ds \quad \leq \frac{C_{\rqu}}{\beta}\|\bar{\varphi}_h\|_{0,\frac43;\Omega}+C_\mP\|\vdiv\bxi\|_{0,\frac43;\Omega} \leq \frac{C_{\rqu}}{\beta}\|\bar{\varphi}_h\|_{0,\frac43;\Omega}+C_\mP\|\bxi\|_\bH  \leq \left(\frac{C_{\rqu}+C_\mP}{\beta}\right)\|\bar{\varphi}_h\|_{0,\frac43;\Omega}. 
\end{array}    
\end{equation}
Finally, according to the definition of $b(\bullet,\bullet)$ and \eqref{eq:beta-h}, we obtain 
\begin{align*}
 \sup_{\substack{0\neq\bxi_h \in \bH_h }}\dfrac{b(\bxi_h,\varphi_h)}{\|\bxi_h\|_\bH} & \geq  \dfrac{b(\bar{\bxi}_h,\varphi_h)}{\|\bar{\bxi}_h\|_\bH} \, = \, \dfrac{\displaystyle\int_\Omega \varphi_h \vdiv\bar{\bxi}_h}{\|\bar{\bxi}_h\|_\bH} \, = \, \dfrac{\displaystyle\int_\Omega \varphi_h \mP_h^k(\vdiv\bxi)}{\|\bar{\bxi}_h\|_\bH}
 \, = \, \dfrac{\displaystyle\int_\Omega \varphi_h \mP_h^k(\bar{\varphi}_h)}{\|\bar{\bxi}_h\|_\bH}\\
 & =\, \dfrac{\displaystyle\int_\Omega \varphi_h \bar{\varphi}_h}{\|\bar{\bxi}_h\|_\bH} \, \geq \,  \frac{\beta}{C_{\rqu}+C_\mP}\frac{\|\varphi_h\|_{\rQ}\|\bar{\varphi}_h\|_{0,\frac43;\Omega}}{\|\bar{\varphi}_h\|_{0,\frac43;\Omega}} \, = \, \frac{\beta}{C_{\rqu}+C_\mP}\|\varphi_h\|_{\rQ},
\end{align*}
with $\beta_d:=\dfrac{\beta}{C_{\rqu}+C_\mP}$.
\end{proof}

Consequently, by applying \cite[Theorem 3.5]{correa22} directly, we obtain that the discrete global inf–sup condition for $A(\bullet,\bullet)$ (cf. \eqref{eq:A}) holds with {$\alpha_{\cA,d} > 0$}, where the constant depends only on ${\alpha_a}, \,\beta_d, \, \|a\|$, and $\|c\|$ (cf.~\eqref{eq:C-A}{, with $\beta$ replaced by $\beta_d$}), {and is therefore independent of $h$}. As a result, we have the following discrete analogue of \eqref{eq:Au-inf-sup}, that is
 \begin{equation}
  \sup_{\substack{(\bxi_h,\varphi_h) \in \bH_h \times \rQ_h \\(\bxi_h,\varphi_h) \neq 0}}\dfrac{A_{{\bu}_h}((\bzeta_h,\psi_h),(\bxi_h,\varphi_h))}
        {\|(\bxi_h,\varphi_h)\|_{\bH\times\rQ}}\geq \left\{\alpha_{\cA,d} - \|a\|\|{\bu}_h\|_{0,4;\Omega}\right\}\|(\bzeta_h,\psi_h)\|_{\bH\times\rQ},
\end{equation}
for all $(\bzeta_h,\psi_h)\in\bH_h\times\rQ_h$.
And under the following discrete version of the data assumption  
\begin{equation}\label{eq:small-u-h}
    \| {\bu_h}\|_{0,4;\Omega} \leq \frac{\alpha_{\cA,d}}{2\|a\|},
\end{equation} 
we obtain
\begin{equation} \label{eq:Au-inf-sup-h}
  \sup_{\substack{0\neq(\bxi_h,\varphi_h) \in \bH_h \times \rQ_h }}
  \dfrac{A_{{\bu}_h}((\bzeta_h,\psi_h),(\bxi_h,\varphi_h))}
        {\|(\bxi_h,\varphi_h)\|_{\bH\times\rQ}}
  \geq \frac{\alpha_{\cA,d}}{2}\,
        \|(\bzeta_h,\psi_h)\|_{\bH\times\rQ}
  \qquad 
  \forall (\bzeta_h,\psi_h)\in \bH_h \times \rQ_h.
\end{equation}
%  Analogously, using the symmetry of $A(\bullet,\bullet)$ and the assumption on ${\bu}_h$ \eqref{eq:small-u-h}, we can assert that 
% \begin{equation} \label{eq:Au-inf-sup-sym-h}
%   \sup_{\substack{(\bzeta_h,\psi_h) \in \bH_h \times \rQ_h \\ (\bzeta_h,\psi_h) \neq 0}}
%     \frac{A_{{\bu}_h}\big((\bzeta_h,\psi_h),(\bxi_h,\varphi_h)\big)}
%          {\|(\bzeta_h,\psi_h)\|_{\bH\times\rQ}}
%     \geq \frac{\alpha_{A,d}}{2} 
%           \|(\bxi_h,\varphi_h)\|_{\bH\times\rQ}
%   \qquad 
%   \forall\,(\bxi_h,\varphi_h)\in \bH_h \times \rQ_h.
% \end{equation}
As a consequence of the above bound, we are now in a position to establish the solvability of \eqref{eq:weak-h} (equivalently, the discrete counterpart of \eqref{eq:Gweak}).
\begin{theorem}\label{th:solv-h}
Given ${\bu_h} \in \bM_h \subseteq \bL^4(\Omega)$ satisfying $\|\bu_h\|_{0,4;\Omega} \leq \frac{\alpha_{\cA,d}}{2\|a\|}$, there exists a unique $(\bzeta_h,\psi_h)\in \bH_h \times \rQ_h$ solution to \eqref{eq:weak-h}, and there holds
\begin{equation}\label{eq:estability-h}
\|\bzeta_h\|_\bH+\|\psi_h\|_\rQ\leq \frac{2}{\alpha_{\cA,d}}\max\{1,C_T\}\left\{\|\psi_D\|_{\frac12,00;\Gamma_D}+\|\cQ g\|_{0,\frac43;\Omega}\right\}.
\end{equation}
\end{theorem}
\begin{proof}
The proof follows from \eqref{eq:Au-inf-sup-h} and from the fact that, in finite-dimensional linear problems, surjectivity and injectivity are equivalent. In addition, \eqref{eq:estability-h} is obtained by using arguments similar to those in Theorem \ref{th:solv}.
\end{proof}

%%%%%%%%%%%%%%%%%%%%%%%%%%%%%%%%%
\section{Error estimates}
\label{sec:strang}
\begin{theorem}\label{thm:quasiopt}
    {Assume that~\eqref{eq:small-u} and~\eqref{eq:small-u-h} hold, and} let $(\bzeta^\star,\psi^\star), (\bzeta_h,\psi_h)$ be the unique solutions of \eqref{eq:weak} and \eqref{eq:weak-h}, respectively. Then there exists $\hat{C}>0$, independent of $h$, such that  
\begin{equation}\label{eq:quasiopt}
\|\bzeta^\star-\bzeta_h\|_\bH+\|\psi^\star-\psi_h\|_\rQ%\\[1ex]
\ds   \leq \hat{C}\left\{\mathrm{dist}(\bzeta^\star,\bH_h)+\mathrm{dist}(\psi^\star,\rQ_h) \right\}.
\end{equation}
\end{theorem}
\begin{proof}
  %{Follows from the classical Babu\v{s}ka--Brezzi theory.} {[perhaps still write sketch of main steps]}
  %%%%%%%%%%%%%
 {The proof follows the same approach as  \cite[Theorem 4.7]{Camano21}. For a given $(\hat{\bxi}_h,\hat{\varphi}_h)\in \bH_h\times\rQ_h$, we decompose the errors as 
\begin{equation}\label{eq:decompose}
\boldsymbol{e}_{\bzeta^\star}
:= \boldsymbol{\eta}_{\bzeta^\star} + \boldsymbol{\chi}_{\bzeta^\star}
= (\bzeta^\star - \hat{\boldsymbol{\xi}}_h)
 + (\hat{\boldsymbol{\xi}}_h - \bzeta_h), \quad 
e_{\psi^\star}
:= \eta_{\psi^\star} + \chi_{\psi^\star}
= (\psi^\star - \hat{\varphi}_h)
 + (\hat{\varphi}_h - \psi_h).
\end{equation}
Next, we apply the Galerkin orthogonality property to ${A_{\bu_h}}(\bullet,\bullet)$ (cf.~\eqref{eq:Au} and \eqref{eq:Gweak})
and to its corresponding discrete formulation and substitute the decomposition of $\psi^\star$ from \eqref{eq:decompose}, to obtain 
\begin{equation}\label{eq:ortho}
A\bigl((\boldsymbol{e}_{\bzeta^\star}, e_{\psi^\star}),
       (\boldsymbol{\xi}_h,\varphi_h)\bigr)
+ d_{\bu_h}(\boldsymbol{\xi}_h{,}e_{\psi^\star}) = 0 \qquad {\forall(\bxi_h,\varphi_h)\in\bH_h\times\rQ_h}.
\end{equation}
Then, using the definition of $A_{\boldsymbol{u}_h}(\bullet,\bullet)$ in the above equality, together with \eqref{eq:decompose}, we get
\begin{equation*}
A_{\boldsymbol{u}_h}\bigl((\boldsymbol{\chi}_{\bzeta^\star},\chi_{\psi^\star}),
                          (\boldsymbol{\xi}_h,\varphi_h)\bigr)= -A\bigl((\boldsymbol{\eta}_{\bzeta^\star},\eta_{\psi^\star}),
(\boldsymbol{\xi}_h,\varphi_h)\bigr)-d_{\bu_h}(\boldsymbol{\xi}_h,\eta_{\psi^\star}).
\end{equation*}
Moreover, by expanding $A(\bullet,\bullet)$ (cf.\eqref{eq:A}), we have for all
$(\boldsymbol{\xi}_h,\varphi_h)\in \bH_h\times\rQ_h$
\begin{equation*}
A_{\boldsymbol{u}_h}
\bigl((\boldsymbol{\chi}_{\bzeta^\star},\chi_{\psi^\star}),
      (\boldsymbol{\xi}_h,\varphi_h)\bigr)
=
- a(\boldsymbol{\eta}_{\bzeta^\star},\boldsymbol{\xi}_h) - b(\boldsymbol{\eta}_{\bzeta^\star},\varphi_h) - b(\boldsymbol{\xi}_h,\eta_{\psi^\star}) {+} c(\eta_{\psi^\star},\varphi_h) - d_{\bu_h}(\boldsymbol{\xi}_h,\eta_{\psi^\star}).
\end{equation*}
 We now bound the right-hand side using the continuity of 
$a(\bullet,\bullet)$, $b(\bullet,\bullet)$, $c(\bullet,\bullet)$, and $d_{\bu_h}(\bullet,\bullet)$ (cf.~\eqref{eq:bound-forms-and-functionals}), 
the smallness assumption on $\bu_h$ \eqref{eq:small-u-h}, this gives
\begin{equation*}
|A_{\boldsymbol{u}_h}
\bigl((\boldsymbol{\chi}_{\bzeta^\star},\chi_{\psi^\star}),
      (\boldsymbol{\xi}_h,\varphi_h)\bigr)|\leq {\max\!\left\{1+\tfrac{\alpha_{\cA,d}}{2},\|a\|,\|c\|\right\}}
\,\bigl\|(\eta_{\bzeta^\star},\eta_{\psi^\star})\bigr\|_{\bH\times\rQ}
\,\bigl\|(\boldsymbol{\xi}_h,\varphi_h)\bigr\|_{\bH\times\rQ} .
\end{equation*}
For the left-hand side, we take the supremum over $(\bxi_h,\varphi_h)$, and apply the discrete inf{--}sup condition of $A_{\bu_h}(\bullet,\bullet)$ (cf. \eqref{eq:Au-inf-sup-h}), which gives
\begin{equation}\label{eq:chi-bound}
\|(\boldsymbol{\chi}_{\bzeta^\star},\chi_{\psi^\star})\|{_{\bH\times\rQ}}
\leq
\tfrac{2}{\alpha_{\cA,d}}
{\max\!\left\{1+\tfrac{\alpha_{\cA,d}}{2},\|a\|,\|c\|\right\}}
\|(\boldsymbol{\eta}_{\bzeta^\star},\eta_{\psi^\star})\|_{\bH\times\rQ} .
\end{equation}
Finally, combining \eqref{eq:decompose}, the triangle inequality, and \eqref{eq:chi-bound}, we obtain 
\begin{equation*}
\bigl\|(\boldsymbol{e}_{\bzeta^\star},e_{\psi^\star})\bigr\|{_{\bH\times\rQ}}
\lesssim
\bigl\|(\boldsymbol{\eta}_{\bzeta^\star},\eta_{\psi^\star})\bigr\|_{\bH\times\rQ},
\end{equation*}
and since $(\hat{\bxi}_h,\hat{\varphi}_h)\in \bH_h\times\rQ_h$ is arbitrary, the result follows.}
\end{proof}
%%%%%%%%%%%%%%%%%%%%%%%%%%%%%%%%%%%%%%%%%%%%%%%%%%%%%%%%%%%%%%%%%%%%%%%%%%%%%%%%%%%%%%%%%%%%%%%%%%%%%%%%%%%%%%%%%%%%%%%%%%%%%%%%%%

We are in a position to state our main result. It is a perturbed quasi-optimality result.
\begin{theorem}\label{thm:errorestimate}
  Let {$g\in\widetilde\rH^{-1}(\Omega)$} and $\psi_D\in \rH^{1/2}_{00}(\Gamma_D)$ be given{, and assume that $\bu$ satisfies~\eqref{eq:assumption:u} and that $\bu_h\in\bM_h$ satisfies~\eqref{eq:assumption:uh}, \eqref{eq:small-u} and~\eqref{eq:small-u-h}}. 
  Suppose that $\psi\in {\rH^1(\Omega)}$ solves~\eqref{eq:PBstrong:primal} {(in the sense of~\eqref{eq:primal:weak})} and set $\bzeta = {\varepsilon}\nabla\psi-\bu\psi\in{\bL^2(\Omega)}$, i.e., $(\bzeta,\psi)$ solves~\eqref{eq:PBstrong}.
  If $(\bzeta_h,\psi_h)\in \bH_h\times \rQ_h$ is the solution of~\eqref{eq:weak-h}, 
  then
  \begin{align}\label{eq:errorestimate}
     \|\psi-\psi_h\|_{0,4;\Omega} + \|\bzeta-\bzeta_h\|_{0,2;\Omega} 
    &\lesssim \min_{{(\btau_h,v_h)}\in\bH_h\times \rQ_h} (\|\psi-v_h\|_{0,4;\Omega} + \|\bzeta-\btau_h\|_{0,2;\Omega})
    \\
    &\quad+(\|g\|_{\rH^{-1}_D(\Omega)}+\|\psi_D\|_{1/2,00;\Gamma_D})\|\bu-\bu_h\|_{0,4;\Omega} + {\min_{\chi_h\in\rQ_h}\|\kappa\psi-\chi_h\|_{0,4/3;\Omega}},\nonumber 
  \end{align}
  {with a hidden constant independent of $h$. Moreover, the left-hand side of~\eqref{eq:errorestimate} can be replaced by $\|\bzeta^\star-\bzeta_h\|_{\bH}+\|\psi^\star-\psi_h\|_{\rQ}$, where $(\bzeta^\star,\psi^\star)$ is the solution of~\eqref{eq:weak}.}
\end{theorem}
\begin{proof}

  Let $(\bzeta^\star,\psi^\star)\in \bH\times\rQ$ be the solution of~\eqref{eq:weak}; $\psi^\star$ solves~\eqref{eq:aux:weak} and $\bzeta^\star=\varepsilon\nabla\psi^\star-\bu_h\psi^\star$.

  %\smallskip\noindent\textit{Step 1 (C\'ea estimate).} 
  Theorem~\ref{thm:quasiopt} with $\hat{\bxi}_h=\Pi^{\mathsf{RT}}_{\rqu}\bzeta^\star$ gives the C\'ea estimate 
  \begin{align}\label{eq:aux-cea}
    \|\bzeta^\star-\bzeta_h\|_{\bH} + \|\psi^\star-\psi_h\|_{0,4;\Omega} 
    \leq \hat C\bigl(\|(1-\Pi^{\mathsf{RT}}_{\rqu})\bzeta^\star\|_{\bH} + \|\psi^\star-v_h\|_{0,4;\Omega}\bigr)
    \qquad\forall v_h\in\rQ_h.
  \end{align}

  %\smallskip\noindent\textit{Step 2 (approximation of the regularised flux).} 
  Next, we consider an approximation of the regularised flux. For $\btau_h\in\bH_h$, \eqref{eq:quasi-invariant}--\eqref{eq:quasi-bdd} give
  \[
    \|(1-\Pi^{\mathsf{RT}}_{\rqu})\bzeta^\star\|_{0,2;\Omega}\leq (1+C_\rqu)\|\bzeta^\star-\btau_h\|_{0,2;\Omega}.
  \]
  For the divergence, \eqref{eq:commuting} and $\vdiv\bzeta^\star=\kappa\psi^\star-\cQ g$ yield
  \[
    \vdiv(1-\Pi^{\mathsf{RT}}_{\rqu})\bzeta^\star = (1-\mP_h^k)\vdiv\bzeta^\star = (1-\mP_h^k)(\kappa\psi^\star) - (1-\mP_h^k)\cQ g = (1-\mP_h^k)(\kappa\psi^\star),
  \]
  since $\cQ g\in\rQ_h$. Using Lemma~\ref{lem:boundednessL2projector} (with constant $C_\mP$) and $(1-\mP_h^k)\chi_h=0$ for $\chi_h\in\rQ_h$, we get
  \begin{align*}
    \|(1-\mP_h^k)(\kappa\psi^\star)\|_{0,4/3;\Omega} &\leq \|(1-\mP_h^k)(\kappa\psi-\chi_h)\|_{0,4/3;\Omega} + \|(1-\mP_h^k)(\kappa(\psi-\psi^\star))\|_{0,4/3;\Omega}
    \\
    &\leq (1+C_\mP)\bigl(\|\kappa\psi-\chi_h\|_{0,4/3;\Omega} + \|\kappa\|_{0,2;\Omega}\|\psi-\psi^\star\|_{0,4;\Omega}\bigr).
  \end{align*}
  Combining these bounds with the triangle inequality, we obtain
  \begin{align}\label{eq:aux-flux}
    \|(1-\Pi^{\mathsf{RT}}_{\rqu})\bzeta^\star\|_{\bH} \lesssim \|\bzeta-\btau_h\|_{0,2;\Omega} + \|\kappa\psi-\chi_h\|_{0,4/3;\Omega} + \|\bzeta-\bzeta^\star\|_{0,2;\Omega} + \|\psi-\psi^\star\|_{0,4;\Omega}
  \end{align}
  for all $\btau_h\in\bH_h$ and $\chi_h\in\rQ_h$.

  %\smallskip\noindent\textit{Step 3 (conclusion).} 
  To conclude, inserting~\eqref{eq:aux-flux} and $\|\psi^\star-v_h\|_{0,4;\Omega}\leq\|\psi-v_h\|_{0,4;\Omega}+\|\psi-\psi^\star\|_{0,4;\Omega}$ into~\eqref{eq:aux-cea} gives
  \begin{align}\label{eq:aux-star-h}
    \|\bzeta^\star-\bzeta_h\|_{\bH} + \|\psi^\star-\psi_h\|_{0,4;\Omega}
    &\lesssim \|\bzeta-\btau_h\|_{0,2;\Omega} + \|\psi-v_h\|_{0,4;\Omega} + \|\kappa\psi-\chi_h\|_{0,4/3;\Omega}
    \nonumber\\
    &\qquad+ \|\bzeta-\bzeta^\star\|_{0,2;\Omega} + \|\psi-\psi^\star\|_{0,4;\Omega},
  \end{align}
  for all $(\btau_h,v_h)\in\bH_h\times \rQ_h$ and $\chi_h\in \rQ_h$. The last two terms are bounded by Lemma~\ref{lem:quasiopt:primal} and~\eqref{eq:dualAndSobolevEstimate}, which proves the second assertion. Finally, the triangle inequality
  \begin{align*}
    \|\bzeta-\bzeta_h\|_{0,2;\Omega} + \|\psi-\psi_h\|_{0,4;\Omega} 
    &\leq \|\bzeta-\bzeta^\star\|_{0,2;\Omega} + \|\psi-\psi^\star\|_{0,4;\Omega} +  \|\bzeta^\star-\bzeta_h\|_{\bH} + \|\psi^\star-\psi_h\|_{0,4;\Omega}, 
  \end{align*}
  together with~\eqref{eq:aux-star-h} and, once more, Lemma~\ref{lem:quasiopt:primal}, concludes the proof.

\end{proof}

\section{Stenberg-type postprocessing}
\label{sec:stenbergL2}
In this section we provide a postprocessing technique that delivers an approximate solution to $
\psi^\star$ with asymptotic better accuracy than the approximate solution $\psi_h$. We follow 
\cite{fuhrer2024quasiinterpolatorsapplicationpostprocessingfinite} for the low-regular forcing case (see also, e.g.,   \cite{chen1995lp} for postprocessing of mixed formulations in $\rL^p$ spaces). 
{The analytical results rely on duality arguments and regularity properties of solutions to the dual problem.
For the latter we assume a regularity shift. In order to utilise the additional regularity we require an additional non-standard approximation property for the dual of the projection operator $\cQ$.
In subsection~\ref{sec:constructionQ} we show how to construct such an operator.
Then, we present duality results, the adapted Stenberg postprocessing scheme, and accuracy enhancement results in Subsection~\ref{sec:adjointAndStenberg}.
}

\subsection{Construction of $\cQ$}\label{sec:constructionQ}

We construct $\cQ\colon\rH^{-1}_D(\Omega)\to\rP_k(\cT_h)$ satisfying~\eqref{eq:prop0-Q} for any $k\geq0$, as the adjoint (with respect to the $\rL^2(\Omega)$ pairing) of an operator $\cQ'$ that combines a weighted Cl\'ement quasi-interpolator $J_h$ with element bubbles. For $k=0$ it coincides with the operator of~\cite{fuhrer2024mixed}. In this section, primes denote adjoint operators.

%\medskip\noindent\textit{Weighted Cl\'ement quasi-interpolator.} 
First, we recall from \cite{fuhrer2024mixed,fuhrer2024quasiinterpolatorsapplicationpostprocessingfinite} the weighted Cl\'ement quasi-interpolator $J_h\colon \rL^1(\Omega) \to S_h:=\mathrm{P}_1(\mathcal{T}_h) \cap \rH_D^1({\Omega})$. 
Let $V_K$ be the set of vertices of $K\in\cT_h$, $V_0$ the set of interior vertices, $V_N$ the set of vertices on $\partial\Omega\setminus\overline{\Gamma_D}$, and $V_{\mathrm{free}}:=V_0\cup V_N$. We assume that each $z\in V_N$ shares an element with some $z_0\in V_0$. The hat functions $\eta_z$, $z\in V_{\mathrm{free}}$, form a basis of $S_h$. For $z\in V_0$, $\mathcal{T}_z$ is the patch of elements sharing $z$, and for $z\in V_N$ we set $\mathcal{T}_z := \mathcal{T}_{z_0}$; $\omega_z$ denotes the corresponding domain. 
For $z\in V_{\mathrm{free}}$ we define the weight $\psi_z\in\rP_0(\cT_h)$ by
\[ \psi_z|_K :=
\begin{cases}
  \tfrac{\alpha_{z,K}}{|K|}  & K \in {\mathcal{T}_z}, \\
0 & \text{otherwise},
\end{cases}
\qquad\text{where}\qquad
\sum_{K\in \mathcal{T}_z}\alpha_{z,K}=1\quad\text{and}\quad\sum_{K\in \mathcal{T}_z}\alpha_{z,K}\,s_K=z,
\] 
and $s_K$ is the centre of mass of $K$. Such coefficients exist since the centres of mass in $\cT_z$ affinely span $\bbR^d$, and they can be chosen with $|\alpha_{z,K}|\lesssim1$, so that $\|\psi_z\|_{0,\infty;\Omega} \lesssim |\omega_z|^{-1}$ \cite[Remark~10]{fuhrer2024mixed}. Then
\[ J_h v := \sum_{z\in V_{\mathrm{free}}} \Bigl(\int_\Omega v\,\psi_z\Bigr)\eta_z \qquad \text{for } v \in \rL^1(\Omega).\]
If $p$ is affine on $\omega_z$, then $\int_\Omega p\,\psi_z = \sum_{K\in\cT_z}\alpha_{z,K}\,p(s_K) = p(z)$, i.e., $J_h$ reproduces nodal values of affine functions.

%\medskip\noindent\textit{Bubble operator.} 
Next, let $\eta_K := \prod_{z\in V_K}\lambda_{K,z}$ be the element bubble ($\lambda_{K,z}$ being the barycentric coordinates of $K$), extended by zero outside $K$. Let $N_k:=\dim\rP_k(K)$, and let $\{q_{K,j}\}_{j=1}^{N_k}$ be a basis of $\rP_k(K)$ obtained by mapping a fixed basis on a reference element $\widehat K$ (and extended by zero). Since $(p,q)\mapsto\int_K\eta_K\,p\,q$ is an inner product on $\rP_k(K)$, there is a unique $\{\chi_{K,j}\}_{j=1}^{N_k}\subset\eta_K\rP_k(K)\subset\rH^1_0(K)$ such that
\begin{align*}
  \int_K \chi_{K,i}\,q_{K,j} = \delta_{ij}, \qquad i,j=1,\dots,N_k.
\end{align*}
We define
\[B_hv := \sum_{K\in \mathcal{T}_h}\sum_{j=1}^{N_k}\Bigl(\int_K v\,q_{K,j}\Bigr)\chi_{K,j} \qquad \text{for } v\in\rL^1(\Omega).\]
Then $B_hv\in\rH^1_0(\Omega)$ and
\begin{align}\label{eq:Bh-moments}
  \int_K (B_h v)\,q = \int_K v\,q \qquad\forall q\in\rP_k(K),\ \forall K\in\cT_h.
\end{align}
For $k=0$, $\chi_{K,1}=\eta_K/\int_K\eta_K$, as in~\cite{fuhrer2024mixed}.

%\medskip\noindent\textit{Definition of $\cQ$.} 
With the ingredients above, we set
\begin{align}\label{eq:Qprime}
  \cQ' := J_h + B_h(1-J_h)\colon \rL^2(\Omega)\to S_h+\sum_{K\in\cT_h}\eta_K\rP_k(K)\subset\rH^1_D(\Omega),
\end{align}
and, for $\phi\in\rH^{-1}_D(\Omega)$,
\begin{align}\label{eq:defQ}
  \beta_h(\phi) := \sum_{K\in\cT_h}\sum_{j=1}^{N_k}\langle\phi,\chi_{K,j}\rangle\,q_{K,j}\in\rP_k(\cT_h),
  \qquad
  \cQ\phi := \beta_h(\phi) + \sum_{z\in V_{\mathrm{free}}}\langle\phi-\beta_h(\phi),\eta_z\rangle\,\psi_z,
\end{align}
where $\langle\bullet,\bullet\rangle$ is the duality pairing between $\rH^{-1}_D(\Omega)$ and $\rH^1_D(\Omega)$ (note that $\chi_{K,j},\eta_z\in\rH^1_D(\Omega)$). Since $\beta_h(\phi)\in\rP_k(\cT_h)$ and $\psi_z\in\rP_0(\cT_h)$, the range of $\cQ$ is contained in $\rP_k(\cT_h)$ for every $k\geq0$. Moreover, $\cQ$ is the adjoint of $\cQ'$:
\begin{align}\label{eq:Q-adjoint}
  \int_\Omega (\cQ\phi)\, v = \langle\phi,\cQ'v\rangle \qquad\forall\phi\in\rH^{-1}_D(\Omega),\ \forall v\in\rL^2(\Omega).
\end{align}
This follows from $\langle\phi,J_hv\rangle = \int_\Omega v\sum_{z}\langle\phi,\eta_z\rangle\psi_z$, $\langle\phi,B_hw\rangle = \int_\Omega w\,\beta_h(\phi)$, and $\langle\phi,B_hJ_hv\rangle = \int_\Omega v\sum_z\bigl(\int_\Omega\beta_h(\phi)\eta_z\bigr)\psi_z$. In other words, with $J_h'\phi = \sum_{z}\langle\phi,\eta_z\rangle\psi_z$ and $B_h'\phi = \beta_h(\phi)$, we have $\cQ = J_h' + (1-J_h)'B_h'$.

Only three properties of $J_h$ are used below: its weights $\psi_z$ belong to $\rP_k(\cT_h)$ (so that $\cQ$ maps into $\rP_k(\cT_h)$), it is locally $\rL^2$-stable, and it reproduces affine functions locally. For $k=0$ the first property requires piecewise constant weights, which is why the weighted Cl\'ement operator is used. For $k\geq1$, one can also take, e.g., the Scott--Zhang operator with dual basis functions supported on single elements~\cite{ern04}.

The proof of the following result follows~\cite[Theorems~11 and~12]{fuhrer2024mixed} and~\cite[Theorem~18]{fuhrer2024quasiinterpolatorsapplicationpostprocessingfinite}. For $K\in\cT_h$, let $\Omega_K$ be the union of $K$ and the domains $\omega_z$, $z\in V_K\cap V_{\mathrm{free}}$; these sets have finite overlap.

\begin{lemma}[Properties of $\cQ$]\label{lem:prop-Q}

  The operator $\cQ$ defined in~\eqref{eq:defQ} satisfies~\eqref{eq:prop0-Q}, and $\cQ'$ satisfies
  \begin{align*}
    \|(1-\cQ')\phi\|_{0,2;\Omega} &\lesssim h \|\nabla \phi\|_{0,2;\Omega} \quad  \forall \phi \in \rH^1_D(\Omega),
    \\
    \|(1-\cQ')\phi\|_{1,2;\Omega} &\lesssim h^s \|\phi\|_{1+s,2;\Omega} \quad  \forall \phi \in \rH^{1+s}(\Omega)\cap \rH^1_D(\Omega),\  s\in[0,1].
  \end{align*} 

\end{lemma}
\begin{proof} 

  %\textit{Step 1 (projection property).} 
First, we state a projection property.  If $q_h|_K=\sum_{j}c_{K,j}q_{K,j}$, then $\langle q_h,\chi_{K,i}\rangle=c_{K,i}$, so $\beta_h(q_h)=q_h$ and~\eqref{eq:defQ} gives $\cQ q_h=q_h$.

%  \smallskip\noindent\textit{Step 2 (local stability of $J_h$ and $B_h$).} 
Secondly, we address the local stability of $J_h$ and $B_h$.  Since $|\int_\Omega v\psi_z|\lesssim|\omega_z|^{-1/2}\|v\|_{0,2;\omega_z}$, $\|\eta_z\|_{0,2;K}\leq|K|^{1/2}$ and $|K|\simeq|\omega_z|$, we have 
  \begin{align*}
    \|J_h v\|_{0,2;K} \lesssim \|v\|_{0,2;\Omega_K}.
  \end{align*}
  A scaling argument gives $\|q_{K,j}\|_{0,2;K}\simeq|K|^{1/2}$ and $\|\chi_{K,j}\|_{0,2;K}\simeq|K|^{-1/2}$, and therefore
  \begin{align*}
    \|B_h v\|_{0,2;K} \leq \sum_{j=1}^{N_k}\|v\|_{0,2;K}\|q_{K,j}\|_{0,2;K}\|\chi_{K,j}\|_{0,2;K}\lesssim\|v\|_{0,2;K},
    \qquad
    \|\nabla B_h v\|_{0,2;K}\lesssim h_K^{-1}\|B_hv\|_{0,2;K},
  \end{align*}
  using an inverse inequality for the second bound.

  %\smallskip\noindent\textit{Step 3 (approximation properties of $J_h$).} 
  By the stability of $J_h$, the reproduction of affine functions and the arguments of~\cite[Theorem~11]{fuhrer2024mixed}, we obtain 
  \begin{align*}
    \|(1-J_h)v\|_{0,2;K} + h_K\|\nabla(1-J_h)v\|_{0,2;K} &\lesssim h_K \|\nabla v\|_{0,2;\Omega_K} && \forall v\in \rH^1_D(\Omega), 
    \\
    \|(1-J_h)v\|_{0,2;K} + h_K\|\nabla(1-J_h)v\|_{0,2;K} &\lesssim h_K^2 |v|_{2,2;\Omega_K} && \forall v\in \rH^2(\Omega)\cap \rH^1_D(\Omega).
  \end{align*}

  %\smallskip\noindent\textit{Step 4 (approximation properties of $\cQ'$).} 
  Since $1-\cQ' = (1-B_h)(1-J_h)$, the local stability of $J_h$ and $B_h$ implies
  \begin{align*}
    \|(1-\cQ')v\|_{0,2;K} &\lesssim \|(1-J_h)v\|_{0,2;K},
    \\
    \|\nabla(1-\cQ')v\|_{0,2;K} &\lesssim \|\nabla(1-J_h)v\|_{0,2;K} + h_K^{-1}\|(1-J_h)v\|_{0,2;K}.
  \end{align*}
  With the approximation property of $J_h$ and the finite overlap of the sets $\Omega_K$, we obtain
  \begin{align*}
    \|(1-\cQ')v\|_{0,2;\Omega} \lesssim h\|\nabla v\|_{0,2;\Omega},\qquad \|(1-\cQ')v\|_{1,2;\Omega} \lesssim \|v\|_{1,2;\Omega}\qquad \forall v\in\rH^1_D(\Omega),
  \end{align*}
  and $\|(1-\cQ')v\|_{1,2;\Omega}\lesssim h\|v\|_{2,2;\Omega}$ for all $v\in\rH^2(\Omega)\cap\rH^1_D(\Omega)$. The case $s\in(0,1)$ follows by interpolation.

  %\smallskip\noindent\textit{Step 5 (boundedness of $\cQ$).} 
  By~\eqref{eq:Q-adjoint} and the approximation properties of $\cQ'$, 
  \begin{align*}
    \|\cQ\phi\|_{\rH^{-1}_D(\Omega)} = \sup_{0\neq v\in\rH^1_D(\Omega)}\frac{\int_\Omega(\cQ\phi)v}{\|v\|_{1,2;\Omega}} = \sup_{0\neq v\in\rH^1_D(\Omega)}\frac{\langle\phi,\cQ'v\rangle}{\|v\|_{1,2;\Omega}} \leq \|\cQ'\|_{\rH^1_D(\Omega)\to\rH^1_D(\Omega)}\|\phi\|_{\rH^{-1}_D(\Omega)}.
  \end{align*}
  Similarly, \eqref{eq:Q-adjoint} with $v\in\rL^2(\Omega)$ and properties of $J_h$ give $\|\cQ\phi\|_{0,2;\Omega}\lesssim\|\phi\|_{0,2;\Omega}$ for $\phi\in\rL^2(\Omega)$.

\end{proof}

\subsection{Adjoint problem and accuracy enhancement}\label{sec:adjointAndStenberg}
Let us first develop Aubin--Nitsche-type estimates. 
To that end, consider the adjoint problem
\begin{subequations} \label{eq:PBstrong:adjoint}
    \begin{align}
        \kappa\phi  - \vdiv(\varepsilon\nabla \phi)-\bu_h\cdot\nabla\phi &= f && \text{in } \Omega,\\
        {\phi} &= 0 && \text{on } \Gamma_D, \\
        \varepsilon\nabla \phi \cdot \bn &= 0 && \text{on } \Gamma_N.
    \end{align}
\end{subequations}
{Under~\eqref{eq:assumption:uh}, its weak formulation in $\rH^1_D(\Omega)$ is uniquely solvable for each $f\in\rL^2(\Omega)$, as in Proposition~\ref{prop:primal}.}

In this section we assume that $\kappa$, $\varepsilon$, $\bu_h$, and $\Omega$ are such that there exist $0<s_\Omega\leq 1$ {and} $C_\Omega>0${, independent of $h$,} such that for any $f\in \rL^2(\Omega)$, the unique weak solution of~\eqref{eq:PBstrong:adjoint} satisfies
\begin{align}\label{eq:regularity:adjoint}
  {\|\varepsilon\nabla \phi\|_{s_\Omega,2;\Omega}} \lesssim \|\phi\|_{1+s_\Omega,2;\Omega} \leq C_\Omega \|f\|_{0,2;\Omega}.
\end{align}

Assumption~\eqref{eq:regularity:adjoint} holds, for instance, if $\varepsilon\in\rW^{1,\infty}(\Omega)$, $\kappa\in\rL^\infty(\Omega)$, and $\|\bu_h\|_{0,\infty;\Omega}$ is bounded uniformly in $h$ (e.g., if $\bu_h$ is the elementwise $\rL^2$-projection of $\bu\in\bL^\infty(\Omega)$). Then $-\Delta\phi=\varepsilon^{-1}(f-\kappa\phi+\bu_h\cdot\nabla\phi+\nabla\varepsilon\cdot\nabla\phi)$ with right-hand side bounded in $\rL^2(\Omega)$ by $C\|f\|_{0,2;\Omega}$, so that $s_\Omega$ is the regularity index of the Laplacian with the same boundary conditions~\cite{grisvard85}. For example, $s_\Omega=1$ for convex domains with pure Dirichlet conditions (Examples~2 and~3), whereas $s_\Omega<1/2$ in general if Dirichlet and Neumann conditions meet at an angle $\pi$.

\begin{lemma}\label{lem:adjoint:continuous}
  {Assume that~\eqref{eq:assumption:u}, \eqref{eq:assumption:uh} and~\eqref{eq:regularity:adjoint} hold, and} let $\psi$ and $\psi^\star$ denote the unique solutions of Problems~\eqref{eq:primal:weak} and~\eqref{eq:aux:weak}, respectively.
  Then, there exist constants $C_1>0$, $C_2>0${, independent of $h$,} such that
  \begin{align*}
    \|\psi-\psi^\star\|_{0,2;\Omega} &\leq C_1(\|g\|_{\rH^{-1}_D(\Omega)}+\|\psi_D\|_{1/2,00;\Gamma_D})\|\bu-\bu_h\|_{0,4;\Omega}
  \\
  &\qquad + C_2 h^{s_\Omega}\biggl(\min_{\btau_h\in\bH_h}\|\bzeta-\btau_h\|_{0,2;\Omega} + \min_{\chi_h\in\rQ_h}\|\kappa\psi-\chi_h\|_{\rH^{-1}_D(\Omega)}\biggr),
  \end{align*}
  where $\bzeta = \varepsilon\nabla \psi - \bu \psi $.
\end{lemma}
\begin{proof}

  We write $\psi-\psi^\star = (\psi-\widetilde\psi^\star)+(\widetilde\psi^\star-\psi^\star)$, where $\widetilde\psi^\star$ is defined as in the proof of Lemma~\ref{lem:quasiopt:primal}. By~\eqref{eq:aux-w}, $\|\psi-\widetilde\psi^\star\|_{0,2;\Omega}\leq\|\psi-\widetilde\psi^\star\|_{1,2;\Omega}$ is bounded by the first term on the right-hand side of the assertion.

  Let $w^\star:=\widetilde\psi^\star-\psi^\star\in\rH^1_D(\Omega)$ and let $\phi\in \rH_D^1(\Omega)$ be the solution of the adjoint problem~\eqref{eq:PBstrong:adjoint} with $f = w^\star$. Testing the weak form of~\eqref{eq:PBstrong:adjoint} with $w^\star$, and the equation satisfied by $w^\star$ (see the perturbation of the load in the proof of Lemma~\ref{lem:quasiopt:primal}) with $\phi$, we obtain
  \begin{align*}
    \|w^\star\|_{0,2;\Omega}^2 = \int_\Omega\bigl(\kappa w^\star\phi+\varepsilon\nabla w^\star\cdot\nabla\phi-\bu_hw^\star\cdot\nabla\phi\bigr) = \langle (1-\cQ)g, \phi\rangle.
  \end{align*}
  By~\eqref{eq:Q-adjoint} and $\cQ(1-\cQ)=0$, we have $\langle(1-\cQ)g,\cQ'\phi\rangle=\int_\Omega\cQ(1-\cQ)g\,\phi=0$, so that, using Lemma~\ref{lem:prop-Q} and~\eqref{eq:regularity:adjoint},
  \begin{align*}
    \|w^\star\|_{0,2;\Omega}^2 &= \langle (1-\cQ)g, (1-\cQ')\phi\rangle \leq \|(1-\cQ)g\|_{\rH_D^{-1}(\Omega)} \|(1-\cQ')\phi\|_{1,2;\Omega}
    \\ &\lesssim h^{s_\Omega}\|(1-\cQ)g\|_{\rH_D^{-1}(\Omega)}\|\phi\|_{1+s_\Omega,2;\Omega}
    \lesssim h^{s_\Omega}\|(1-\cQ)g\|_{\rH_D^{-1}(\Omega)}\|w^\star\|_{0,2;\Omega}.
  \end{align*}
  The proof is concluded by bounding $\|(1-\cQ)g\|_{\rH^{-1}_D(\Omega)}$ with~\eqref{eq:aux-Qg} and using $h^{s_\Omega}\lesssim 1$.

\end{proof}

In addition to the assumption on the regularity shift for the adjoint problem~\eqref{eq:PBstrong:adjoint}, we suppose that the solution to the adjoint problem enjoys the approximation estimate
\begin{align}\label{eq:regularity:adjoint:approx}
  \|(1-\mP_h^0)(\kappa\phi-\bu_h\cdot\nabla\phi)\|_{0,4/3;\Omega} \leq C h^{s_\Omega} \|\phi\|_{1+s_\Omega,2;\Omega}.
\end{align}
{Sufficient conditions for~\eqref{eq:regularity:adjoint:approx} are given next.}

\begin{lemma}\label{lem:approx-assumption}

  Assume that $\kappa\in\rW^{1,\infty}(\Omega)$ and that $\bu_h\in\bM_h\subseteq[\rP_m(\cT_h)]^d$ satisfies $\|\bu_h\|_{0,4;\Omega}+\|\nabla_h\bu_h\|_{0,4;\Omega}\leq C_{\bu}$ uniformly in $h$, with $\nabla_h$ the elementwise gradient (e.g., $\bu_h$ the $\rL^2$-projection of $\bu\in\bW^{1,4}(\Omega)$). Then, for all $\phi\in\rH^{1+s}(\Omega)$ with $s\in[0,1]$,
  \[
    \|(1-\mP_h^0)(\kappa\phi-\bu_h\cdot\nabla\phi)\|_{0,4/3;\Omega} \lesssim (\|\kappa\|_{1,\infty;\Omega}+C_{\bu})\,h^{s}\|\phi\|_{1+s,2;\Omega}.
  \]
  In particular, \eqref{eq:regularity:adjoint:approx} holds under~\eqref{eq:regularity:adjoint}.

\end{lemma}
\begin{proof}

  The Poincar\'e inequality on each element gives $\|(1-\mP_h^0)(\kappa\phi)\|_{0,4/3;\Omega}\lesssim h\|\kappa\|_{1,\infty;\Omega}\|\phi\|_{1,2;\Omega}$. For the advective term we write
  \[
    (1-\mP_h^0)(\bu_h\cdot\nabla\phi) = (1-\mP_h^0)\bigl((\bu_h-\mP_h^0\bu_h)\cdot\nabla\phi\bigr) + (\mP_h^0\bu_h)\cdot(1-\mP_h^0)\nabla\phi.
  \]
  By Lemma~\ref{lem:boundednessL2projector}, H\"older's and Poincar\'e's inequalities, the $\rL^{4/3}(\Omega)$ norm of the first term is bounded by $h\,C_{\bu}\|\phi\|_{1,2;\Omega}$, and that of the second term by $C_{\bu}\|(1-\mP_h^0)\nabla\phi\|_{0,2;\Omega}\lesssim C_{\bu}h^s\|\phi\|_{1+s,2;\Omega}$.

\end{proof}

\begin{lemma}\label{lem:supercloseness}
  {Assume that the hypotheses of Theorem~\ref{thm:errorestimate} hold, together with~\eqref{eq:regularity:adjoint} and~\eqref{eq:regularity:adjoint:approx}.} Let $\psi^\star$ denote the unique solution of Problem~\eqref{eq:aux:weak} and let 
  $(\bzeta_h,{\psi_h})\in\bH_h\times\rQ_h$ denote the discrete solution of~\eqref{eq:weak-h}.
  Then, there exists {a} constant $C>0${, independent of $h$,} such that
  \begin{align*}
    \|\mP_h^0(\psi^\star-\psi_h)\|_{0,2;\Omega} &\leq C h^{s_\Omega}
    \Big(\min_{{(\btau_h,v_h)}\in\bH_h\times \rQ_h} (\|\psi-v_h\|_{0,4;\Omega} + \|\bzeta-\btau_h\|_{0,2;\Omega})
    \\
    &\qquad+(\|g\|_{\rH^{-1}_D(\Omega)}+\|\psi_D\|_{1/2,00;\Gamma_D})\|\bu-\bu_h\|_{0,4;\Omega} + \min_{\chi_h\in\rQ_h}\|\kappa\psi-\chi_h\|_{0,4/3;\Omega}\Big),
  \end{align*}
  where $\bzeta = \varepsilon\nabla \psi - \bu \psi $.
\end{lemma}
\begin{proof}

  Let $\be_{\bzeta}:=\bzeta^\star-\bzeta_h$, $e_\psi:=\psi^\star-\psi_h$, and let $\phi\in \rH_D^1(\Omega)$ denote the unique solution of~\eqref{eq:PBstrong:adjoint} with 
  $f := \mP_h^0 e_\psi$.
  Then $\br := -\varepsilon\nabla \phi\in\bH$, with $\vdiv\br = f-\kappa\phi+\bu_h\cdot\nabla\phi$, and
  \begin{align*}
    A_{\bu_h}\bigl((\be_{\bzeta},e_\psi),(\br,-\phi)\bigr) 
    &= \int_\Omega\frac1\varepsilon\be_{\bzeta}\cdot\br + \int_\Omega e_\psi\vdiv\br + \int_\Omega\frac1\varepsilon(\bu_h\cdot\br)e_\psi  - \int_\Omega\phi\vdiv\be_{\bzeta} + \int_\Omega\kappa e_\psi\phi
    \\
    &= -\int_\Omega\bigl(\be_{\bzeta}\cdot\nabla\phi+\phi\vdiv\be_{\bzeta}\bigr) + \int_\Omega e_\psi\bigl(\vdiv\br-\bu_h\cdot\nabla\phi+\kappa\phi\bigr)
    \\ &= \int_\Omega e_\psi f = \|f\|_{0,2;\Omega}^2,
  \end{align*}
  since the first integral vanishes ($\be_{\bzeta}\in\bH$, $\phi\in\rH^1_D(\Omega)$). Galerkin orthogonality and the continuity of $A_{\bu_h}(\bullet,\bullet)$ yield, for all $(\br_h,\phi_h)\in \bH_h\times \rQ_h$,
  \begin{align*}
    \|f\|_{0,2;\Omega}^2 = A_{\bu_h}\bigl((\be_{\bzeta},e_\psi),(\br-\br_h,-(\phi-\phi_h))\bigr)
    \lesssim \|(\be_{\bzeta},e_\psi)\|_{\bH\times\rQ}
    \|(\br-\br_h,\phi-\phi_h)\|_{\bH\times \rQ}.
  \end{align*}
  The first factor is bounded by Theorem~\ref{thm:errorestimate}. We choose $\br_h = \Pi^{\mathsf{RT}}_{\rqu}\br$ and $\phi_h = \mP_h^0\phi$. Approximation properties of $\Pi^{\mathsf{RT}}_{\rqu}$~\cite{ern2016mollification,Ern18} and $\mP_h^0$, together with~\eqref{eq:regularity:adjoint}, give 
  \begin{align*}
    \|\br-\br_h\|_{0,2;\Omega} + \|\phi-\phi_h\|_{0,4;\Omega} \lesssim h^{s_\Omega} \|\phi\|_{1+s_\Omega,2;\Omega} \lesssim h^{s_\Omega} \|f\|_{0,2;\Omega}.
  \end{align*}
  Since $f\in\rP_0(\cT_h)$, \eqref{eq:commuting} gives 
  $$\vdiv(\br-\br_h) = (1-\mP_h^k)\vdiv\br = (1-\mP_h^k)(\bu_h\cdot\nabla \phi-\kappa\phi) = -(1-\mP_h^k)(1-\mP_h^0)(\kappa\phi-\bu_h\cdot\nabla \phi).$$ 
  and Lemma~\ref{lem:boundednessL2projector}, \eqref{eq:regularity:adjoint:approx} and~\eqref{eq:regularity:adjoint} imply
  \begin{align*}
    \|\vdiv(\br-\br_h)\|_{0,4/3;\Omega} \lesssim \|(1-\mP_h^0)(\kappa\phi-\bu_h\cdot\nabla \phi)\|_{0,4/3;\Omega}\lesssim h^{s_\Omega}\|f\|_{0,2;\Omega}.
  \end{align*}
  This concludes the proof.

\end{proof}

Let us now consider a discrete auxiliary local problem on each $K \in \cT_h$ (which arises from the pseudo potential flux constitutive equation) similarly to the classical Stenberg postprocess \cite{stenberg1991postprocessing}, but adapted to our case with advecting velocity. We define {$\psi^\sharp_h\in\rP_{k+1}(\cT_h)$ elementwise, i.e., $\psi^\sharp_h|_K \in \mathrm{P}_{k+1}(K)$ is such that} 
\begin{equation}\label{eq:postprocess} \begin{split}
  \int_K \varepsilon \nabla \psi^\sharp_h \cdot \nabla v_h &= \int_K \bzeta_h\cdot \nabla v_h + \int_K(\bu_h\cdot\nabla v_h) \psi_h
\qquad \forall v_h \in \mathrm{P}_{k+1}(K),
\\
  {\int_K\psi^\sharp_h} &= {\int_K\psi_h}.  
\end{split}\end{equation}
{The first equation determines $\nabla\psi^\sharp_h|_K$, and the second one, equivalent to $\mP_h^0\psi^\sharp_h=\mP_h^0\psi_h$, fixes the constant.}

\begin{theorem}\label{thm:postprocess}
  Under the {assumptions of Lemma~\ref{lem:supercloseness}}, 
  \begin{align*}
    \|\psi-\psi^\sharp_h\|_{0,2;\Omega} &\leq C(\|g\|_{\rH^{-1}_D(\Omega)}+\|\psi_D\|_{1/2,00;\Gamma_D})\|\bu-\bu_h\|_{0,4;\Omega}
    \\
    &+C h^{s_\Omega}
    \Big(\min_{{(\btau_h,v_h)}\in\bH_h\times \rQ_h} (\|\psi-v_h\|_{0,4;\Omega} + \|\bzeta-\btau_h\|_{0,2;\Omega})
    \\
    &\qquad+(\|g\|_{\rH^{-1}_D(\Omega)}+\|\psi_D\|_{1/2,00;\Gamma_D})\|\bu-\bu_h\|_{0,4;\Omega} + \min_{\chi_h\in\rQ_h}\|\kappa\psi-\chi_h\|_{0,4/3;\Omega}\Big)
    \\
    &+Ch\min_{v_h\in \mathrm{P}_{k+1}(\cT_h)} \|{\nabla_h}(\psi-v_h)\|_{0,2;\Omega},
  \end{align*}
  {with $C>0$ independent of $h$, and where $\nabla_h$ denotes the elementwise gradient.}
\end{theorem}
\begin{proof}

  By the triangle inequality,
  \begin{align*}
    \|\psi-\psi^\sharp_h\|_{0,2;\Omega} \leq \|\psi-\psi^\star\|_{0,2;\Omega} + \|\mP_h^0(\psi^\star-\psi^\sharp_h)\|_{0,2;\Omega}
    + \|(1-\mP_h^0)(\psi^\star-\psi^\sharp_h)\|_{0,2;\Omega}.
  \end{align*}
  The first term is estimated with Lemma~\ref{lem:adjoint:continuous}, and the second one with Lemma~\ref{lem:supercloseness}, since $\mP_h^0\psi^\sharp_h = \mP_h^0\psi_h$. For the last term, $\|(1-\mP_h^0)(\psi^\star-\psi^\sharp_h)\|_{0,2;\Omega} \lesssim h \|\nabla_h(\psi^\star-\psi^\sharp_h)\|_{0,2;\Omega}$. 

  Let $\psi^\star_K\in\rP_{k+1}(K)$ be such that $\int_K\varepsilon\nabla(\psi^\star-\psi^\star_K)\cdot\nabla v_h=0$ for all $v_h\in\rP_{k+1}(K)$; then $\|\nabla(\psi^\star-\psi^\star_K)\|_{0,2;K}\lesssim\min_{v_h\in\rP_{k+1}(K)}\|\nabla(\psi^\star-v_h)\|_{0,2;K}$. Since $\varepsilon\nabla\psi^\star = \bzeta^\star+\bu_h\psi^\star$, \eqref{eq:postprocess} gives
  \[
    \int_K\varepsilon\nabla(\psi^\star_K-\psi^\sharp_h)\cdot\nabla v_h = \int_K(\bzeta^\star-\bzeta_h)\cdot\nabla v_h + \int_K(\bu_h\cdot\nabla v_h)(\psi^\star-\psi_h)\qquad\forall v_h\in\rP_{k+1}(K).
  \]
  Testing with $v_h=\psi^\star_K-\psi^\sharp_h|_K$, summing over $K\in\cT_h$ and using H\"older's inequality, we arrive at
  \begin{align*}
    \|\nabla_h(\psi^\star-\psi^\sharp_h)\|_{0,2;\Omega} &\lesssim \min_{v_h\in \mathrm{P}_{k+1}(\cT_h)} \|\nabla_h(\psi^\star-v_h)\|_{0,2;\Omega}
    + \|\bzeta^\star-\bzeta_h\|_{0,2;\Omega} + \|\bu_h\|_{0,4;\Omega}\|\psi^\star-\psi_h\|_{0,4;\Omega}
    \\
    &\lesssim \min_{v_h\in \mathrm{P}_{k+1}(\cT_h)} \|\nabla_h(\psi-v_h)\|_{0,2;\Omega} + \|\psi-\psi^\star\|_{1,2;\Omega}
     + \|\bzeta^\star-\bzeta_h\|_{\bH} + \|\psi^\star-\psi_h\|_{0,4;\Omega},
  \end{align*}
  The remaining terms are bounded by Lemma~\ref{lem:quasiopt:primal} and Theorem~\ref{thm:errorestimate}, and $h\lesssim h^{s_\Omega}$ concludes the proof.

\end{proof}

{Note that if $s_\Omega=1$, and $(\bzeta,\psi)$ {and $\kappa$ are} sufficiently smooth, as well as $\|\bu-\bu_h\|_{0,4;\Omega} = \mathcal{O}(h^{k+2})$, then, 
\begin{equation}\label{eq:ideal}
    \|\psi-\psi^\sharp_h\|_{0,2;\Omega} = \mathcal{O}(h^{k+2}),
\end{equation}
by the last theorem.}

If $\varepsilon$ is not piecewise polynomial, it is replaced in practice by an approximation $\varepsilon_h\in\rP_m(\cT_h)$ (or a quadrature rule is used). In~\eqref{eq:weak-h} this adds a consistency error of order $\|\varepsilon^{-1}-\varepsilon_h^{-1}\|_{0,\infty;\Omega}$ \cite[Section~4.1]{ciarlet78}. In~\eqref{eq:postprocess}, if $\widetilde\psi^\sharp_h$ denotes the postprocessed solution computed with $\varepsilon_h$, subtracting the local problems and using that $\widetilde\psi^\sharp_h-\psi^\sharp_h$ has zero mean on each element gives
\[
  \|\widetilde\psi^\sharp_h-\psi^\sharp_h\|_{0,2;\Omega}\lesssim h\|\varepsilon-\varepsilon_h\|_{0,\infty;\Omega}\|\nabla_h\psi^\sharp_h\|_{0,2;\Omega} = \mathcal{O}(h^{m+2}).
\]
Hence~\eqref{eq:ideal} is preserved if $m\geq k$; for $k=0$, a piecewise constant $\varepsilon_h$ suffices (see Example~2 below).

%%%%%%%%%%%%%%%%%%%%%%
\section{Numerical experiments}\label{sec:results}
We present simple computational tests that illustrate the properties of the proposed mixed finite element method as well as the postprocess mechanism. All examples, including the action of the map $\cQ$ and the local Stenberg-type postprocess, are implemented in the open source finite element library \texttt{FEniCS} \cite{alnaes15}.

\medskip
\noindent\textbf{Example 1.} First we assess the convergence of the finite element discretisation using manufactured solutions of different regularity on simple domains. We consider on $\Omega = (0,1)^2$ the smooth and singular  solutions to the primal governing PDE
\begin{equation}\label{eq:manuf}\psi_{\mathrm{ex}}(x,y) = \sin(\pi x)\sin(\pi y) \quad \text{and}\quad 
\psi_{\mathrm{ex}}(x,y) = |x-y|^{\frac34}\sin(\pi x)\sin(\pi y),\end{equation}
and use a smooth and divergence-free vector field as external velocity 
\[ \bu_{\mathrm{ex}}(x,y) = \begin{pmatrix} \cos(\pi x)\sin(\pi y) \\ -\sin(\pi x)\cos(\pi y)\end{pmatrix}.\]
An exact mixed variable $\bzeta_{\mathrm{ex}}$, an exact source $g_{\mathrm{ex}}$, and exact Dirichlet datum $\psi_D$ and (non-homogeneous) Neumann datum $\bzeta\cdot \bn = \zeta_N$ are manufactured from these exact primal solutions and considering unit parameters $\varepsilon = \kappa = 1$. {The non-homogeneous Neumann datum is imposed essentially in $\bH_h$.}  

A sequence of successively refined uniform meshes is constructed and on each refinement level we compute $\cQ g_{\mathrm{ex}}$ and with it we generate approximate solutions $(\bzeta_h,\psi_h) \in \bH_h\times \rQ_h$  with the mixed finite element method \eqref{eq:weak-h} (using lowest-order Raviart--Thomas spaces), and taking $\bu_h$ as the $\bL^2(\Omega)$ projection of $\bu$ onto {$\bM_h = [\rP_1(\cT_h)]^2$ (required  by the postprocess as indicated in \eqref{eq:ideal})}. With the obtained discrete pseudo potential flux, at each mesh refinement level we also compute a post-processed double layer potential $\psi^\sharp_h$ from the auxiliary problem \eqref{eq:postprocess}. 
Errors  between the discrete and exact solutions are defined in the following way 
\[e_{\vdiv_{4/3}}(\bzeta):= \|\bzeta_h-\bzeta_{\mathrm{ex}}\|_{\vdiv_{4/3}{;}\Omega}, \quad 
e_{0,4}(\psi):= \|\psi_h-\psi_{\mathrm{ex}}\|_{0,4;\Omega}, \quad e_{0}(\psi^\sharp):= \|\psi^\sharp_h-\psi_{\mathrm{ex}}\|_{0,2;\Omega},\]
{where $\|\bullet\|_{\vdiv_{4/3},\Omega}$ is the norm of $\bH$ (cf. Section~\ref{sec:cont}),} and we compute experimental order of convergence as usual 
\[\texttt{EoC}= \frac{\log[e_{i+1}(\bullet)/e_i(\bullet)]}{\log(h_{i+1}/h_i)}, \]
where $e_i(\bullet)$ and $h_i$ denote an individual error and mesh size associated with the $i$-th refinement level, respectively. We remark that for the second expression in \eqref{eq:manuf} the exact flux $\bzeta_{\mathrm{ex}}$ is not regular enough to compute the full $e_{\vdiv_{4/3}}(\bzeta)$ error, so in that case we only report $e_{0}(\bzeta):=\|\bzeta_h-\bzeta_{\mathrm{ex}}\|_{0,2;\Omega}$. %Similarly, for the rough case the postprocessed double layer potential is not in $\rL^4(\Omega)$, so we compute only $e_{0}(\psi^\sharp):= \|\psi^\sharp_h-\psi_{\mathrm{ex}}\|_{0,\Omega}$. 

The results from the convergence tests are reported in Table 
\ref{tab:convergence-mixed} for  mixed boundary conditions{---}the Neumann boundary is the right edge of the domain whereas the Dirichlet part is the remainder of the boundary. 
The first part of the table focuses on the smooth case, where we observe optimal {first-order} convergence for the flux and potential in the natural norms for the mixed regularised problem---we recall that we are only using here the lowest-order case---and we also see the asymptotic second-order convergence of the locally postprocessed solution, as predicted by Theorem~\ref{thm:postprocess}. 

\begin{table}[width=\linewidth,cols=8,pos=t!]
\caption{Example 1. Error history for the lowest-order mixed finite element scheme including regularisation through $\cQ$ and postprocessing, when subjected to a smooth load (top) and a rough load that is not in $\rL^2(\Omega)$ (bottom). Here we use mixed boundary conditions.}\label{tab:convergence-mixed}
\begin{tabular*}{\tblwidth}{@{\extracolsep{\fill}}rccccccc@{}}
\toprule
\multicolumn{8}{c}{$g\in C^\infty(\Omega)$}\\
\midrule
\texttt{DoFs} & $h$ & $e_{\vdiv_{4/3}}(\bzeta)$ & \texttt{EoC} & $e_{0,4}(\psi)$ & \texttt{EoC} & $e_{0}(\psi^\sharp)$ & \texttt{EoC} \\
\midrule
24 & 0.7071 & $5.58{\cdot}10^{0}$ & $\star$ & $3.08{\cdot}10^{-1}$ & $\star$ & $1.22{\cdot}10^{-1}$ & $\star$ \\
88 & 0.3536 & $2.85{\cdot}10^{0}$ & 0.968 & $1.73{\cdot}10^{-1}$ & 0.831 & $3.66{\cdot}10^{-2}$ & 1.730 \\
336 & 0.1768 & $1.42{\cdot}10^{0}$ & 1.004 & $8.71{\cdot}10^{-2}$ & 0.988 & $9.99{\cdot}10^{-3}$ & 1.874 \\
1312 & 0.0884 & $7.08{\cdot}10^{-1}$ & 1.007 & $4.36{\cdot}10^{-2}$ & 0.999 & $2.56{\cdot}10^{-3}$ & 1.967 \\
5184 & 0.0442 & $3.53{\cdot}10^{-1}$ & 1.003 & $2.18{\cdot}10^{-2}$ & 1.000 & $6.42{\cdot}10^{-4}$ & 1.992 \\
20608 & 0.0221 & $1.76{\cdot}10^{-1}$ & 1.001 & $1.09{\cdot}10^{-2}$ & 1.000 & $1.61{\cdot}10^{-4}$ & 1.998 \\
82176 & 0.0110 & $8.81{\cdot}10^{-2}$ & 1.001 & $5.45{\cdot}10^{-3}$ & 1.000 & $4.02{\cdot}10^{-5}$ & 2.000 \\
\midrule
\multicolumn{8}{c}{$g\in \rH^{-1}(\Omega)$}\\
\midrule
\texttt{DoFs} & $h$ & $e_{0}(\bzeta)$ & \texttt{EoC} & $e_{0,4}(\psi)$ & \texttt{EoC} & $e_{0}(\psi^\sharp)$ & \texttt{EoC} \\
\midrule
24 & 0.7071 & $7.82{\cdot}10^{-1}$ & $\star$ & $1.91{\cdot}10^{-1}$ & $\star$ & $1.54{\cdot}10^{-1}$ & $\star$ \\
88 & 0.3536 & $6.91{\cdot}10^{-1}$ & 0.178 & $9.21{\cdot}10^{-2}$ & 1.054 & $5.52{\cdot}10^{-2}$ & 1.478 \\
336 & 0.1768 & $5.04{\cdot}10^{-1}$ & 0.456 & $4.48{\cdot}10^{-2}$ & 1.039 & $1.79{\cdot}10^{-2}$ & 1.623 \\
1312 & 0.0884 & $3.93{\cdot}10^{-1}$ & 0.359 & $2.26{\cdot}10^{-2}$ & 0.990 & $5.79{\cdot}10^{-3}$ & 1.630 \\
5184 & 0.0442 & $3.20{\cdot}10^{-1}$ & 0.296 & $1.14{\cdot}10^{-2}$ & 0.981 & $2.16{\cdot}10^{-3}$ & 1.421 \\
20608 & 0.0221 & $2.66{\cdot}10^{-1}$ & 0.268 & $5.79{\cdot}10^{-3}$ & 0.981 & $9.24{\cdot}10^{-4}$ & 1.227 \\
82176 & 0.0110 & $2.22{\cdot}10^{-1}$ & 0.257 & $2.93{\cdot}10^{-3}$ & 0.982 & $4.20{\cdot}10^{-4}$ & 1.137 \\
\bottomrule
\end{tabular*}
\end{table}

The bottom part of {Table~\ref{tab:convergence-mixed}} shows only the $\bL^2$ contribution of the flux error, the potential and postprocessed potential. For these unknowns we see, respectively, approximate rates of convergence of $O(h^{1/4})$, $O(h)$ and $O(h^{5/4})$, which agree with the results from \cite[Section 5.2]{fuhrer2024mixed}. In Figure~\ref{fig:ex01} we show the discrete solutions obtained for both smooth and rough loads, indicating accurate approximations of flux and potential, as well as well-resolved postprocessed potential (with sharper gradients). 

\begin{figure}[pos=t!]
\centering
% labels are placed with \put(x,y) in overpic coordinates (percent of image width); adjust after regenerating the figures
\begin{overpic}[width=0.32\linewidth]{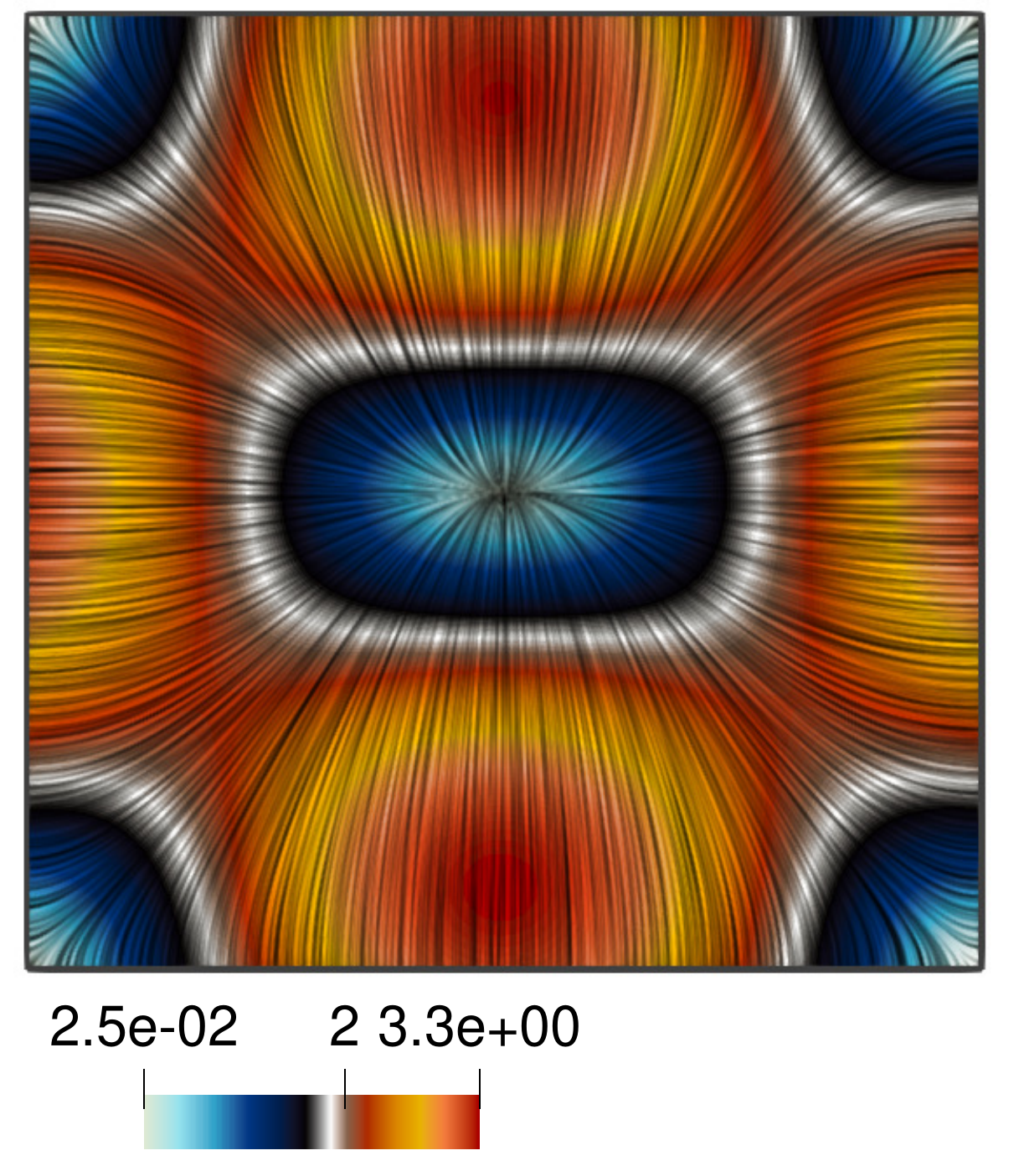}
  \put(54,3){\normalsize $|\bzeta_h|$}
\end{overpic}\hfill
\begin{overpic}[width=0.32\linewidth]{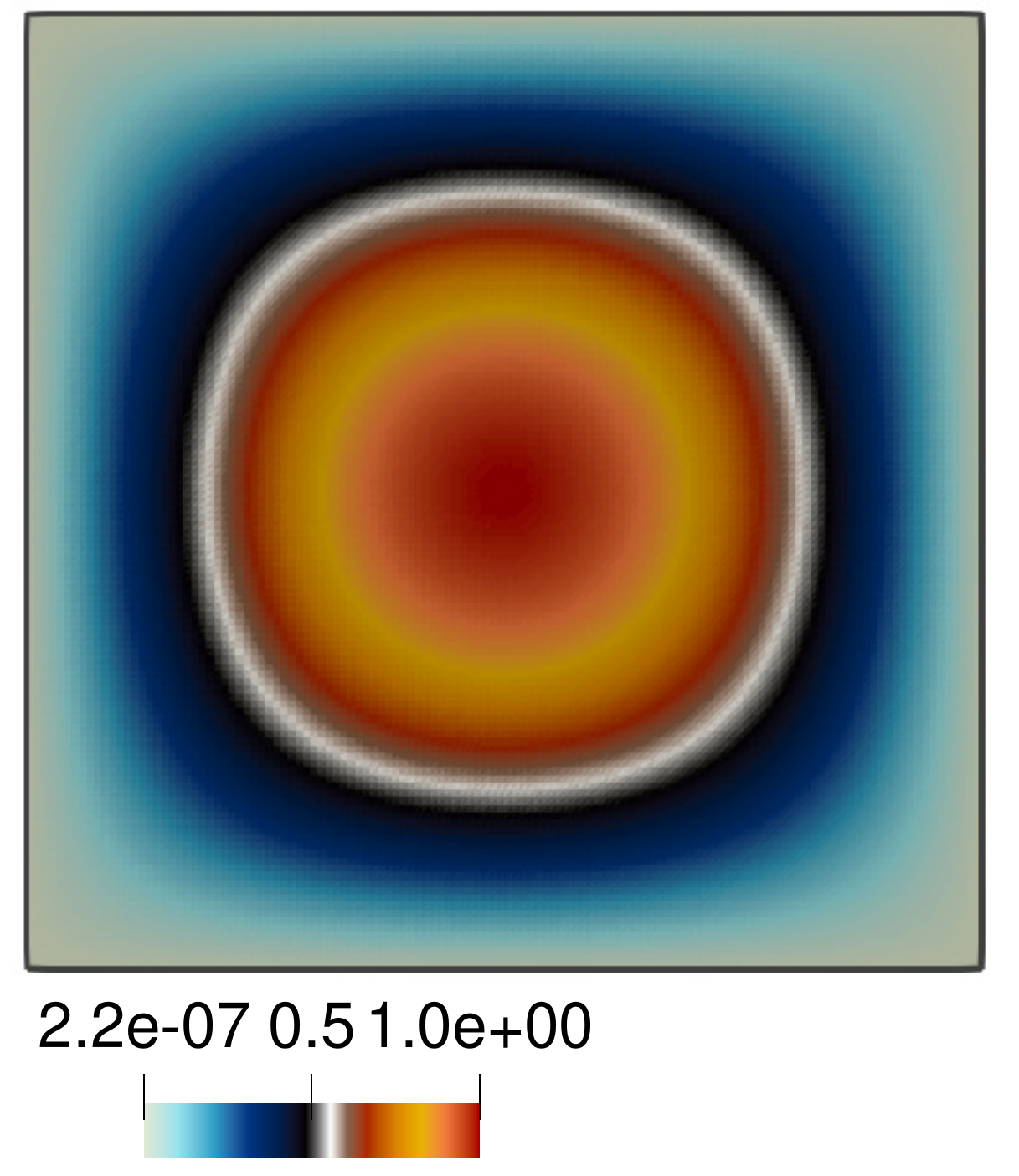}
  \put(54,3){\normalsize $\psi_h$}
\end{overpic}\hfill
\begin{overpic}[width=0.32\linewidth]{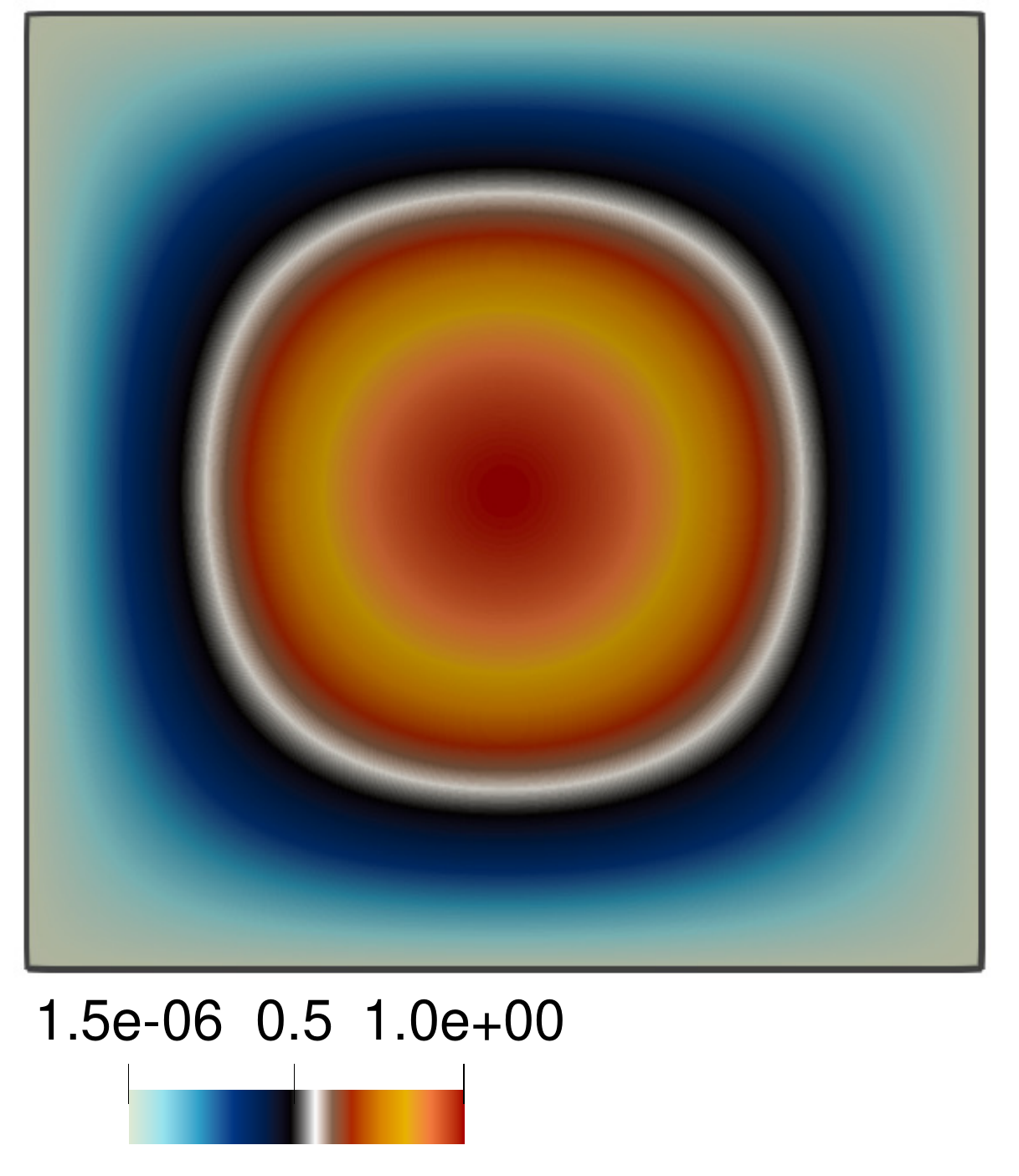}
  \put(54,3){\normalsize $\psi_h^\sharp$}
\end{overpic}\\[3ex]
\begin{overpic}[width=0.32\linewidth]{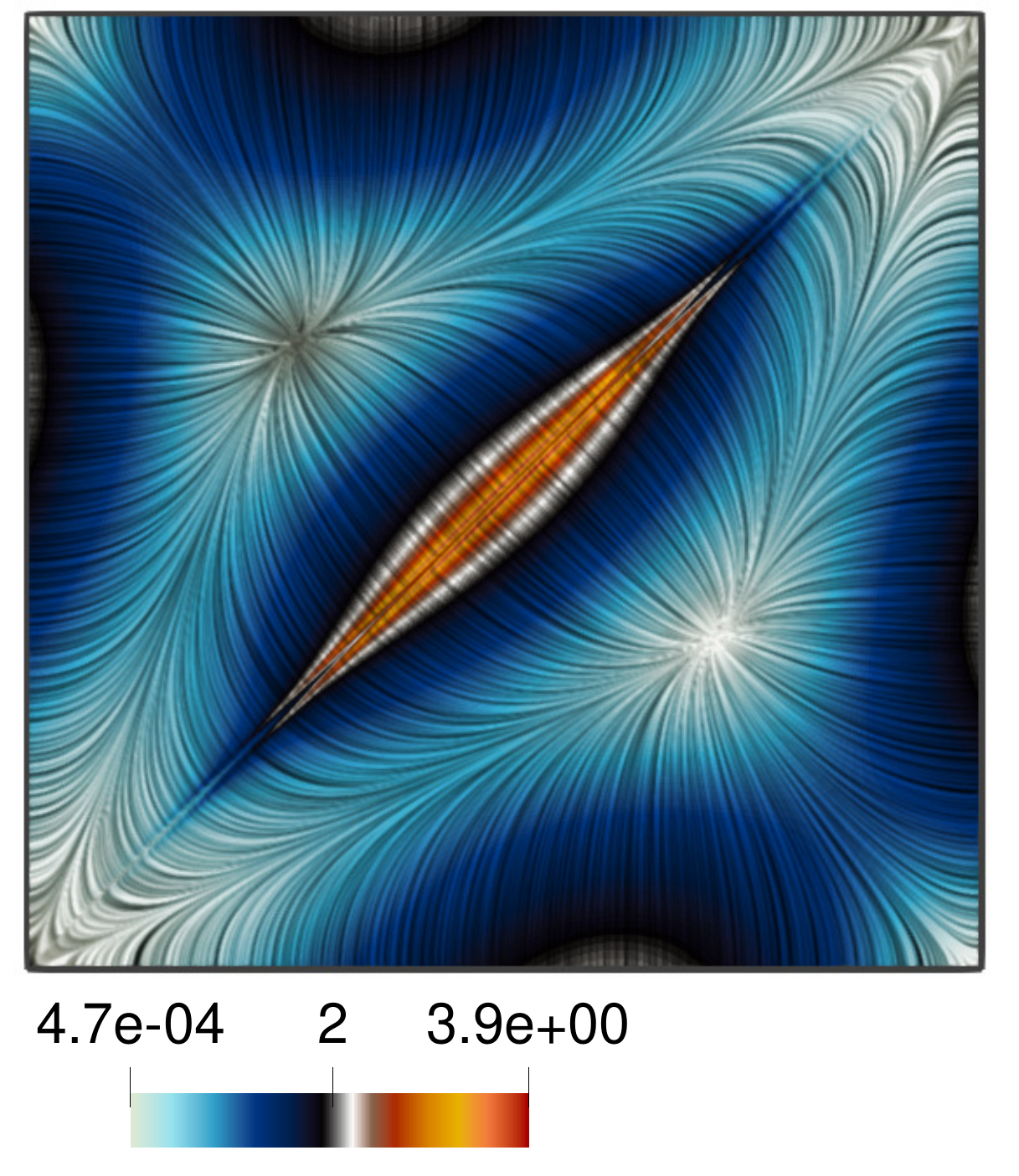}
  \put(54,3){\normalsize $|\bzeta_h|$}
\end{overpic}\hfill
\begin{overpic}[width=0.32\linewidth]{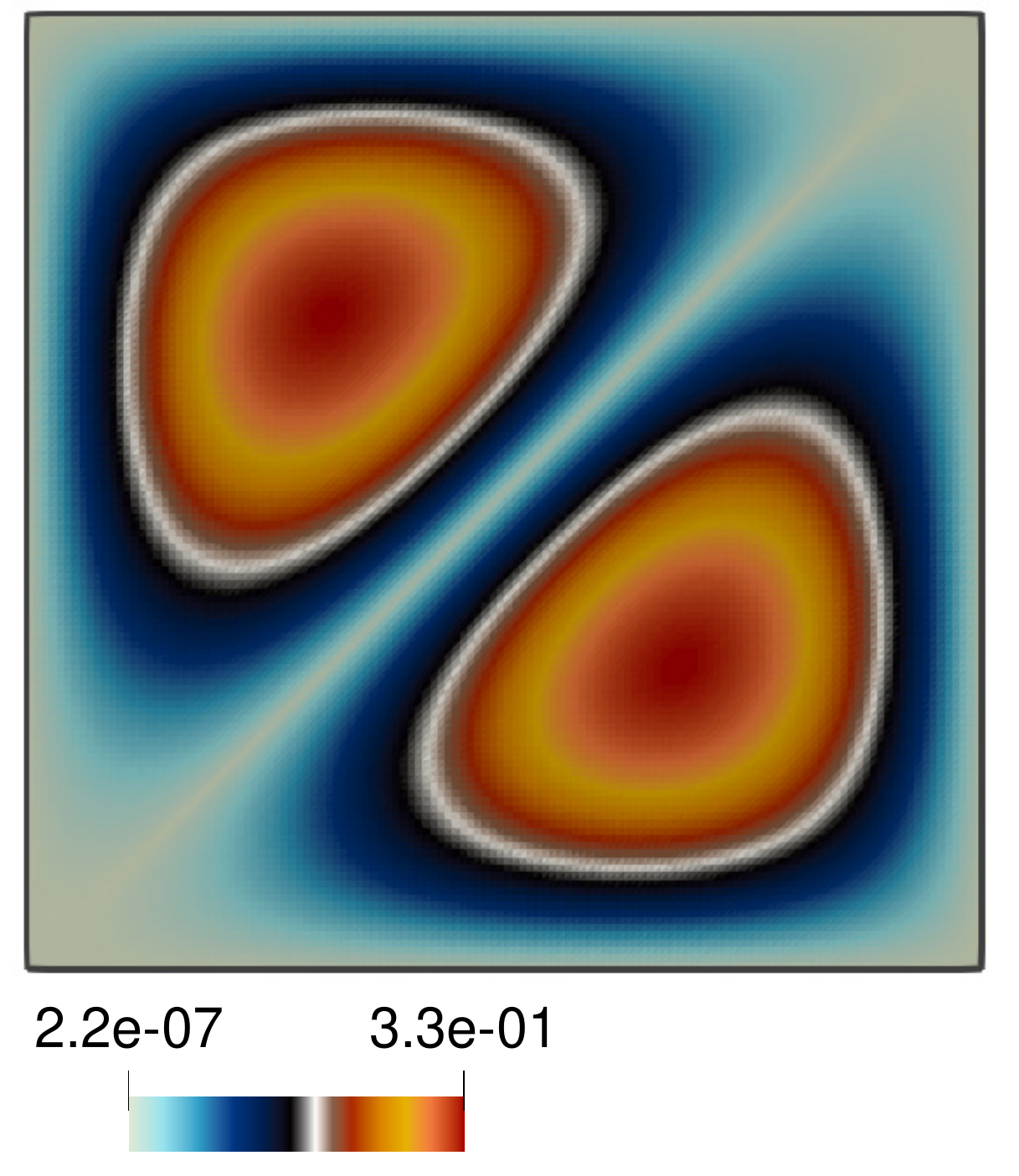}
  \put(54,3){\normalsize $\psi_h$}
\end{overpic}\hfill
\begin{overpic}[width=0.32\linewidth]{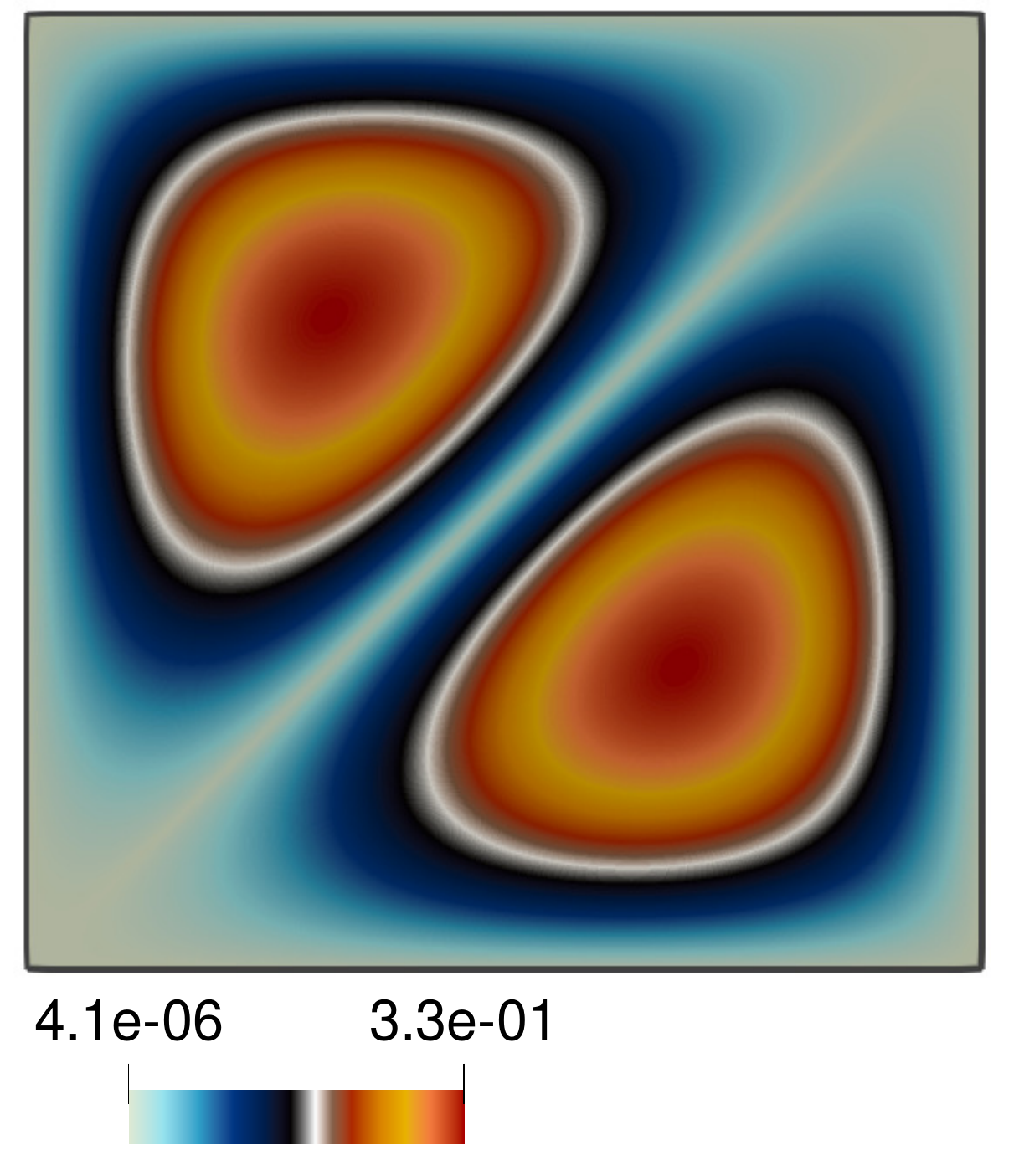}
  \put(54,3){\normalsize $\psi_h^\sharp$}
\end{overpic}

\vspace{4ex}
\caption{Example 1. Approximate solutions (line integral contours and magnitude for the pseudo potential flux and double layer potential profile) for the lowest-order mixed finite element scheme, and postprocessed potential. We show results using smooth and rough loadings (top and bottom rows, respectively).}
\label{fig:ex01}
\end{figure}

\medskip
\noindent\textbf{Example 2.} In this test we  examine the behaviour of the postprocess without the regularisation. For that we follow \cite[Section 6.3]{fuhrer2022minres} and use the domain $\Omega=(-1,1)^2$, {and take} the given advection velocity $\bu_h$ {as the $\bL^2(\Omega)$-projection onto $[\rP_1(\cT_h)]^2$} of \[ \bu_{\mathrm{ex}}(x,y) = \begin{pmatrix} \cos(\frac{\pi}{2} x)\sin(\frac{\pi}{2} y) \\ -\sin(\frac{\pi}{2} x)\cos(\frac{\pi}{2} y)\end{pmatrix},\]
and the manufactured potential 
\[ \psi_{\mathrm{ex}}(x,y) = x|x|^{65/128}(1-x^2)(1-y^2),\]
with the now space-dependent coefficients 
\[\varepsilon(x,y) = \exp(-xy), \quad \kappa(x,y) = \frac12 + \sin(xy)^2. \]
With these functions the load is such $g_{\mathrm{ex}} \in \rL^2(\Omega)$ but $g_{\mathrm{ex}}\notin \rH^s(\Omega)$ for $s \geq \frac{1}{128}$. {The least regular term of $g_{\mathrm{ex}}$ is $-\varepsilon\,\partial_{xx}\psi_{\mathrm{ex}}$, which behaves like $\mathrm{sgn}(x)|x|^{-63/128}$ times a smooth function that does not vanish on $\{x=0\}$, and $\mathrm{sgn}(x)|x|^{\gamma}\in\rH^s(-1,1)$ if and only if $s<\gamma+\frac12$ (as can be seen from its Fourier transform).} This smoothness implies that we can also use $g_{\mathrm{ex}}$ directly without $\cQ$. We report on the error history of the case with and without $\cQ$, which confirms an order of convergence similar to the aforementioned reference: $O(h)$ for the $\bL^2(\Omega)$ norm of the flux and for the $\rL^4(\Omega)$ norm of the potential, whereas for the postprocessed solution we obtain approximately $O(h^{5/3})$ and $O(h^2)$ for the methods without and with regularisation $\cQ$, respectively. The regularised method is then appropriate to use even if there is no need to regularise the load.

\begin{table}[width=\linewidth,cols=8,pos=t!]
\caption{Example 2. Error history for the lowest-order mixed finite element scheme subjected to an $\rL^2(\Omega)$ load, without (top) and with (bottom) the regularisation through $\cQ$. Here we use pure Dirichlet boundary conditions.}\label{tab:convergence-2}
\begin{tabular*}{\tblwidth}{@{\extracolsep{\fill}}rccccccc@{}}
\toprule
\texttt{DoFs} & $h$ & $e_{0}(\bzeta)$ & \texttt{EoC} & $e_{0,4}(\psi)$ & \texttt{EoC} & $e_{0}(\psi^\sharp)$ & \texttt{EoC} \\
\midrule
\multicolumn{8}{c}{with load $g_{\mathrm{ex}}$}\\
\midrule
24 & 1.4142 & $9.39{\cdot}10^{-1}$ & $\star$ & $1.59{\cdot}10^{-1}$ & $\star$ & $1.65{\cdot}10^{-1}$ & $\star$ \\
88 & 0.7071 & $7.45{\cdot}10^{-1}$ & 0.335 & $1.36{\cdot}10^{-1}$ & 0.227 & $1.14{\cdot}10^{-1}$ & 0.540 \\
336 & 0.3536 & $4.40{\cdot}10^{-1}$ & 0.758 & $7.69{\cdot}10^{-2}$ & 0.817 & $3.65{\cdot}10^{-2}$ & 1.641 \\
1312 & 0.1768 & $2.32{\cdot}10^{-1}$ & 0.926 & $4.04{\cdot}10^{-2}$ & 0.928 & $1.07{\cdot}10^{-2}$ & 1.774 \\
5184 & 0.0884 & $1.18{\cdot}10^{-1}$ & 0.971 & $2.05{\cdot}10^{-2}$ & 0.983 & $3.19{\cdot}10^{-3}$ & 1.743 \\
20608 & 0.0442 & $5.98{\cdot}10^{-2}$ & 0.983 & $1.03{\cdot}10^{-2}$ & 0.996 & $9.99{\cdot}10^{-4}$ & 1.674 \\
82176 & 0.0221 & $3.02{\cdot}10^{-2}$ & 0.986 & $5.13{\cdot}10^{-3}$ & 0.999 & $3.27{\cdot}10^{-4}$ & 1.610 \\
\midrule
\multicolumn{8}{c}{with load $\cQ g_{\mathrm{ex}}$}\\
\midrule
24 & 1.4142 & $1.02{\cdot}10^{0}$ & $\star$ & $2.49{\cdot}10^{-1}$ & $\star$ & $2.81{\cdot}10^{-1}$ & $\star$ \\
88 & 0.7071 & $7.13{\cdot}10^{-1}$ & 0.517 & $1.61{\cdot}10^{-1}$ & 0.629 & $1.56{\cdot}10^{-1}$ & 0.850 \\
336 & 0.3536 & $4.20{\cdot}10^{-1}$ & 0.764 & $8.25{\cdot}10^{-2}$ & 0.967 & $5.40{\cdot}10^{-2}$ & 1.531 \\
1312 & 0.1768 & $2.19{\cdot}10^{-1}$ & 0.941 & $4.11{\cdot}10^{-2}$ & 1.004 & $1.55{\cdot}10^{-2}$ & 1.800 \\
5184 & 0.0884 & $1.12{\cdot}10^{-1}$ & 0.969 & $2.06{\cdot}10^{-2}$ & 1.001 & $4.14{\cdot}10^{-3}$ & 1.904 \\
20608 & 0.0442 & $5.68{\cdot}10^{-2}$ & 0.976 & $1.03{\cdot}10^{-2}$ & 1.001 & $1.08{\cdot}10^{-3}$ & 1.945 \\
82176 & 0.0221 & $2.88{\cdot}10^{-2}$ & 0.980 & $5.13{\cdot}10^{-3}$ & 1.000 & $2.77{\cdot}10^{-4}$ & 1.955 \\
\bottomrule
\end{tabular*}
\end{table}

\begin{figure}[pos=t!]
\centering
\begin{overpic}[width=0.32\linewidth]{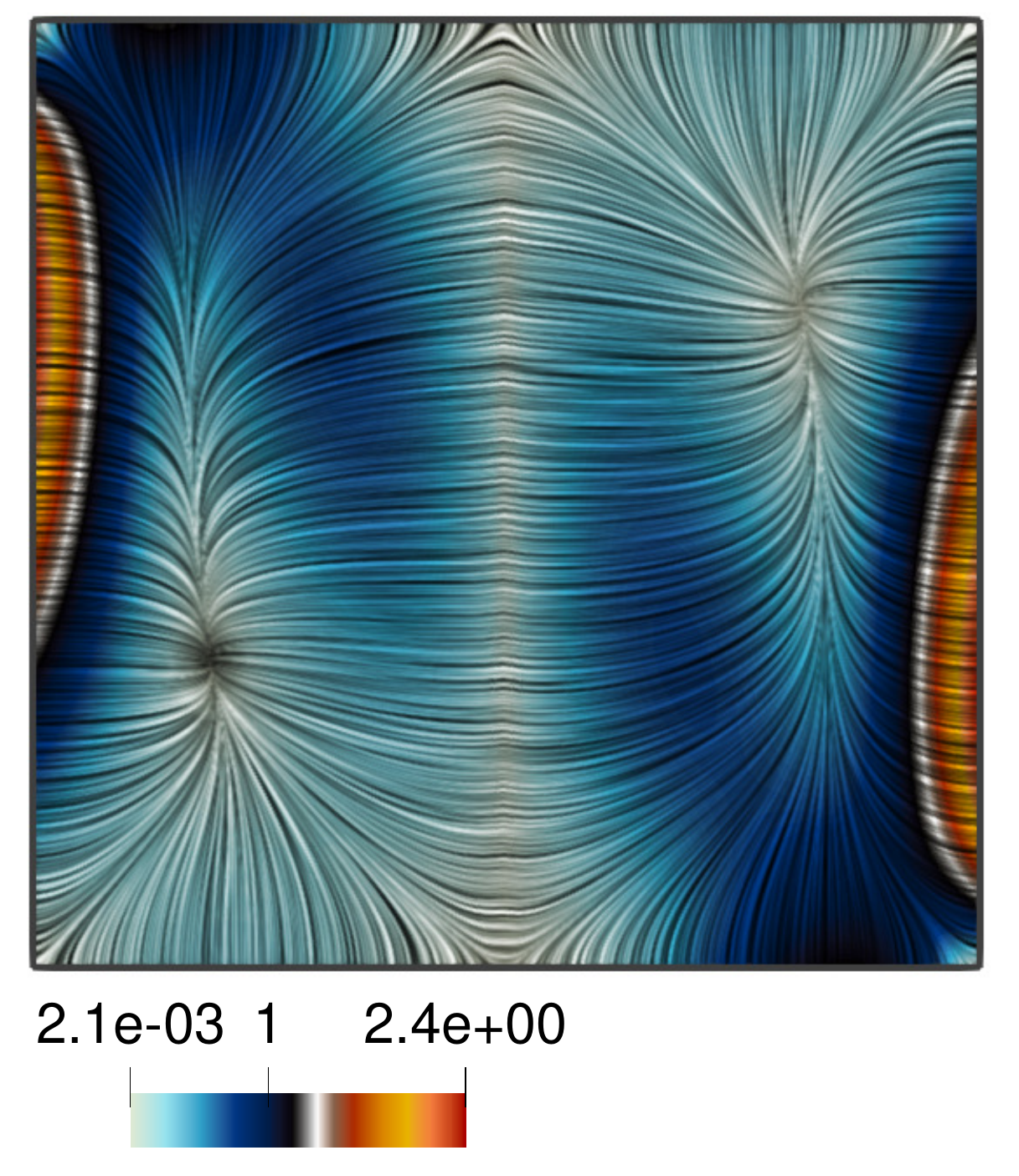}
  \put(54,3){\normalsize $|\bzeta_h|$}
\end{overpic}\hfill
\begin{overpic}[width=0.32\linewidth]{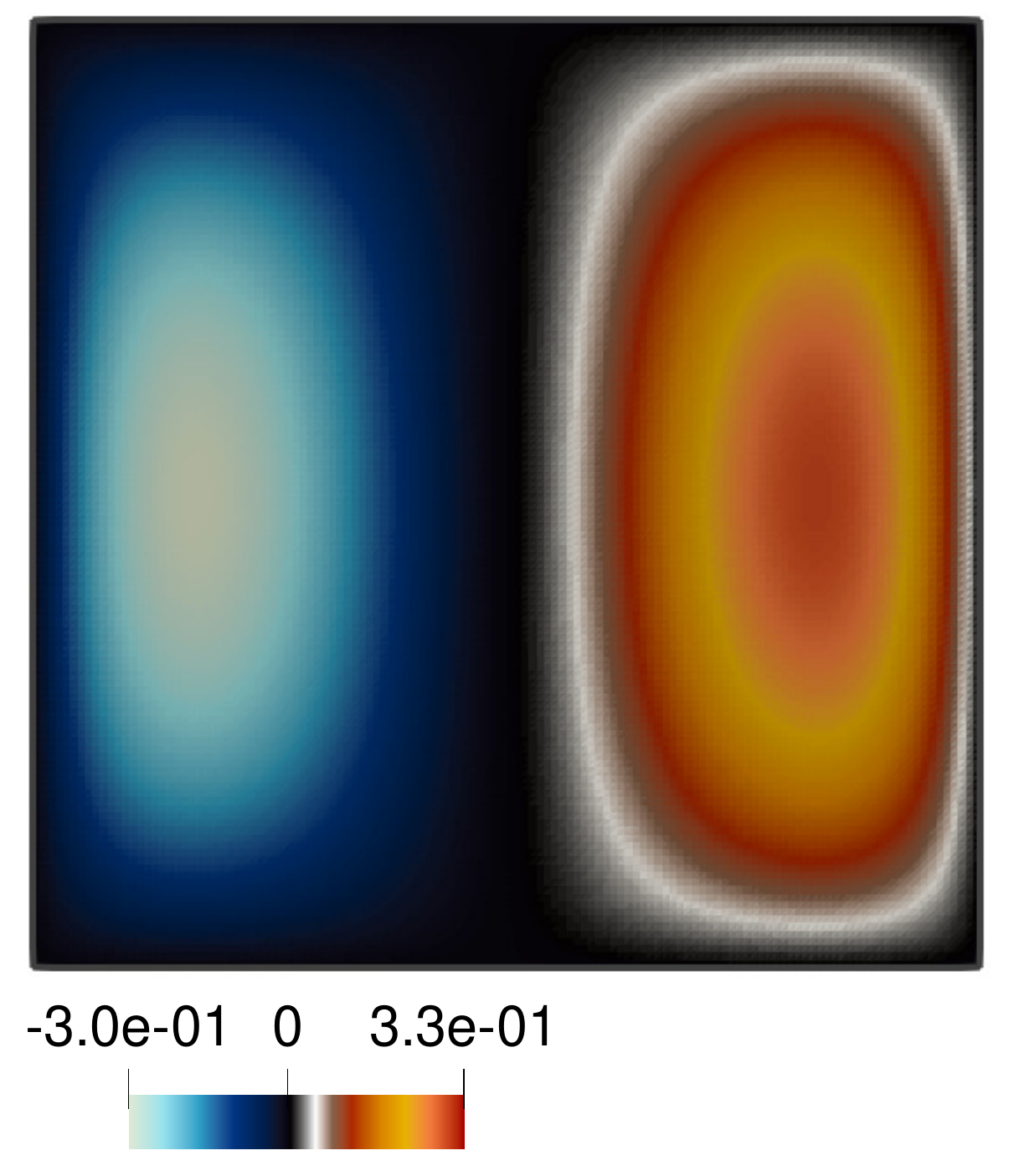}
  \put(54,3){\normalsize $\psi_h$}
\end{overpic}\hfill
\begin{overpic}[width=0.32\linewidth]{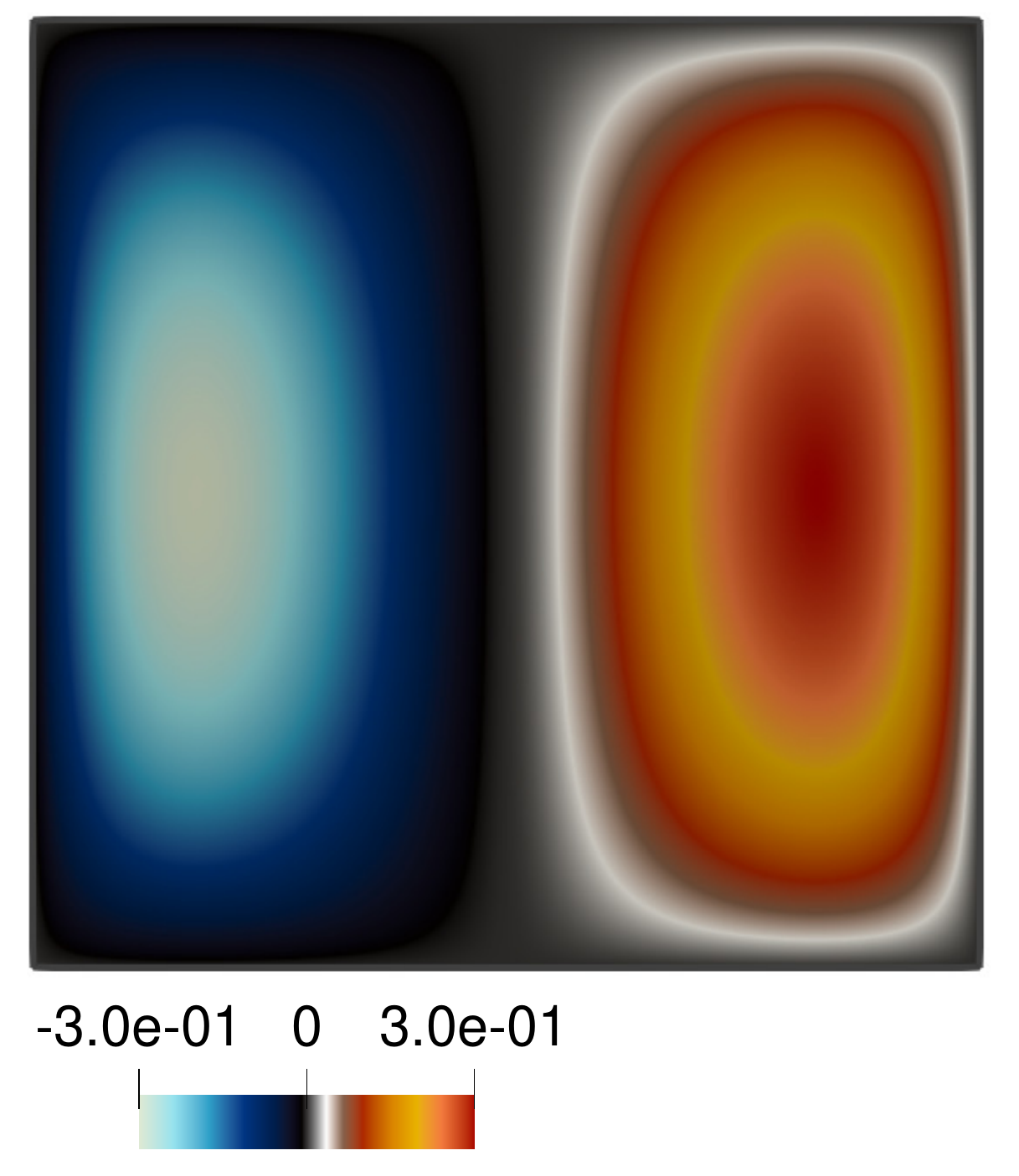}
  \put(54,3){\normalsize $\psi_h^\sharp$}
\end{overpic}

\vspace{4ex}
\caption{Example 2. Approximate solutions for the lowest-order mixed finite element scheme and postprocessed potential using the regularisation $\cQ$.}
\label{fig:ex02}
\end{figure}

\medskip 
\noindent\textbf{Example 3.} 
Finally, we address the advection{-}diffusion{-reaction} problem with a line Dirac delta function as the source term, which can be relevant not only in the Poisson--Boltzmann context but also in mono-phasic flow in porous media and tissue perfusion \cite{li2021regularity,cao2025regularity}. We consider the polygonal domain enclosed by the points $(0,0)$, $(1,0)$, $(1.25,0.75)$, $(0.5,1.25)$, $(-0.25,0.75)$ and consider that there is a line fracture $\gamma$ joining the points $(0.4,0.25)$ and $(0.6,0.85)$. We construct a sequence of non-structured triangulations that are conforming with the line fracture, and that are slightly {refined} near the line. We use a smooth advecting velocity going from left to right with higher magnitude near the domain centre:  
\[ \eta(x,y) = U_0 \sigma \tanh\bigl(\frac{y-\frac12}{\sigma}\bigr) (1+ \theta \sin (2\pi x)), \quad \text{and}\quad \bu(x,y) = \mathrm{rot}\,\eta(x,y),\]
and set the parameters $U_0 {\in}\{ 0.25,0.0025\}$, $\sigma = 0.08$, $\theta = 0.12$ and the remaining model coefficients are $\varepsilon = 10^{-3}$ and $\kappa =1$. We use pure Dirichlet boundary conditions setting $\psi = 0$ on $\partial\Omega$. The source term is $\delta_{\gamma}$ and in this case {$\delta_\gamma\in\rH^{-\frac12-\epsilon}(\Omega)$ for all $\epsilon >0$. We note that for $U_0=0.25$ the smallness assumptions~\eqref{eq:small-u}--\eqref{eq:small-u-h} are not satisfied, so this case goes beyond our theory.} 

{Since the meshes conform with $\gamma$ and the element bubbles vanish on $\gamma$, \eqref{eq:defQ} gives $\cQ\delta_\gamma=\sum_{z}\bigl(\int_\gamma\eta_z\bigr)\psi_z$.} Since we do not know the closed-form solution, we use a numerical convergence rate computing errors between approximate solutions and a reference numerical solution (obtained on a finer mesh){, and evaluate the errors on the reference mesh}. The convergence history for this test is reported in Table~\ref{table:line}. {Since $\vdiv\bzeta\notin\rL^{4/3}(\Omega)$, we only report the $\bL^2$ part of the flux error. The solution has a layer of width $\mathcal{O}(\varepsilon^{1/2})$ around $\gamma$, so the rates are still pre-asymptotic; the postprocessed potential converges with rate close to $1.7$.}  We  show snapshots of  the approximate solutions for two advecting intensities $U_0$, as well as an example of the unstructured mesh. The top panels in Figure~\ref{fig:ex03} illustrate how the higher velocity is able to advect the higher potential produced by the source $\delta_{\gamma}$ whereas the second row shows flux and potential that remain concentrated near $\gamma$. 

\begin{table}[width=\linewidth,cols=8,pos=t!]
\caption{Example 3. Error history for the lowest-order mixed finite element scheme subjected to a line source, including the regularisation through $\cQ$, for $U_0=0.25$. Errors are computed on a finer reference mesh, with respect to the corresponding reference solution. Pure Dirichlet boundary conditions are used.}\label{table:line}
\begin{tabular*}{\tblwidth}{@{\extracolsep{\fill}}rccccccc@{}}
\toprule
\texttt{DoFs} & $h$ & $e_{0}(\bzeta)$ & \texttt{EoC} & $e_{0,4}(\psi)$ & \texttt{EoC} & $e_{0}(\psi^\sharp)$ & \texttt{EoC} \\
\midrule
64 & 0.7906 & $1.17{\cdot}10^{-1}$ & $\star$ & $2.98{\cdot}10^{0}$ & $\star$ & $2.43{\cdot}10^{0}$ & $\star$ \\
179 & 0.3953 & $9.31{\cdot}10^{-2}$ & 0.327 & $2.69{\cdot}10^{0}$ & 0.148 & $1.64{\cdot}10^{0}$ & 0.569 \\
649 & 0.2146 & $6.48{\cdot}10^{-2}$ & 0.595 & $1.84{\cdot}10^{0}$ & 0.624 & $6.45{\cdot}10^{-1}$ & 1.523 \\
2778 & 0.1166 & $4.28{\cdot}10^{-2}$ & 0.678 & $9.84{\cdot}10^{-1}$ & 1.026 & $2.40{\cdot}10^{-1}$ & 1.622 \\
10515 & 0.0622 & $2.62{\cdot}10^{-2}$ & 0.781 & $4.94{\cdot}10^{-1}$ & 1.095 & $8.43{\cdot}10^{-2}$ & 1.662 \\
40304 & 0.0325 & $1.47{\cdot}10^{-2}$ & 0.896 & $2.53{\cdot}10^{-1}$ & 1.033 & $2.74{\cdot}10^{-2}$ & 1.732 \\
\bottomrule
\end{tabular*}
\end{table}

\begin{figure}[pos=t!]
\centering
\begin{overpic}[width=0.32\linewidth]{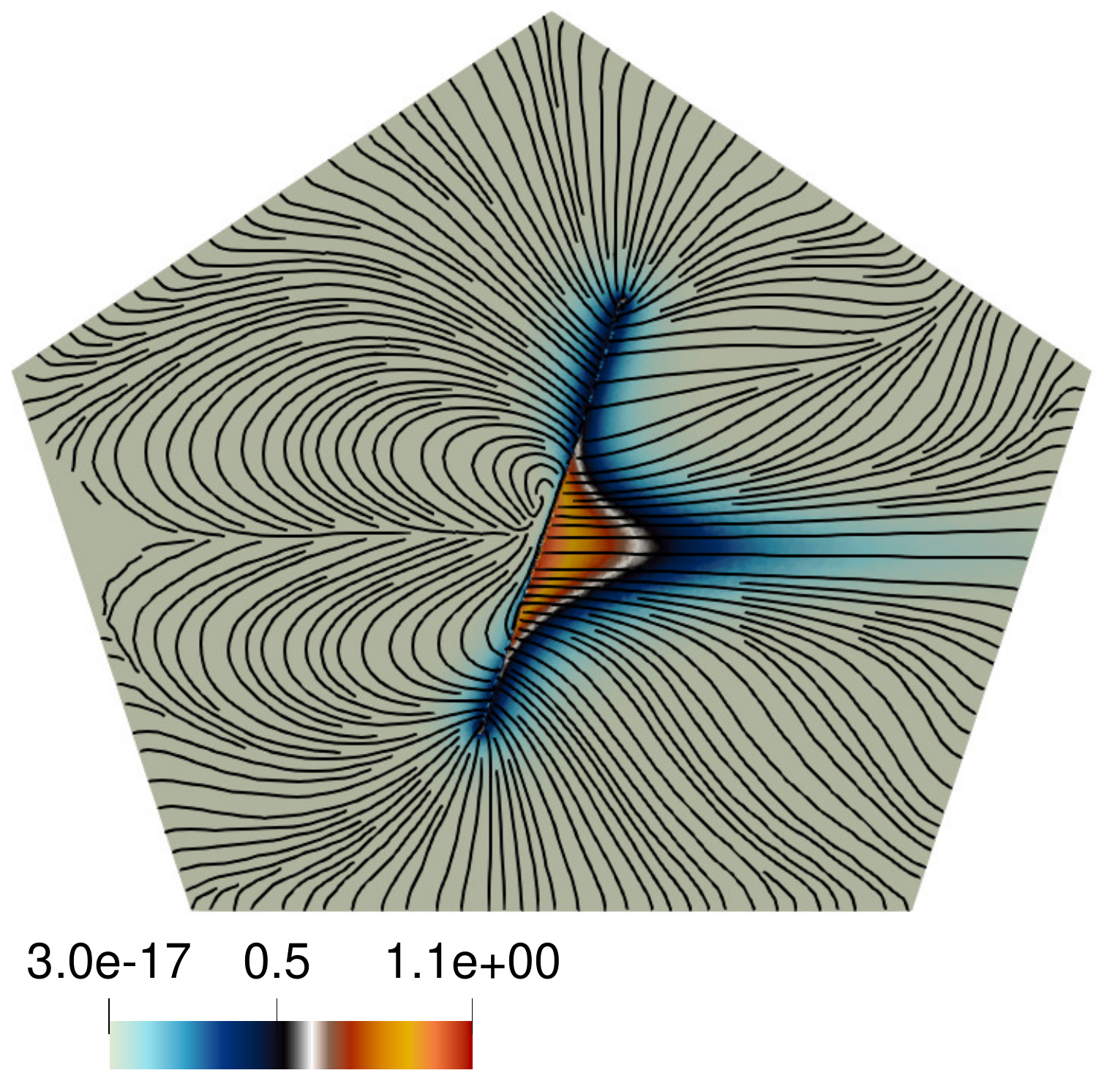}
  \put(54,3){\normalsize $|\bzeta_h|$}
\end{overpic}\hfill
\begin{overpic}[width=0.32\linewidth]{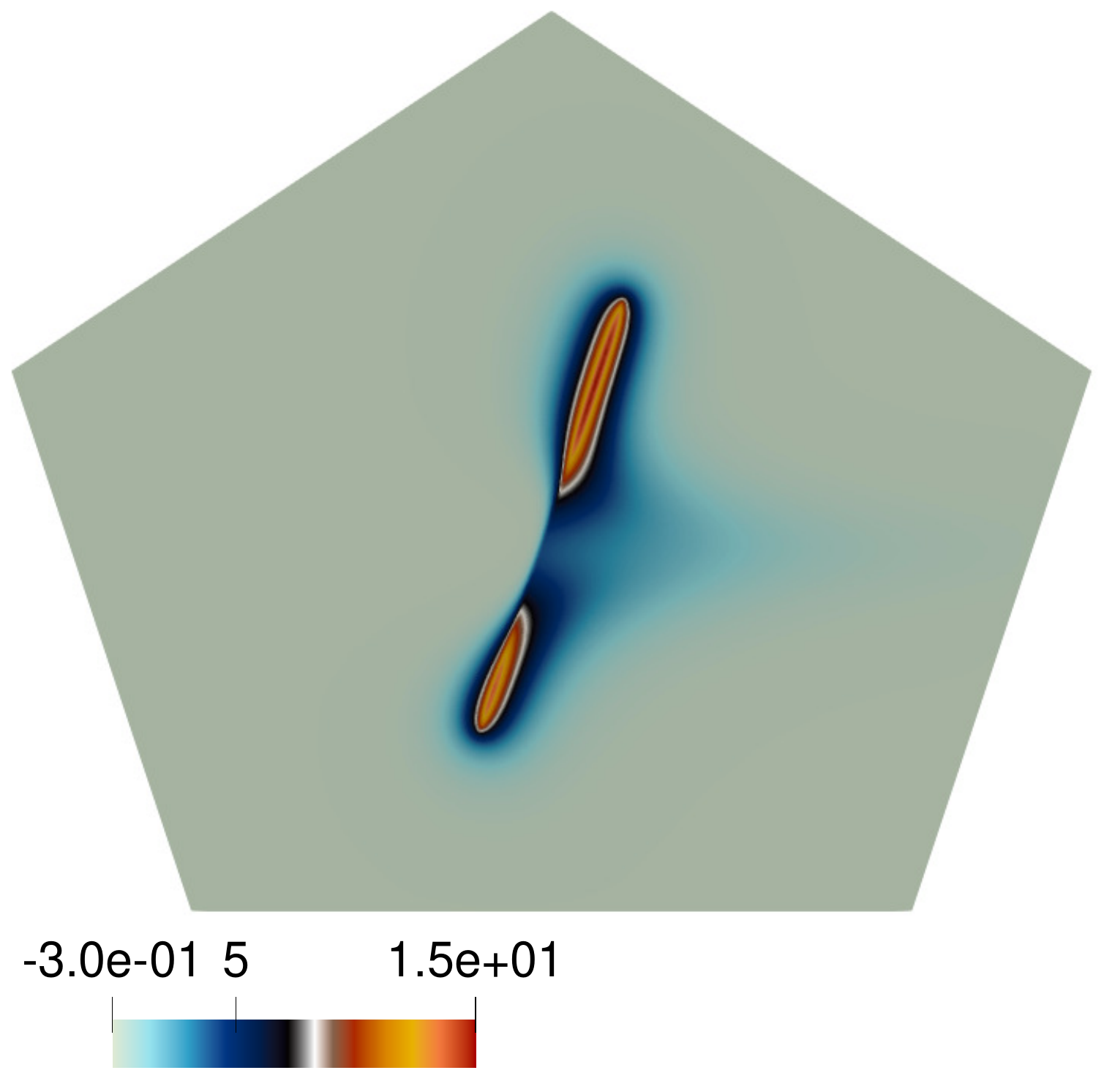}
  \put(54,3){\normalsize $\psi_h$}
\end{overpic}\hfill
\begin{overpic}[width=0.32\linewidth]{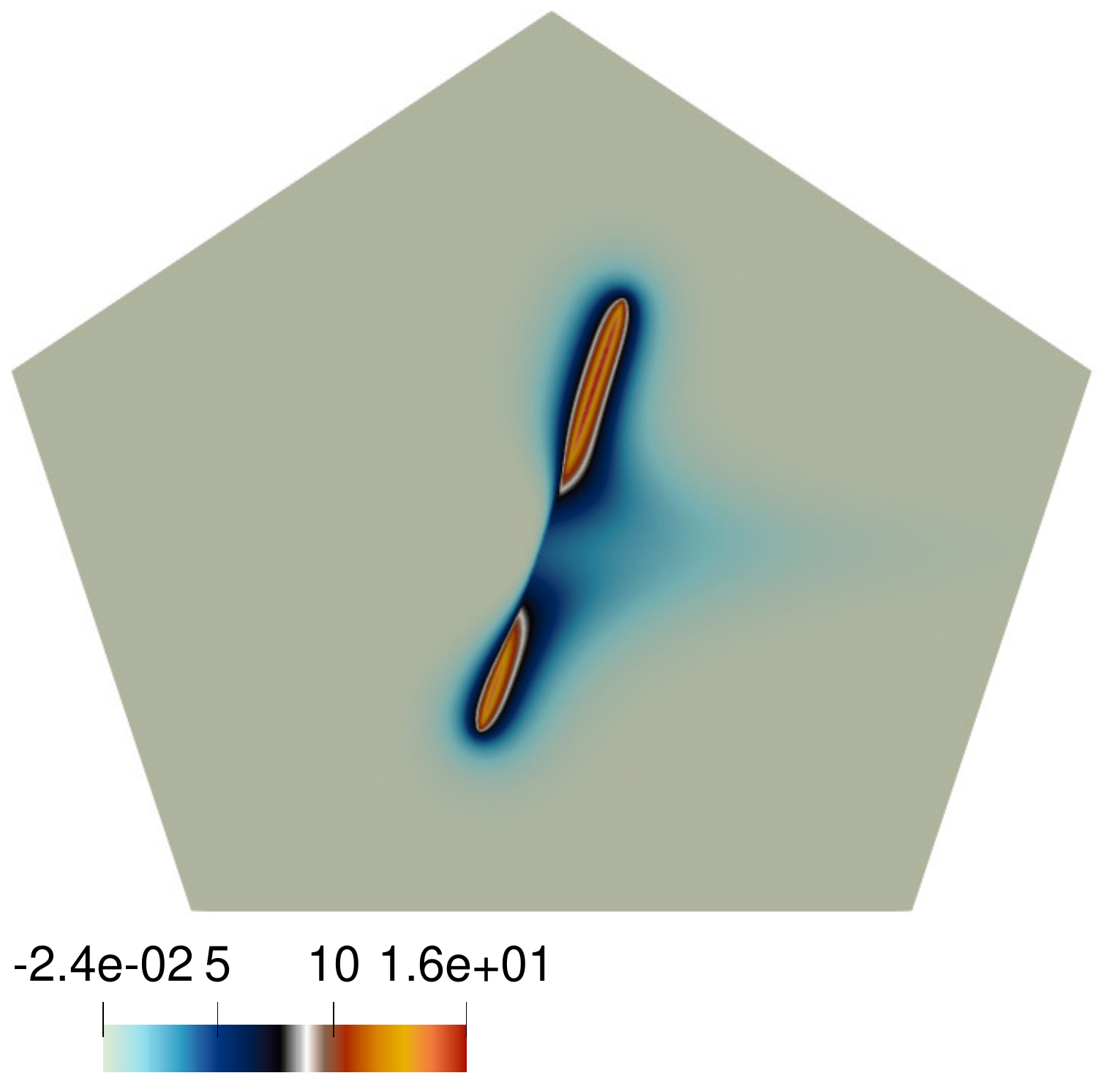}
  \put(54,3){\normalsize $\psi_h^\sharp$}
\end{overpic}\\[3ex]
\begin{overpic}[width=0.32\linewidth]{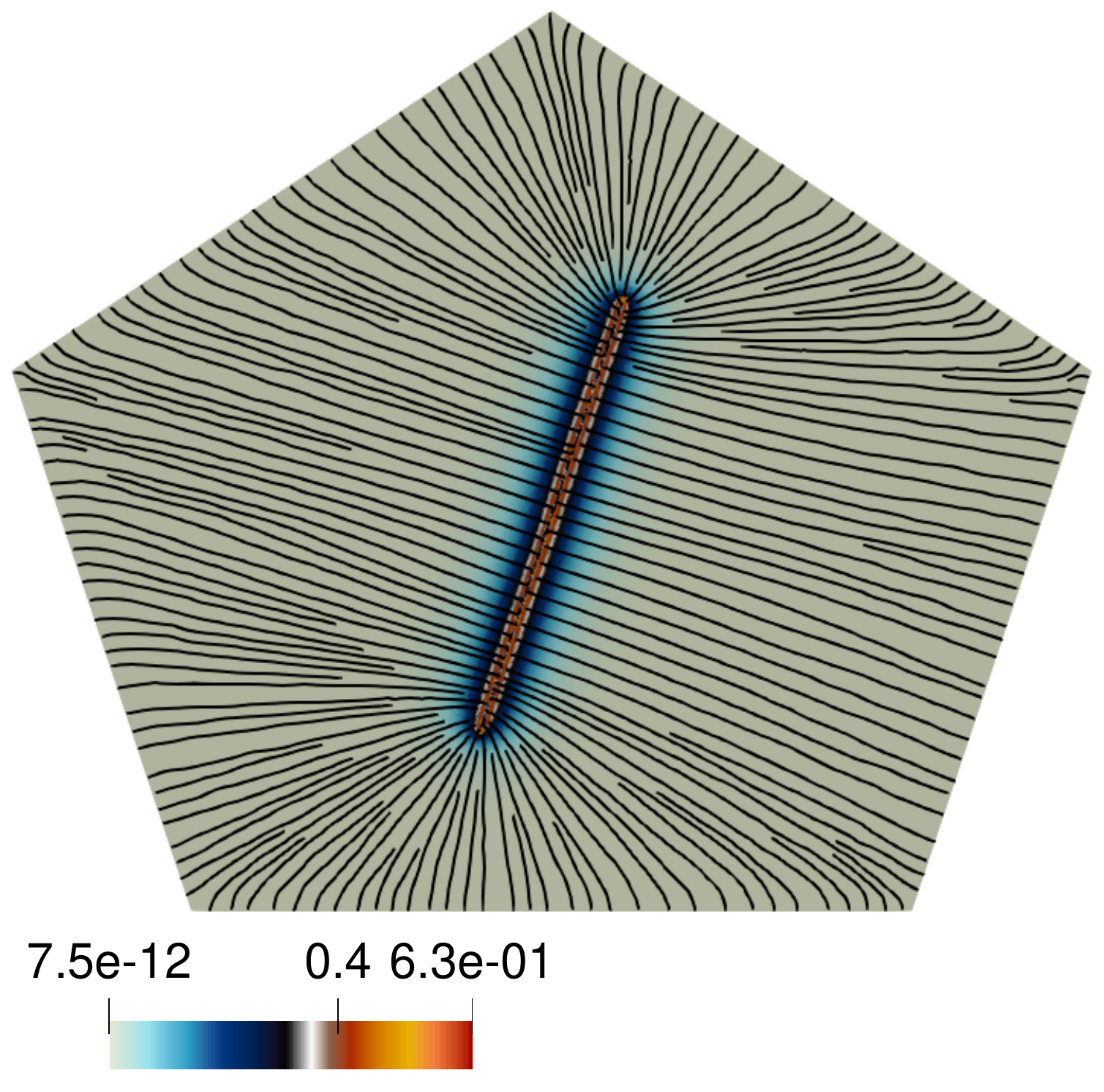}
  \put(54,3){\normalsize $|\bzeta_h|$}
\end{overpic}\hfill
\begin{overpic}[width=0.32\linewidth]{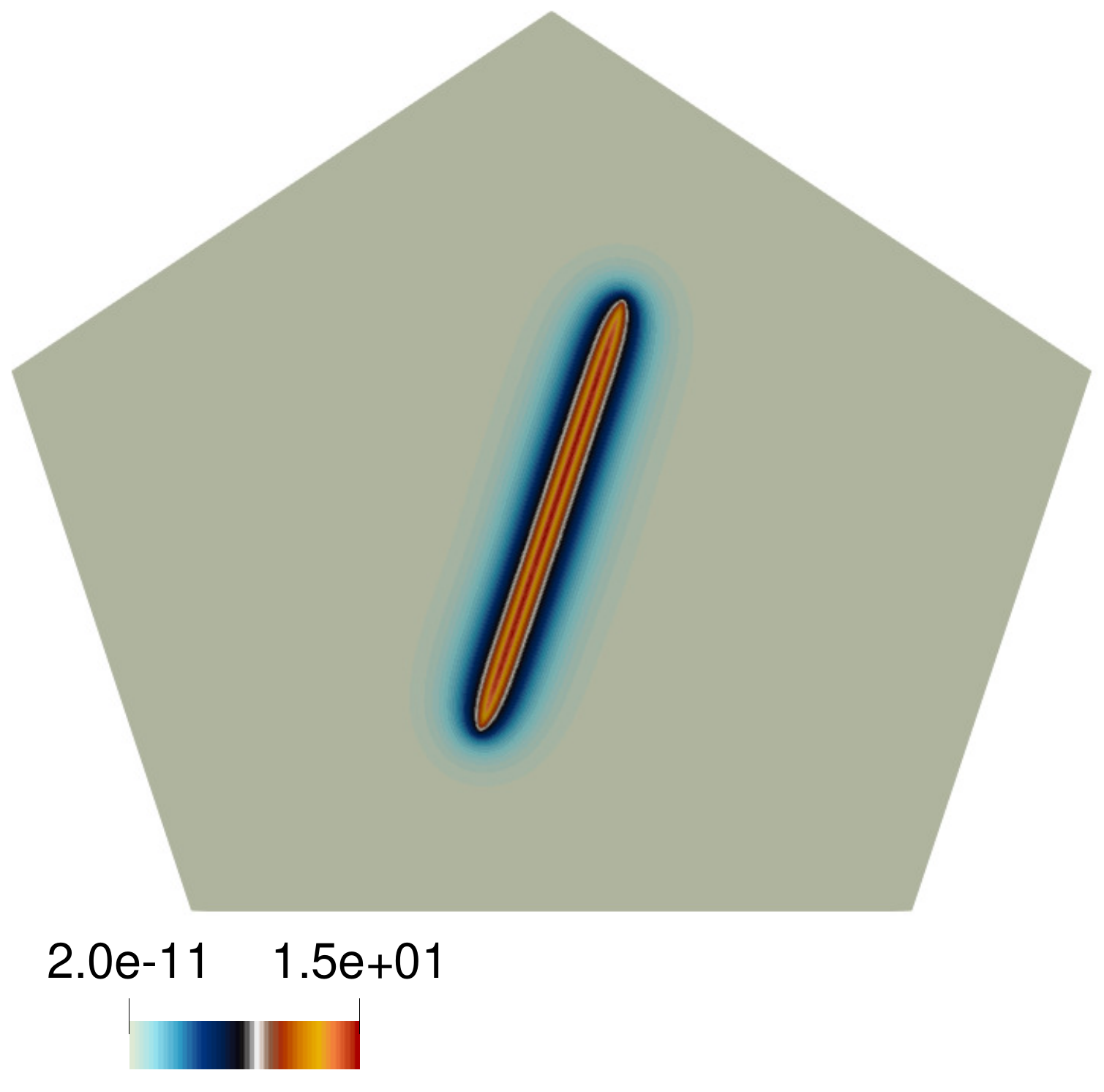}
  \put(54,3){\normalsize $\psi_h$}
\end{overpic}\hfill
\begin{overpic}[width=0.32\linewidth]{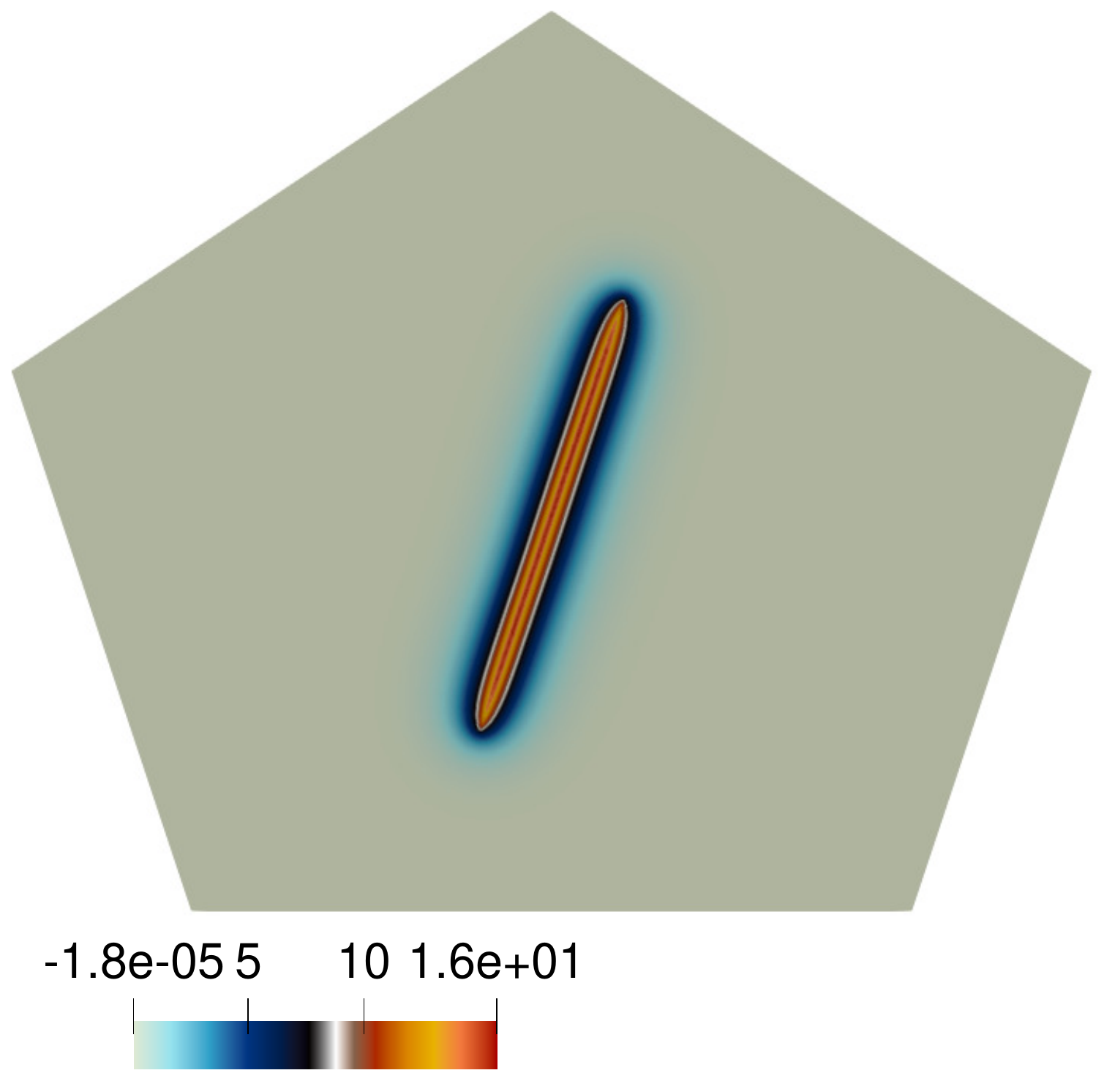}
  \put(54,3){\normalsize $\psi_h^\sharp$}
\end{overpic}\\[3ex]
\begin{overpic}[width=0.32\linewidth]{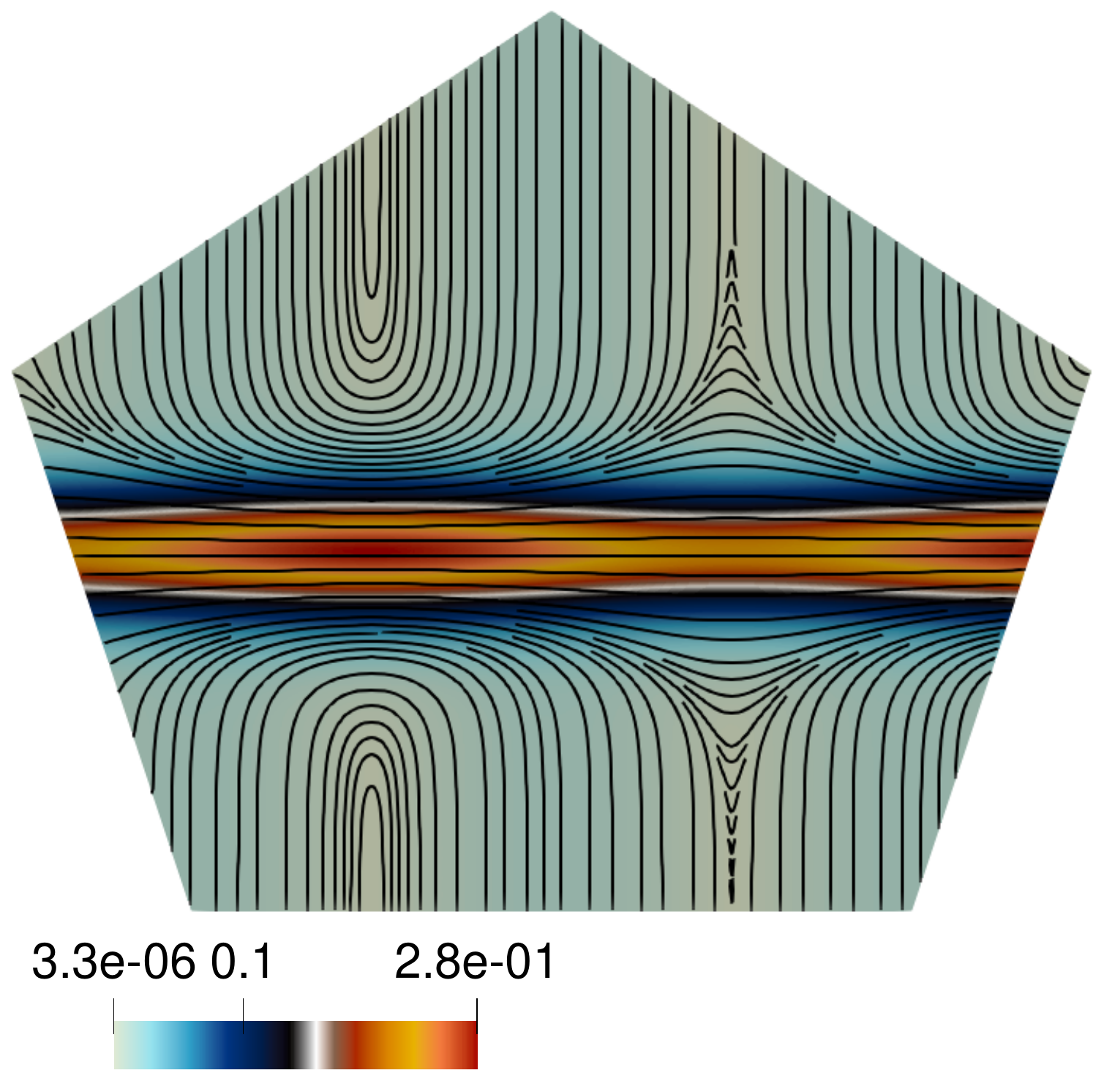}
  \put(54,3){\normalsize $|\bu_h|$}
\end{overpic}\hfill
\begin{overpic}[width=0.32\linewidth]{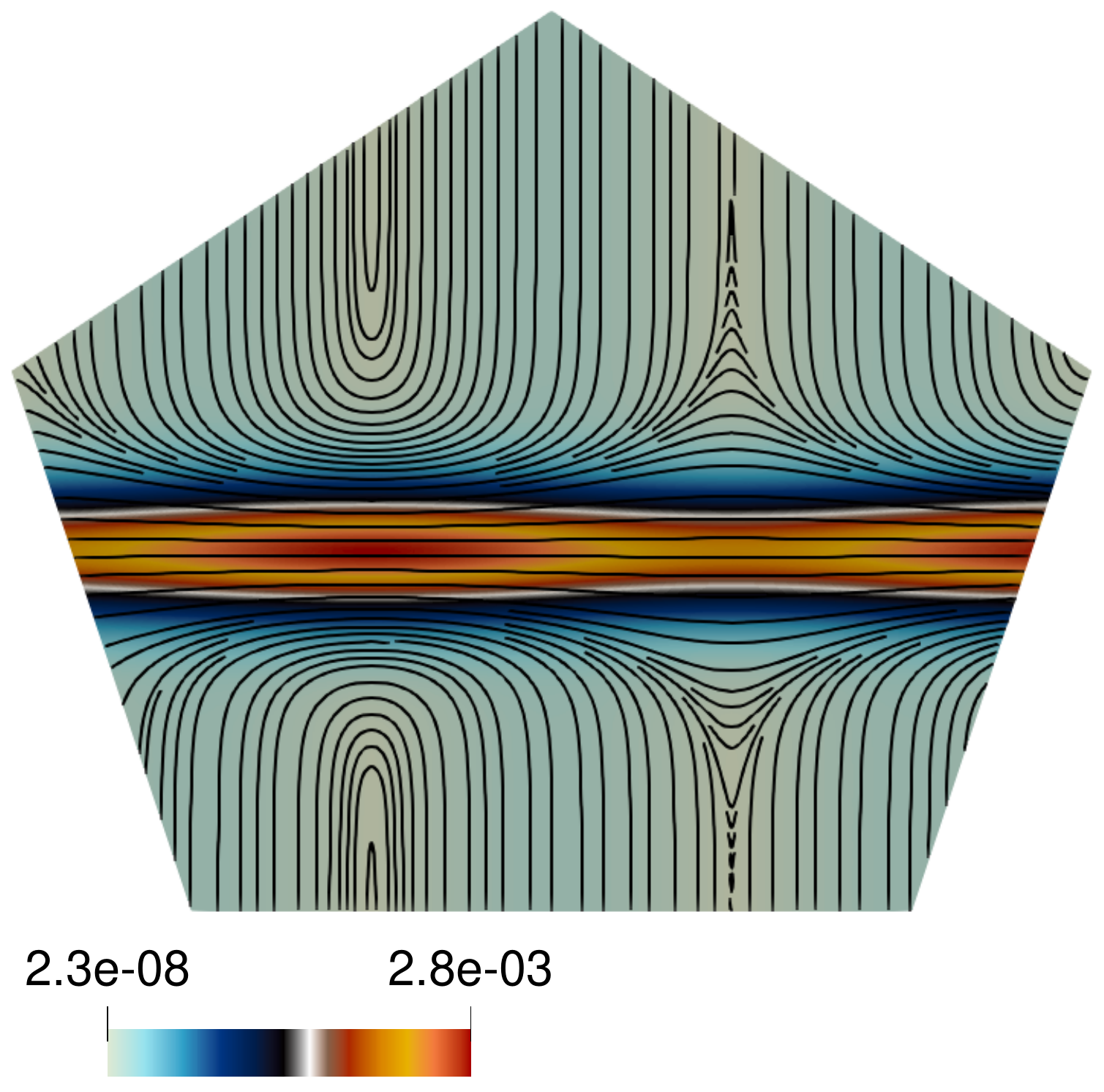}
  \put(54,3){\normalsize $|\bu_h|$}
\end{overpic}\hfill
\makebox[0.32\linewidth]{\raisebox{0.045\linewidth}{\includegraphics[width=0.32\linewidth]{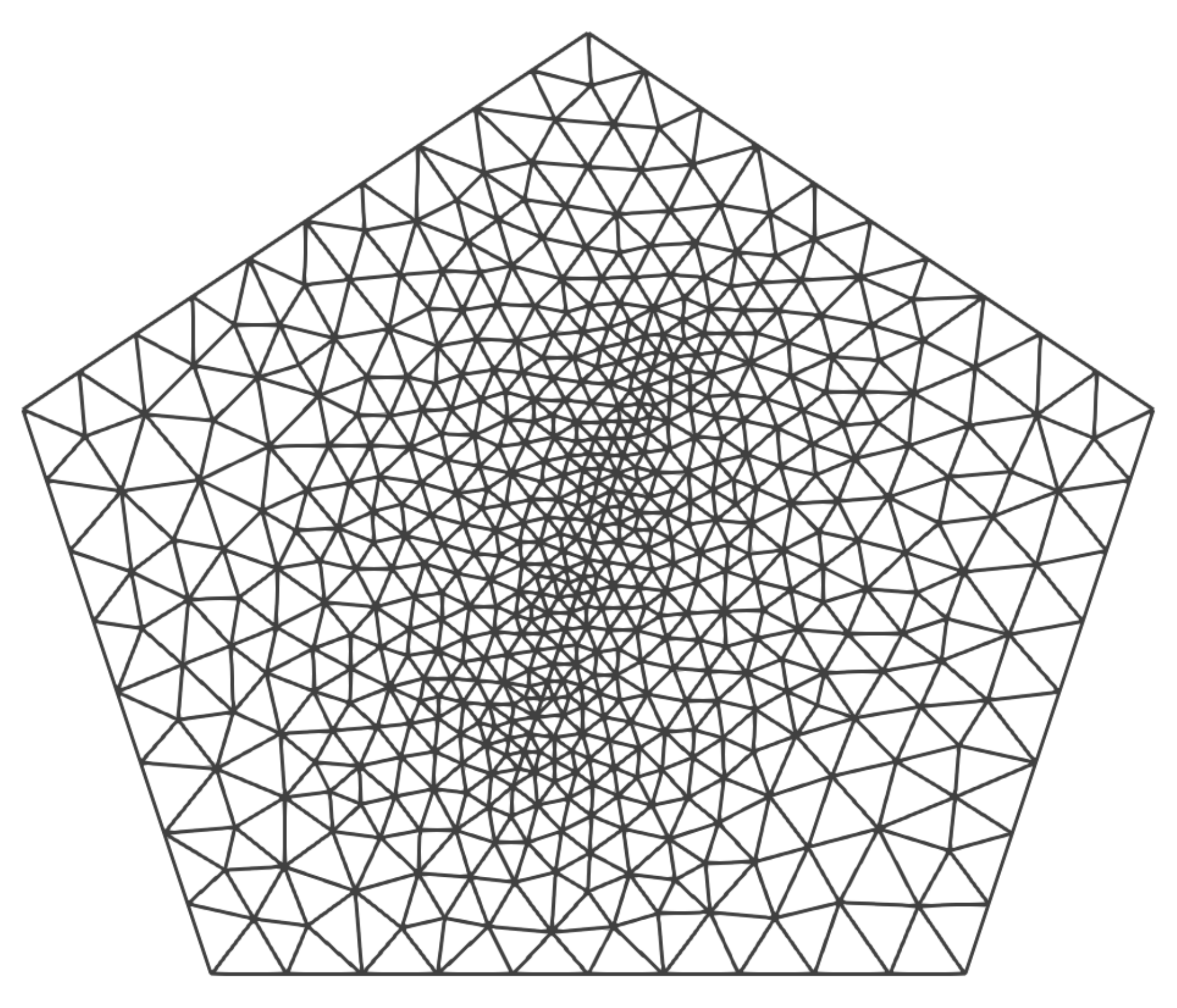}}}

\vspace{4ex}
\caption{Example 3. Approximation of pseudo potential flux and streamlines, double layer potential and postprocessed potential for the case of $U_0 = 0.25$ (top row) and $U_0 = 0.0025$ (middle row). The bottom panels show the approximate velocity field and streamlines for the high and low intensity, and a sample coarse mesh after 3 refinements.}
\label{fig:ex03}
\end{figure}

%%%%%%%%%%%%%%%%%%%%%%%%%%%%%%%%%%%%%
\section{Conclusion}\label{sec:concl}
We have developed a mixed finite element method for the linearised Poisson--Boltzmann equations assuming an advective velocity in $\bL^4(\Omega)$. One of the challenges is to derive error estimates when the load is in $\rH^{-1}(\Omega)$, and this is tackled using a quasi-interpolation operator, here conveniently redefined for $\rL^p(\Omega)$ spaces{,} mixed boundary conditions{, and arbitrary polynomial degree}. The continuous and discrete problems were shown to be well-posed and we established quasi-optimality and convergence rates. A postprocess of the double layer potential allows us to get superconvergence. 

For us, a natural extension to this work is to perform the full coupling analysis with the Navier--Stokes equations (see, for example, \cite{caucao2020new,correa24,Colmenares20,Camano21}), and to incorporate the unbounded nonlinearity in the reaction function $\kappa$ as done for the primal coupling in \cite{alsohaim25}. This requires to add box constraints in the double layer potential space, and to use a fixed-point structure to handle a perturbed saddle-point problem when the nonlinearity is in the lower-diagonal block.  We will also address a posteriori error estimates as in \cite{cao2025regularity}, and  other types of singular forcing terms including Dirac point sources {(which do not belong to $\rH^{-1}(\Omega)$ for $d\geq2$)} following \cite{koppl2014optimal,drelichman2020weighted}.

%%%%%%%%%%%%%%%%%%%%%%%
\smallskip 
\subsection*{Acknowledgements} This research has been supported by   ANID through the FONDECYT project 1250070 (TF) and \textit{Concurso de Subvenci\'on a la Instalaci\'on en la Academia, convocatoria 2025}, Project number 85250057 (SVF); by the Australian Research Council through the \textit{Future Fellowship} grant FT220100496 (RRB); by the Swedish Research Council under grant no. 2021-06594, Institut Mittag--Leffler in Djursholm, Sweden (RRB); and by the Center of Advanced Study (CAS) at the Norwegian Academy of Science and Letters under the program \textit{Mathematical Challenges in Brain Mechanics} (RRB). 
%%%%%%%%%%%%%%%%%%%%%%

\bibliographystyle{elsarticle-num}
\bibliography{refs}
\end{document}